\documentclass[aop]{imsart}

\RequirePackage{amsthm,amsmath,amsfonts,amssymb}
\RequirePackage[numbers]{natbib}

\usepackage[latin1]{inputenc}  
\usepackage[T1]{fontenc}       

\usepackage{latexsym}
\usepackage[dvips]{epsfig}
\usepackage{dcolumn}
\usepackage{tabularx}
\usepackage{longtable}
\usepackage{graphics}
\usepackage{pstricks}
\usepackage{pst-node}
\usepackage{latexsym}
\usepackage[english]{babel}
\usepackage{euscript}
\usepackage{dsfont}
\usepackage{bbm}

\startlocaldefs

\renewcommand{\theequation}{\thesection.\arabic{equation}}
\newtheorem{theorem}{Theorem}[section]
\newtheorem{lemma}[theorem]{Lemma}
\newtheorem{proposition}[theorem]{Proposition}
\newtheorem{corollary}[theorem]{Corollary}
\newtheorem{remark}[theorem]{Remark}
\newtheorem{definition}[theorem]{Definition}

\newtheorem{definition-proposition}[theorem]{Definition-Proposition}

\newcommand{\eqnsection}{
\renewcommand{\theequation}{\thesection.\arabic{equation}}
    \makeatletter
    \csname  @addtoreset\endcsname{equation}{section}
    \makeatother}
\eqnsection

\def\ali{\hfill\break}
\def\w{\omega}

\def\sss{{\mathcal S}}
\def\demi{{1\over 2}}
\def\ddd{{\cal D}}

\def\ccc{{\cal C}}

\def\Z{{{\Bbb Z}}}
\def\P{{{\Bbb P}}}
\def\N{{{\Bbb N}}}
\def\R{{{\Bbb R}}}

\def\ppp{{{\mathcal P}}}
\def\aaa{{{\mathcal A}}}
\def\lll{{{\mathcal U}}}

\def\im{{{\hbox{Im}}}}
\def\fff{{{\mathcal F}}}

\def\ttt{{\mathcal T}}
\def\ccc{{\mathcal C}}

\def\E{{{\Bbb E}}}
\def\dive{{\text{div}}}

\def\hhh{{{\mathcal H}}}

\def\indic{{{\mathbbm 1}}}
\def\Id{{\hbox{Id}}}
\def\bbb{{{\mathcal B}}}

\def\Ktrans{{}^t\! K}
\def\otau{{\overline \tau}}

\renewcommand{\le}{\leqslant}

\renewcommand{\ge}{\geqslant}

\renewcommand{\subset}{\subseteq}

\newcommand{\bal}{\begin{align*}}
\newcommand{\eal}{\end{align*}}
\newcommand{\beq}{\begin{eqnarray*}}
\newcommand{\eeq}{\end{eqnarray*}}
\newcommand{\bte}{\begin{theorem}}
\newcommand{\ete}{\end{theorem}}
\newcommand{\bl}{\begin{lemma}}
\newcommand{\el}{\end{lemma}}
\newcommand{\bd}{\begin{description}}
\newcommand{\ed}{\end{description}}

\newcommand{\bc}{\begin{cases}}
\newcommand{\ec}{\end{cases}}
\newcommand{\bp}{\begin{proof}}
\newcommand{\ep}{\end{proof}}
\newcommand{\bco}{\begin{corol}}
\newcommand{\eco}{\end{corol}}

\newcommand{\iy}{\infty}
\newcommand{\tx}{\text}

\newcommand{\D}{\mathcal{D}}

\newcommand{\al}{\alpha}

\newcommand{\be}{\beta}

\newcommand{\tet}{\theta}

\newcommand{\s}{\sigma}

\def\bdes{\begin{description}}
\def\edes{\end{description}}

\def\sERRW{{$\star$-ERRW }}
\def\sVRJP{{$\star$-VRJP }}
\def\sERRWse{{$\star$-ERRW}}
\def\sVRJPse{{$\star$-VRJP}}
\def\Gen{{\mathcal {L}}}
\def\uuu{{\mathcal{U}}}
\def\ddd{{\mathcal D}}
\def\ggg{{\mathcal G}}
\def\H{{\mathbb H}}
\def\qqq{{\mathcal Q}}

\endlocaldefs

\begin{document}

\begin{frontmatter}
\title{The $\star$-Vertex-Reinforced Jump Process }
\runtitle{The  $\star$-VRJP I}

\begin{aug}

\author[A]{\fnms{Christophe}~\snm{Sabot}\ead[label=e1]{sabot@math.univ-lyon1.fr}},
\and
\author[B]{\fnms{Pierre}~\snm{Tarr\`es}\ead[label=e2]{tarres@nyu.edu}}
\address[A]{Universit\'e Claude Bernard Lyon 1,
Institut Camille Jordan, CNRS UMR 5208, 43, Boulevard du 11 novembre 1918,
69622 Villeurbanne Cedex, France, and Institut Universitaire de France\printead[presep={,\ }]{e1}}
\address[B]{NYU-ECNU Institute of Mathematical Sciences at NYU Shanghai, China; Courant Institute of Mathematical Sciences, New York, USA; CNRS and Universit\'e Paris-Sorbonne, Laboratoire de Probabilit\'es, Statistique et Mod\'elisation, 75005 Paris, France.\printead[presep={,\ }]{e2}}
\end{aug}

\begin{abstract}
We investigate the non-reversible generalization of the Vertex-Reinforced Jump Process (VRJP), 
called  the $\star$-Vertex-Reinforced Jump Process ($\star$-VRJP) and introduced in \cite{BST20}. It can be seen as the continuous-time counterpart to the $\star$-Edge-Reinforced Random Walk (\sERRWse) (see \cite{Bacallado1,BST20}), which is itself a non-reversible, and in fact Yaglom reversible, generalization of the original ERRW  introduced by Coppersmith and Diaconis in 1986 \cite{coppersmith}. In contrast to the classical VRJP, the \sVRJP is not exchangeable after time-change, which leads to several difficulties and new phenomena. 

Firstly, we show that with some appropriate randomization of the initial local time, it becomes partially exchangeable after time-change. We provide a representation of the "randomized" \sVRJP as a mixture of Yaglom reversible Markov jump processes with an explicit mixing measure, and we prove that the non-randomized \sVRJP can be written as a mixture of conditioned Markov processes.

Secondly, we give a representation of the \sVRJP in terms of a random Schr\"odinger operator. The corresponding representation for the classical VRJP has proved to be very useful in the understanding of its asymptotic behavior. 
The construction is based on several new and rather remarkable identities between integrals on the space of $\star$-symmetric and $\star$-antisymmetric functions on vertices. We give a description of the randomized \sVRJP in terms of that random Schr\"odinger operator, which allows us to prove the representation of the randomized \sVRJP as a mixture of Markov jump processes in a different and more analytic manner. Similarly as for the VRJP, we think that the representation by a random Schr\"odinger operator and the associated identities are key-features of the \sVRJPse.
\end{abstract}



\end{frontmatter}

\setcounter{tocdepth}{1}

The purpose of this paper is to investigate a natural non-reversible generalization of the Vertex Reinforced Jump Process (VRJP). The VRJP is a time-continuous self-interacting jump process defined on an undirected conductance network. This process was proposed by Werner and initially investigated by Davis and Volkov (\cite{dv1,dv2}, see also \cite{bs}). In \cite{ST15}, it was shown to be intimately related on the one hand to the Edge Reinforced Random Walk (ERRW), and on the other hand to a supersymmetric hyperbolic sigma-field  introduced by Zirnbauer \cite{zirnbauer91} and investigated by Disertori, Spencer and Zirnbauer \cite{DSZ10,DS10}. These relations, together with the crucial estimates of \cite{DSZ10,DS10} and several other works \cite{ACK14,DST15,STZ17,SZ19,Poudevigne19}, allowed to clarify the picture about the recurrent/transient phases of the VRJP and the ERRW. 

The aim of this paper is to start a similar program for a family of processes interpolating between the non-directed case (ERRW and  VRJP) and the fully directed case (directed ERRW, or equivalently the random walk in random Dirichlet environment, see below).

The processes considered here are defined on a directed graph $\ggg=(V,E)$ endowed with an involution on the vertices: the involution is denoted by $\star:V\mapsto V$, and the graph satisfies the property that 
\begin{equation}
(i,j)\in E \iff (j^\star,i^\star)\in E.
\end{equation}
Hence, the involution $\star$ also acts on the edges and all the weights defined on the edges will be $\star$-symmetric. We call $\ggg$ a $\star$-directed graph.  

The natural generalization of the ERRW on $\star$-directed graphs was introduced in \cite{BST20} and generalizes the $k$-dependent ERRW defined by Baccalado in \cite{Bacallado1}. Starting from some positive $\star$-symmetric weights $(\alpha_e)_{e\in E}$ on the edges, the \sERRW behaves as follows: each time an edge $(i,j)$ is crossed, the weights of the edge $(i,j)$ and of the $\star$-symmetric edge $(j^\star,i^\star)$ are increased by one, and the process jumps with a probability proportional to the current weights. With that definition, the weights remain $\star$-symmetric at all times (see precise definition below). 

Under a divergence condition on the weights, we proved in \cite{BST20} that the \sERRW is partially exchangeable and we compute its mixing measure using a new discrete Feynman-Kac type formula. This yields the counterpart on $\star$-directed graphs of the Diaconis-Coppersmith "magic formula" \cite{coppersmith,keane-rolles1}. 

This model naturally interpolates between the undirected ERRW and the directed Edge reinforced random walk: indeed, when $\star$ is the identity, then 
$(j^\star,i^\star)=(j,i)$ is dual to the edge $(i,j)$ and this
yields the standard ERRW. On the contrary, when the graph is composed of two disconnected pieces $V_1$ and $V_1^\star$ which are mapped to one  another by the involution $\star$ then, starting from $V_1$, the \sERRW remains on $V_1$ and follows the directed ERRW. A more interesting version
arises when one glues $V_1$ and $V_1^\star$ by the starting point of the process, see Section \ref{sec:sERRW} for more details.  

The undirected and directed ERRW have very different behavior: e.g. on $\Z^d$, $d\ge 3$, the first exhibits a phase transition between recurrence and transience \cite{ST15,ACK14,DST15} and the second is always transient \cite{sabot1}. The \sERRW lives on a directed graph but keeps some reversibility in the reinforcement scheme, in that respect it shows some features of both processes.

In \cite{BST20}, following \cite{ST15}, we introduced the \sVRJP and proved that the \sERRW is a mixture of  \sVRJPse. Let $(W_e)_{e\in E}$ be a family of positive $\star$-symmetric real weights on the edges of the graph. By abuse of terminology, and by analogy with the VRJP, we sometimes call $(W_e)$ the "conductances", even though those are not symmetric in general. In the exponential time-scale, the \sVRJP is the continuous-time process which, at time $t$, jumps from $i$ to $j$ at a rate given by
\begin{eqnarray}\label{jump-rates1}
W_{i,j}e^{T_i(t)+T_{j^\star}(t)},
\end{eqnarray}
where  $T_i(t)$ is the occupation time of the vertex $i\in V$ at time $t$. At fixed time $t$, the transition rates are obviously not reversible. However the $\star$-symmetry on the transition rates, see \eqref{symmetry-condition}, leads to a type of Yaglom reversibility \cite{Yaglom0,Yaglom,Dobrushin} in the limit, see below.
A Markov jump process on the graph $\ggg$, with jump rates $(K_{i,j})_{(i,j)\in E}$, will be called Yaglom reversible if it admits an invariant measure $(\pi_i)_{i\in V}$ such that $\pi_i=\pi_{i^\star}$ for all $i\in V$, and $\pi_{i} K_{i,j}=\pi_{j^\star}K_{j^\star,i^\star}$ for all edge $(i,j)\in E$.  

Note that after a change of time similar to the VRJP case, see \cite{ST15}, the jumps rates \eqref{jump-rates1} can also be written as $W_{i,j}(1+L_{j^{\star}}(s))$ where $L$ is the local of the new process. That corresponds to the original time-scale of the VRJP.

Contrary to the standard VRJP, the \sVRJP is not partially exchangeable, even though the \sERRW is partially exchangeable under a condition on the weights. This adds a major difficulty to the model. However, we obtain in Part~\ref{partI} a counterpart to the results of \cite{ST15}:
\begin{itemize}
\item
Under a suitable randomization of the initial local times, we prove that the \sVRJP becomes partially exchangeable after a proper time change. We identify the limiting measure and represent this randomized \sVRJP as a mixture of Yaglom reversible Markov jump processes (with the definition above). 
\item
Coming back to the non randomized process, we prove that in general the \sVRJP can be written, after a proper time change, as a mixture of some self-interacting jump processes which can be seen as conditioned Markov jump processes.
\end{itemize}
In a second part, we give a random Schr\"odinger representation of the randomized \sVRJPse, which generalizes the representation obtained in \cite{STZ17} for the VRJP. Several new phenomena and difficulties appear compared to the VRJP, due to the necessary extra randomization of the initial local times. The representation is based on new and rather remarkable identities between integrals on the space of $\star$-symmetric and  $\star$-antisymmetric functions.
The corresponding representation of the classical VRJP by a random Schr\"odinger operator has given rise to several important developments, such as recurrence for all reinforcement parameters in dimension two, diffusive behavior in the transient phase in dimension at least 3, a monotonicity property and the uniqueness of phase transition  \cite{SZ19,Poudevigne19,Sabot21,Kozma_Peled21}. We present more detailed motivations for that random Schr\"odinger representation later in Section~\ref{Motivation}, since it is requires a prior understanding of partial exchangeability and the mixing measure.  

A natural question which emerges from these works is that of the existence of a counterpart to the supersymmetric hyperbolic sigma-field $\H^{2|2}$, which is deeply linked to the VRJP (see e.g. \cite{DSZ10,ST15,BHS20}). This question will be the object of a further work with Andrew Swan.   

\tableofcontents

\part{Exchangeability and mixing measure}\label{partI}

\section{A non-reversible counterpart to the Edge-Reinforced Random Walk and the Vertex Reinforced Jump Process}
\subsection{$\star$-directed graphs}\label{s_star-graph}
In this paper, we consider a directed graph $\ggg=(V,E)$ endowed with an involution on the vertices denoted by $\star$, and such that
\begin{equation}
\label{invol}
(i,j)\in E\;\;\; \Rightarrow \;\;\;(j^\star,i^\star)\in E.
\end{equation}
We write $i\to j$ to mean that $(i,j)$ is a directed edge of the graph, i.e. that $(i,j)\in E$.

Denote by $V_0$ the set of fixed points of $\star$ 
$$
V_0=\{i\in V \hbox{ s.t. } i=i^\star\},
$$
and by $V_1$ a subset of $V$ such that $V$ is a disjoint union
$$
V= V_0\sqcup V_1\sqcup V_1^\star.
$$

Denote by $\tilde E$ the set of edges quotiented by the relation $(i,j)\sim(j^\star,i^\star)$.
In the sequel we also consider $\tilde E$ as a subset of $E$ obtained
by choosing a representative among two equivalent edges $(i,j)$ and $(j^\star,i^\star)$.

If $v=(v_i)_{i\in V}$ is a function on the vertices we simply write $v^\star$ for the function given by $v^\star_i=v_{i^\star}$. The spaces $\sss$ and $\aaa$ of  $\star$-symmetric and $\star$-antisymmetric functions on the vertices, defined below, play a central role throughout the paper,
\begin{eqnarray}\label{A}
\aaa=\left\{a\in \R^V, \; a^\star=-a\right\}
\simeq\R^{V_1},
\end{eqnarray}
\begin{eqnarray}\label{S}
\sss= \left\{s\in \R^V, \; s^\star=s\right\}.
\end{eqnarray}
The orthogonal projections on the subspaces $\aaa$ (resp. $\sss$), are given by
\begin{eqnarray}\label{Projection}
\ppp_{\aaa}(u)= \demi(u-u^\star), \;\;\; \ppp_{\sss} (u)= \demi (u+u^\star).
\end{eqnarray}
We denote by $\left<\cdot,\cdot\right>$ the bilinear form on $\R^V\times\R^V$ given by
\begin{eqnarray}\label{product}
\left<x,y\right>=\sum_{i\in V} x_{i^\star}y_i,
\end{eqnarray}
Note  that, in general, $\left<\cdot,\cdot\right>$ is not a scalar product. We also use sometimes the Euclidian scalar product on $\R^V$ that we denote by $(\cdot, \cdot)$,
so that $\left<x,y\right>=(x^\star,y)$.


A function $(\alpha_{i,j})_{(i,j)\in E}$ on the edges is called $\star$-symmetric if it satisfies
\begin{eqnarray}\label{symmetry-condition}
\alpha_{i,j}=\alpha_{j^\star,i^\star}\text{ for all }(i,j)\in E.
\end{eqnarray}

In Part~\ref{partI} we always assume that the graph $\ggg$ is strongly connected, i.e. that for any vertices $i$ and $j$, there exists a directed path in the graph between $i$ and $j$.

\subsection{The $\star$-Edge Reinforced Random Walk ($\star$-ERRW)} 
\label{sec:sERRW}
Suppose we are given some $\star$-symmetric positive weights $(\alpha_{i,j})_{(i,j)\in E}$ on the edges.
Let $(X_n)_{n\in\N}$ be a nearest-neighbor discrete-time process taking values in $V$. For all $(i,j)\in E$, denote by $N_{i,j}(n)$
the number of crossings at time $n$ of the directed edge $(i,j)$, 
$$
N_{i,j}(n)=\sum_{k=0}^{n-1}\indic_{(X_k,X_{k+1})=(i,j)},
$$
and let  
$$
\alpha_{i,j}(n)= \alpha_{i,j}+ N_{i,j}(n)+N_{j^\star,i^\star}(n).
$$
Note that $(\alpha_e(n))_{e\in E}$ remains $\star$-symmetric at all time $n$. 

The process $(X_n)_{n\in\N}$ is called the $\star$-Edge Reinforced Random Walk (\sERRWse), with initial weights $(\alpha_e)_{e\in E}$ and starting from $i_0$, 
if $X_0=i_0$ and, for all $n\in\N$ and $j\in V$, 
$$
\P(X_{n+1}=j\;|\; X_0,X_1, \ldots X_n)= \indic_{X_n\rightarrow j}{\alpha_{X_n,j}(n)\over \sum_{l, X_n\rightarrow l} \alpha_{X_n,l}(n)}.
$$
As mentioned above, the \sERRW is a generalisation of the Edge-Reinforced Random Walk (ERRW) introduced in the seminal work of Diaconis and Coppersmith \cite{coppersmith}, corresponding to the case where $\star$ is the identity map. 

An example of \sERRW  is the $k$-dependent Reinforced Random Walk, already considered by Bacallado in \cite{Bacallado2}, which takes values in the de Bruijn graph $\ggg=(V,E)$ of a finite set $S$, defined as follows: $V=S^k$, $E=\{((i_1, \ldots i_k),(i_2, \ldots, i_{k+1})),\,i_1,\ldots i_{k+1}\in S\}$. Let $\star$ be the involution mapping $(i_1, \ldots i_k)\in V$ to the reversed sequence $(i_k, \ldots i_1)\in V$, which obviously satisfies \eqref{invol}. Note that, in that case, $V_0$ is the set of palindromes. Then the $k$-dependent Reinforced Random Walk is defined as the \sERRW on that de Bruijn graph, see \cite{BST20}.

As we pointed out in the introduction, the Random Walk in Random Dirichlet Environment (RWDE), considered in \cite{Enriquez-Sabot06,sabot1} (see also \cite{sabot-tournier17} for a review), can also be seen as a particular case of \sERRWse. 
Indeed, given a directed graph $\ggg_1=(V_1, E_1)$, consider its "reversed" graph $\ggg_1=(V_1^\star, E_1^\star)$, where $V_1^\star$ is a copy of $V_1$ and $E_1^\star$ is obtained by reversing each edge in $E_1$. Then the full graph $\ggg$ is a disjoint union $(V_1\sqcup V_1^\star,E_1\sqcup E_1^\star)$. If $i_0\in V_1$, then the \sERRW is a Random Walk in Dirichlet environment on $\ggg_1$ under the annealed law.  An interesting modification of the previous construction, which makes the $\star$-directed graph strongly connected, is obtained by gluing the two graphs $\ggg_1$ and $\ggg_1^\star$ at the starting point $i_0$. In that case, the condition for partial exchangeability of the $\star$-ERRW, see condition \eqref{divergence} below, coincides with the condition for the key property of statistical invariance by time-reversal of RWDE, see \cite{sabot1} lemma~1 or \cite{sabot-tournier17}, Lemma~3 and Remark~1 of \cite{BST20}.

The divergence operator  $\dive: \R^E\to \R^V$ is defined on functions on the edges $(x_e)_{e\in E}$ by
$$
\dive(x)(i)= \sum_{j, i\to j} x_{i,j}-\sum_{j, j\to i} x_{j,i}.
$$
The following result, which extends the corresponding result of Baccalado \cite{Bacallado2} in the case of the  $k$-dependent ERRW, is proved in \cite{BST20}:

\begin{proposition}[\cite{BST20}]
\label{bac}
Let $i_0\in V$ be the starting point of the \sERRWse, and suppose that the weights $(\alpha_e)_{e\in E}$ satisfy the symmetry condition \eqref{symmetry-condition}, as well as the following property
\begin{eqnarray}\label{divergence}
\dive(\alpha)(i)=\indic_{i=i_0^\star}-\indic_{i=i_0}.
\end{eqnarray}
Then the \sERRW $(X_n)$ is partially exchangeable in the sense of \cite{diaconis-freedman}.
\end{proposition}

\subsection{The $\star$-Vertex Reinforced Jump Process ($\star$-VRJP)} \label{s-def-svrjp}
Let us now describe the representation of the  \sERRW in terms of a $\star$-Vertex-Reinforced Jump Process (\sVRJPse) with independent gamma conductances, similarly as for the ERRW. 

Suppose we are given $\star$-symmetric positive weights $(W_{i,j})_{(i,j)\in E}$, i.e. satisfying \eqref{symmetry-condition}, that we call 
"conductances" by analogy with VRJP, even though they are not symmetric (see introduction).
We often extend $W$ to a $V\times V$ matrix by setting $W_{i,j}=0$ if $(i,j)\not\in E$. 
Clearly, the matrix $W$ is $\left<\cdot,\cdot\right>$-symmetric, since $W_{i,j}=W_{j^\star,i^\star}$, where the bilinear form $\left<\cdot,\cdot\right>$ is defined in \eqref{product}.

If $(u_i)_{i\in V}\in \R^V$, 
define $W^u\in \R^E$ by
\begin{eqnarray}\label{Wu}
W^u_{i,j}= W_{i,j}e^{u_i+u_{j^\star}}.
\end{eqnarray}
Remark that $(W^u_e)_{e\in E}$ is $\star$-symmetric, i.e. $W^u_{i,j}= W^u_{j^\star,i^\star}$ for $(i,j)\in E$. 

Consider the process $(X_t)_{t\ge 0}$ with state space $V$, which, conditioned on the past at time $t$, with $X_t=i$, jumps from $i$ to $j$, such that $i\to j$, at a rate
\begin{eqnarray}\label{jump-rates}
W^{T(t)}_{i,j}=W_{i,j}e^{T_i(t)+T_{j^\star}(t)},
\end{eqnarray}
where $T(t)$ is the local time of $X$ at time $t$, i.e.
$$
T_i(t)=\int_0^t \indic_{X_u =i} \; du, \;\;\; \forall i\in V.
$$
We call this process the $\star$-VRJP.
When $\star$ is the identity, the jump rates in \eqref{jump-rates} correspond to the ones for the standard VRJP in the exponential time scale considered in~\cite{ST15}.
We denote by $\P^W_{i_0}$ the law of the \sVRJP starting from the point $i_0$. We always define the \sVRJP on the canonical space ${\mathcal{D}}([0,\infty),V)$ of c\`adl\`ag functions from $[0,\infty)$ to  $V$, and $(X_t)_{t\ge 0}$ will denote the canonical process on ${\mathcal{D}}([0,\infty),V)$. 

The following relation between the \sVRJP and \sERRW was given in~\cite{BST20}.
\begin{lemma}\label{lem_ERRW-VRJP}
Let $(\alpha_{e})_{e\in E}$ be $\star$-symmetric positive weights on the edges.
Let $(W_e)_{e\in \tilde E}$ be independent Gamma random variables with weights $(\alpha_{e})_{e\in \tilde E}$, and let $(W_e)_{e\in E}$ be the conductances on $E$ constructed to be $\star$-symmetric.
Let $(X_t)_{t\ge 0}$ be the $\star$-VRJP with conductances $(W_e)$ starting from $i_0$, 
i.e. the process with law  $\P^W_{i_0}$, and let $(\tilde X_n)_{n\in \N}$ be
the discrete time process describing the successive jumps of $(X_t)_{t\ge 0}$.
Then, after taking the expectation with respect to the r.v. $W$, under $\P_{i_0}^W$,  $(\tilde X_n)$ has the law of the $\star$-ERRW with initial weights $(\alpha_e)_{e\in E}$ and starting at $i_0$.

\end{lemma}
\section{The limiting manifold}\label{s_limiting}
We will use the following notation :
$$
\hhh_0=\{(u_i)_{i\in V}, \;\; \sum_{i\in V} u_i =0\}, \;\;\; \sss_0=\sss\cap \hhh_0.
$$
and we remark that $\aaa\subset \hhh_0$.
For the $\star$-symmetric conductances $(W_e)_{e\in E}$ 
we let $\lll^W_0$ be the following manifold (which, as we will see later, is smooth)
$$
\lll_0^W=\{u\in\hhh_0, \;\; \dive(W^u)(i)=0 \hbox{ for all $i\in V$}\},
$$
where $W^u$ is the element of $(0,\iy)^E$ defined in \eqref{Wu}.
The following lemma, proved in Appendix~\ref{pflemconv}, asserts that $\lll_0^W$ is the manifold of limiting values of the local time.
\begin{lemma}\label{Convergence}
For all $a\in\aaa$, under $\P_{i_0}^{W^a}$, for all $i\in V$,
$$
 U_i:= \lim_{t\to \infty} \left(a_i+T_i(t)-t/N\right) \tx{ exists a.s.}
$$
and $U:=(U_i)_{i\in V}\in \lll^W_0$ a.s..
\end{lemma}
\begin{remark} Note that for $a=0$, $U$ is the limit of the centered local time of the \sVRJP under $\P_{i_0}^{W}$. In general, in the sequel, $U$ is defined as in Lemma \ref{Convergence}, as this will turn out to be convenient for notational purposes.
\end{remark} 
The following result shows that $\aaa$ is transversal to the manifold $\lll^W_0$.
\begin{lemma}[Projection on $\lll^W_0$ parallel to $\aaa$] \label{transversal}
\label{lem:proj}
Let $u$ be in $\hhh_0$. There exists a unique element $a\in \aaa$ such that
$$
u-a\in \lll_0^W.
$$
Moreover, $a^\star$ is the unique minimizer of the strictly convex 
function $J_u: \hhh_0\mapsto \R$
$$
J_u(v)= 
\sum_{i\in V}\sum_{j,i\to j}W_{i,j}^u(e^{v_j-v_i}-1).
$$
We denote by $p_{\lll_0^W}(u):=u-a $ this projection on $\lll_0^W$ parallel to $\aaa$.
\end{lemma}
{Lemma \ref{lem:proj} gives a natural parametrization of $\lll_0^W$,
by orthogonal projection on the subspace $\sss_0$:
\begin{eqnarray*} 
\label{dsigma}
\lll_0^W& \rightarrow& \sss_0
\\
h&\mapsto & s= \ppp_{\sss} (h)= \demi (h+h^\star).
\eeq
This parametrization is inverted by the application $h=p_{\lll_0^W}(s)$, since
$s$ is in the fiber $h+\aaa$. Hence it gives a natural positive measure on $\lll_0^W$, denoted by $d\sigma_{\lll_0^W}$, and defined by:
\begin{eqnarray}\label{volume}
d\sigma_{\lll_0^W}= 
(p_{\lll_0^W})_* (d\lambda_{\sss_0})
\end{eqnarray}
where  $(p_{\lll_0^W})_*(d\lambda_{\sss_0})$
is the push-forward
of the euclidean volume measure $\lambda_{\sss_0}$ on $\sss_0$ by the projection on $\lll_0^W$ restricted to $\sss_0$. More precisely we have, for any measurable subset $B$ of $\lll_0^W$, 
\begin{eqnarray}\label{volume_bis}
\sigma_{\lll_0^W}(B)=\lambda_{\sss_0}((p_{\lll_0^W})^{-1}(B))=\lambda_{\sss_0}(\ppp_S(B)).
\end{eqnarray} 
}

\section{Randomization of the initial local time and exchangeable time scale}\label{s_randomization}
\label{pervrjp}
Contrary to the VRJP, the $\star$-VRJP 
is in general not a mixture of time changed Markov Jump Processes. 
Indeed, otherwise this would imply that,
taking expectation with respect to random conductances $W$ as in Lemma~\ref{lem_ERRW-VRJP}, the \sERRW were a mixture of 
random walks, hence partially exchangeable, but it is not the case when the weights do not have the
property of null divergence, see Proposition \ref{bac}. Nevertheless, the \sVRJP has some exchangeability property
after randomization of the initial local times. The results of this Section are proved in Section~\ref{ss_Proof_random}. 

For $i_0\in V$, let $\nu^W_{\aaa,i_0}$ be the measure on $\aaa$ defined by
\begin{eqnarray}\label{nu_a}
\nu^{W}_{\aaa,i_0}
(da) =
{1\over \sqrt{2\pi}^{\vert V_1\vert}}
e^{a_{i_0^\star}}e^{-\demi\sum_{(i,j)\in E} W_{i,j}(e^{a_i+a_{j^\star}}-1)}(\prod_{i\in  V_1} da_i),
\end{eqnarray}
and let
\begin{eqnarray}\label{FWi}
F^W_{i_0}:=
\int_\aaa \nu_{\aaa,i_0}^W(da)
\end{eqnarray}
be its normalization constant.
\begin{lemma}\label{finitness}
Under the assumption that $\ggg$ is strongly connected, we have $F^W_{i_0}<\infty$.
\end{lemma}
 We will be interested in the \sVRJP starting from an antisymmetric random initial local time $A$ distributed according to ${1\over F^W_{i_0}} \nu^{W}_{\aaa,i_0}$, i.e. the \sVRJP with conductances $W^A$. To that purpose, we define the probability measure $\overline \P^W_{i_0} (\cdot )$ on the canonical space $ \ddd([0,\infty),V)\times \aaa$ by
$$
\int \phi(\w,a) \overline \P^W_{i_0} (d\w,da)= \int_{\ddd([0,\infty), V)\times \aaa} \phi(\w,a) \P^{W^a}_{i_0}(d\w) {1\over F^W_{i_0}} \nu^{W}_{\aaa,i_0}(da),
$$
for any measurable bounded test function $\phi$. On the canonical space $ \ddd([0,\infty), V)\times \aaa$, we define $(X_t)_{t\ge 0}$ as the canonical random variables on path space $\ddd([0,\infty), V)$,
and  $A$ as the random variable on $\aaa$ equal to the second coordinate. Hence, under $\overline \P^W_{i_0} (\cdot )$, $A$ is distributed according to  ${1\over F^W_{i_0}}\nu^{W}_{\aaa,i_0}$, and conditionally on $A$, $(X_t)_{t\ge 0}$ is distributed according to $\P^{W^A}_{i_0}$. We call the associated process $(X_t)$ under  $\overline \P^W_{i_0}$, the {\it randomized \sVRJP}. 

By a slight abuse of notation, the marginal of $\overline \P^W_{i_0} (\cdot )$ on the canonical space $\ddd([0,\infty), V)$ will also be denoted by $\overline \P^W_{i_0} (\cdot )$. It will be clear from the context whether $\overline \P^W_{i_0} (\cdot )$ represents the probability on $\ddd([0,\infty), V)\times \aaa$ or its marginal on $\ddd([0,\infty), V)$.

The following Lemma \ref{flow} shows the stability of the family of laws ${1\over F^W_{i_0}} \nu^{W}_{\aaa,i_0}$ under the posterior distribution of $ \overline \P^{W}_{i_0}$, which enables to deduce in Corollary~\ref{cor_jump_rate} the jump rate of the annealed \sVRJPse.

\begin{lemma}\label{flow}
Under  $\overline \P^{W}_{i_0}$, conditionally on $\sigma\{X_s, \; s\le t\}$, 
$(A_i)$ is distributed according to ${1\over F^{W^{T(t)}}_{X_t}}\nu^{W^{T(t)}}_{\aaa,X_t}$.
\end{lemma}
\begin{corollary}
\label{cor_jump_rate}
Under $\overline \P^W_{i_0}$, conditionally on the the past at time $t$, if the process $X_t$ is at position $i$, then it 
jumps to $j$ such that $i\to j$ at a rate 
$$
W_{i,j}^{T(t)} {F^{W^{T(t)}}_j\over F^{W^{T(t)}}_i}.
$$
\end{corollary}


The following Proposition \ref{exchangeability} shows that, with that randomization of initial local time, the \sVRJP becomes partially exchangeable after a time-change, which is
a counterpart to the time change introduced in \cite{ST15}, Theorem~2.
\begin{proposition}\label{exchangeability}
Let $(C(t))_{t\ge 0}$ be the increasing functional of the process $(X_t)$ defined by
$$
C(t)=\demi\sum_{i\in V} (e^{T_i(t)+T_{i^\star}(t)}-1).
$$
Let $(Z_s)_{s\ge 0}$ be the time changed process defined by
$$
Z_s=X_{C^{-1}(s)}.
$$
Then, under the randomized distribution  $\overline \P^W_{i_0}$,  $(Z_s)$ is partially exchangeable in the sense of \cite{Freedman}, i.e. for any real $h>0$, the discrete time process 
$(Z_{kh})_{k\in \N}$ is partially exchangeable in the sense of Diaconis  and Freedman \cite{diaconis-freedman}. 
\end{proposition} 
\begin{remark}
This partial exchangeability implies by \cite{diaconis-freedman} that the process $(Z_s)_{s\ge 0}$ under the randomized distribution $\overline \P^W_{i_0}$ is a mixture of Markov jump processes. This is the object of Theorem~\ref{main}, where we make explicit the law of the mixture of Markov jump processes.
\end{remark} 
\begin{remark}
Note that, compared to the original time-change for the VRJP in \cite{ST15}, Theorem~2, there is an extra $\demi$ in the definition of $C(t)$, which will lead to a multiplicative factor $2$ in the jump rate in Theorem~\ref{main}~(ii), which applied to the VRJP is $W_{i,j} e^{U_j-U_i}$, instead of $\demi W_{i,j}e^{U_j-U_i}$ in \cite{ST15}. This choice removes several annoying factors $\demi$ in the statements. 
\end{remark}

The following lemma states that the  $\star$-ERRW is a mixture of these randomized $\star$-VRJP, under the condition that the weights of the $\star$-ERRW makes it partially exchangeable.
\begin{lemma}
\label{randomized-mixture}
Let $(\alpha_{e})_{e\in E}$ be $\star$-symmetric positive weights which satisfy divergence-free condition \eqref{divergence}.
Assume that $W=(W_e)_{e\in \tilde E}$ are independent Gamma random variables with weights $(\alpha_{e})_{e\in \tilde E}$ and that $(W_e)_{e\in E}$ are the associated $\star$-symmetric conductances. If $A$ has the law ${1\over F_{i_0}^W} \nu_{\aaa,i_0}^W$, then $W^A$
has the same law as $W$. It implies by Lemma~\ref{lem_ERRW-VRJP} that after expectation with respect to $(W_e)$, the discrete time process associated with the randomized \sVRJP is a \sERRWse.
\end{lemma}
\begin{remark}
The stability property in Lemma~\ref{randomized-mixture} cannot hold when the divergence condition \eqref{divergence} is not satisfied, since otherwise it would imply by Proposition \ref{exchangeability} above that the $\star$-ERRW is always
exchangeable, irrespective of that divergence condition.
\end{remark}
\begin{remark} In the classical reversible model, the VRJP was instrumental in the analysis of the asymptotic behavior of the ERRW. In this respect, the previous Lemma is important, since it means that the exchangeable \sERRW can be described by randomized \sVRJP with independent gamma random conductances.
\end{remark}

\section{The limiting distribution and representation as mixtures}
\subsection{The case of the randomized \sVRJP}
\label{sec:limdis}
Let us first introduce some notation. For $u\in \uuu_0^W$, we denote by $K^u=(K^u_{i,j})_{i,j\in E}$  the infinitesimal generator
of the Markov jump process at rate $W^u_{i,j}$ from $i$ to $j$, 
i.e. we let $K^u=(K^u_{i,j})_{i,j\in E}$ be the matrix defined by
\begin{align}\label{Ku}
K^u_{i,j}=
\begin{cases}
W_{i,j}^u,& \hbox{ if $i\neq j$,}
\\
-\sum_{l\in V}W_{i,l}^u, & \hbox{ if $i=j$.}
\end{cases}
\end{align}
 We set
$$
D(W):=\sum_{T\in \ttt_{j_0}} \prod_{(i,j) \in T} W_{i,j},
$$
where the sum is on the set $\ttt_{j_0}$ of {\bf directed}  spanning trees $T$ of the graph {\bf rooted} at some vertex $j_0\in V$. By the matrix-tree theorem, we also have
$$
D(W^u)= \det(-K^u)_{V\setminus\{j_0\},V\setminus\{j_0\}}
$$
where $\det(-K^u)_{V\setminus\{j_0\},V\setminus\{j_0\}}$ is the diagonal minor obtained after removing the line and the column $j_0$ of the matrix $-K^u$.
The value of this determinant does not depend on the choice of $j_0$ since the sums on any line or column of $K^u$ is null, using $u\in \uuu_0^W$.

We denote by
$$
\hbox{det}_\aaa(K^u)
$$
the determinant of the linear operator $\ppp_{\aaa} K^u \ppp_\aaa$
restricted to the subspace $\aaa$, where $\ppp_\aaa$ is the projection on $\aaa$, see \eqref{Projection}.

\begin{definition}\label{def_mui0W}
We introduce the following positive measure on the manifold $\lll_0^W$:
\begin{eqnarray}
\label{mui0W}
\mu_{i_0}^{W}(du):=C_{\ggg} e^{u_{i_0^\star}-\sum_{i\in V_0} u_i} e^{-\demi\sum_{i\to j} W_{i,j} (e^{u_{j^\star}-u_{i^\star}}-1)} {\sqrt{D(W^u)}\over \det_{\aaa}( -K^u)}\sigma_{\lll_0^W}(du),
\end{eqnarray}
where
$$C_{\ggg}:={{\sqrt{\vert V\vert }{\sqrt{2}^{-\vert V_1\vert}}}\over\sqrt{2\pi}^{\vert V_0\cup V_1\vert-1}},$$
and $\sigma_{\lll_0^W}$ is the volume measure on the manifold $\lll_0^W$ defined in Section~\ref{s_limiting}.
\end{definition}

\begin{theorem} \label{main}

{\bf(i)} Under $\overline \P^W_{i_0}$, the following limit exists a.s.
$$
U_i:=\lim_{t \to \infty} \left( A_i+T_i(t)-t/N \right),
$$
and $U\in \lll_0^W$ is distributed according to
$$
{1\over F^W_{i_0}}\mu_{i_0}^{W}(du).
$$

{\bf(ii)} Under $\overline \P^W_{i_0}$, conditionally on $U$,  the \sVRJP in exchangeable time scale, $(Z_t)_{t\ge 0}$, is a Yaglom reversible Markov jump process with jump rate from $i$ to $j$ equal to
$$
W_{i,j} e^{U_{j^\star}-U_{i^\star}}.
$$
\end{theorem}
\begin{remark}\label{identity_mixing}
We deduce from {\bf(i)} that ${1\over F^W_{i_0}}\mu_{i_0}^{W}$ is a probability measure, which implies the following identity of integrals:
$$
\int_{\lll_0^W}  \mu_{i_0}^{W}(du)= F^W_{i_0}=\int_\aaa
\nu_{\aaa,i_0}^W(da).
$$
This is a "non trivial" identity
involving a measure
on the limiting manifold $\lll_0^W$ and a measure on the transversal space $\aaa$. We provide two different proofs of this equality : a probabilistic one using Feynman-Kac formula in this Part I,
and another one is by direct computation, in Part~\ref{partII}, see Corollary~\ref{thm:identity_0_b}. A supersymmetric interpretation of that identity will also be given in a forthcoming work with Swan.
\end{remark}
\begin{remark}
When all points are self-dual, i.e. $V_0=V$ and $V_1=\emptyset$, one can check that the limiting distribution \eqref{mui0W} coincides with the one of Theorem~1 of \cite{ST15}. The difference in the factor $\sqrt{\vert V\vert}$ instead of ${\vert V\vert}$ comes from the choice of the volume measure $\sigma_0^W(du)$: indeed, in the case of a self-dual graph, $\sigma_0^W$ is the Euclidean measure on $\hhh_0=\{(u_i),\; \sum_i u_i=0\}$, whereas in \cite{ST15} we chose the measure $\prod_{i\neq i_0} du_i$ (there is a factor $\sqrt{\vert V\vert}$ between the two).
\end{remark}
\subsection{The case of the (non-randomized) \sVRJPse
}
\label{s_non-randomized}
Recall that $(X_t)_{t\ge 0}$ is the canonical process on the space  $\ddd([0,\infty), V)$, and that on that space, $\P^W_{i_0}$ is the law of the \sVRJP starting at $i_0$ and $\overline\P^W_{i_0}$ the law of the randomized \sVRJPse. 
Besides, $(Z_s)_{s\ge 0}:=(X_{C^{-1}(s)})_{s\ge 0}$ is the time-changed of the canonical process $X$, as defined in Proposition~\ref{exchangeability}. We will denote by $(\ell_i(s))_{i\in V, s\ge 0}$ the local time process of $(Z_s)$:
\begin{eqnarray}\label{loc-time-Z}
\ell_i(s)=\int_0^s \indic_{Z_u=i}\; du, \;\;\; i\in V,\; s\ge 0.
\end{eqnarray}

Given $u\in \uuu_0^W$, we define $P^{W,u}_{i_0}$ as the law on the path space $\ddd([0,\infty), V)$ such that the time-changed process $(Z_t)$ on the graph $\mathcal{\ggg}$, is a Markov jump process starting from $i_0$ and with jump rate from  $i$ to $j$ given by
$$
W_{i,j}e^{u_{j^\star}-u_{i^\star}}.
$$
It is equivalent to say that, under  $P^{W,u}_{i_0}$, the canonical process $(X_t)_{t\ge 0}$ is the jump process which, conditioned on the past at time $t$, jumps from $i$ to $j$, $i\to j$, at rate
$$
W_{i,j}e^{u_{j^\star}-u_{i^\star}} e^{T_i(t)+T_{i^\star}(t)},
$$
see forthcoming Lemma~\ref{canonical-time-change}. With this notation, Theorem~\ref{main}~(ii) means that on $\ddd([0,\infty), V)$, we have the following equality of probabilities
$$
\overline\P^W_{i_0}(\cdot)=\int_{\uuu_0^W} P^{W,u}_{i_0} (\cdot) {1\over F^W_{i_0}}\mu_{i_0}^W(du).
$$
 
We introduce the following functional of the process $(Z_s)$, which plays a crucial role in the analysis of the non-randomized \sVRJP: for all $(\theta_i)_{i\in V}\in \sss$, let
\begin{eqnarray}\label{def_A}
B^\theta_i(s)=\demi \int_0^s {\indic_{\{Z_u=i\}}-\indic_{\{Z_u=i^\star\}}\over \theta_i+\ell_i(u)+\ell_{i^\star}(u)}du, \;\;\; \forall i\in V, \; s\ge 0.
\end{eqnarray}
We denote by $B^\theta(s)=(B^\theta_i(s))_{i\in V}$ the associated random vector, taking values in $\aaa$. We simply write $B(s):=B^1(s)$ when $\theta_i=1$ for all $i\in V$.

Obviously, under $P^{W,u}_{i_0} $, $(Z_s)$ is a Markov jump process with invariant measure $e^{u_i+u_{i^\star}}$, since $u\in\uuu_0^W$. This implies the following simple result.
\begin{proposition}\label{Prop_convergence_B}
Under $P^{W,u}_{i_0} $, a.s., there exists a random vector $B^\theta(\infty)\in \aaa$ such that 
$$
\lim_{s\to\infty} B^\theta(s)=B^\theta(\infty).
$$
Moreover, $B^\theta(\infty)$ has a density  on $\aaa$ that we denote by $f_{i_0}^{W,u,\theta}$. We simply write $f_{i_0}^{W,u}:=f_{i_0}^{W,u,1}$ when $\theta_i=1$ for all $i \in V$.

Besides, $B^\theta(\infty)$ has the following additivity property:
\begin{eqnarray}
\label{add}
B^\theta(\infty)=B^\theta(s)+B^{\theta+\ell(s)+\ell^\star(s)}(\infty)\circ\Theta_s^Z,
\end{eqnarray}
where $\Theta_s^Z$ is the time-shift of the trajectories of $Z$ (i.e. $\Theta^Z_s=\Theta_{C^{-1}(s)}$ if $\Theta_t$ is the time shift on the canonical space).
\end{proposition}

We will be interested in the conditioned law, defined for a.e. $b\in \aaa$,
$$
P^{W,u}_{i_0}\left(\; \cdot\; \vert B^1(\infty)=b\right),
$$
under which $(Z_s)$ is a conditioned Markov process.
The previous conditioned law should be understood as an $h$-process, which has jump rate at time $t$ from $i$ to $j$ given by 
$$W_{ij}^u \frac{f_j^{W,u,1+\ell(t)+\ell^\star(t)}(b-B^1(t))}{f_i^{W,u,1+\ell(t)+\ell^\star(t)}(b-B^1(t))}.$$

We are now ready to state the main theorem concerning the non-randomized \sVRJPse.
\begin{theorem}[Limit theorem for the non-randomized $\star$-VRJP]\label{main_bis}
Given $i_0\in V$, for Lebesgue-a.e. conductances $W$, under the law of the non-randomized $\star$-VRJP $\P^{W}_{i_0}$,  $U$ has law
\begin{eqnarray}\label{density-U_NR}
f_{i_0}^{W,u} \left((u-u^\star)/2\right) \cdot \mu_{i_0}^W(du).
\end{eqnarray}
Conditionally on $U$, the time-changed $\star$-VRJP $Z$ has the law of a Markov jump process $P^{W,u}_{i_0}$ conditioned on $B^1(\infty)=\demi(U-U^\star)$, and thus it has a jump rate  at time $t$ from $i$ to $j$ given by 
$$W_{ij}^{U_{j^\star}-U_{i^\star}} \frac{f_j^{W,U,1+\ell(t)+\ell^\star(t)}(\demi(U-U^\star)-B^1(t))}{f_i^{W,U,1+\ell(t)+\ell^\star(t)}(\demi(U-U^\star)-B^1(t))}.$$
\end{theorem}
{
\begin{remark}  Theorem \ref{main_bis} is an easy consequence of Theorem~\ref{main} by Bayes formula, and is proved in Section \ref{proof:main}. The statement is a.s. in $W$, even though it should be true for all $W$, but the latter would require some regularity of the density $f_{i}^{W,u}(b)$, both in its variable and its parameters. This does not seem completely obvious but should be doable with some more work.
\end{remark}
}

\part{Random Schr\"odinger representation}\label{partII}
\section{Motivation}\label{Motivation}
As explained in Remark~\ref{identity_mixing} above, Theorem~\ref{main} implies the following identity we rewrite for convenience:
\begin{eqnarray}\label{key_identity}
\int_\aaa \nu_{\aaa,i_0}^W(da)=\int_{\lll_0^W}\mu_{i_0}^{W}(du):=F_{i_0}^W.
\end{eqnarray}
In the special case of the VRJP, i.e. when $\star$ is the identity, the left-hand side term is trivially equal to 1 and  $\mu_{i_0}^W(du)$
is the first marginal of the supersymmetric $\H^{2|2}$ model considered by Disertori, Spencer and Zirnbauer \cite{DSZ10}. 

Thus the equality above generalizes the property that the mixing measure of the VRJP is normalized by a constant which does not depend on the weights $W$. That property plays a fundamental role in the analysis of the VRJP and has been proved in three very different ways. The first proof is due to Disertori Spencer and Zirnbauer (\cite{DSZ10}, Theorem (1.4) page 437): it uses that the measure is the marginal of the $\H^{2|2}$ model, and that the $\H^{2|2}$ model is invariant by a supergroup of transformations. The second proof, due to the authors of this paper, is probabilistic (\cite{ST15}, Theorem~2), and is a consequence of the representation of the VRJP as a mixture of Markov jump processes: Theorem~\ref{main} 
above generalizes that approach to the \sVRJPse. The third proof, due to the Sabot, Tarr\`es and Zeng \cite{STZ17}
is based on a direct computation going through a representation of the measure by a random Schr\"odinger operator. 
That representation has proved to be very useful in the investigation of the properties of the VRJP, see below.
In Part~\ref{partII}, we generalize that  representation to the \sVRJPse, several new phenomena appear due to the extra randomization by the measure $\nu_{\aaa,i_0}^W$. 


Before we state the main results of this part, we briefly recall the corresponding results in the standard VRJP case. As explained above, the VRJP corresponds to the case where $\star$ is the identity, in which case we can consider $\ggg=(V,E)$ as a non-directed graph with some conductances on the non-directed edges $(W_e)_{e\in E}$.
Assume that $V$ is finite. If $\beta=(\beta_i)_{i\in V}$ is a function of the vertices we define the Schr\"odinger operator $H_\beta=-W+\beta$ where $W$ is the (symmetric) matrix of conductances (with value $0$ at $(i,j)$ when $\{i,j\}$ is not an edge of the graph) and $\beta$ is the operator of multiplication by $\beta$ (considered as a potential on the graph). We write $H_\beta>0$ when $H_\beta$ is positive definite, in which case we set $G_\beta=(H_\beta)^{-1}$, which has positive coefficients if the graph is connected, as the inverse of an $M$-matrix. In \cite{STZ17}, definition~1, we introduced a probability distribution on the set of potentials $\beta=(\beta_i)_{i\in V}$, which in its simplest form has the following expression
\begin{eqnarray}\label{VRJP_beta}
\nu_V^{W}(d\beta)= {\indic_{H_\beta >0} \over \sqrt{2\pi}^{\vert V\vert}} \indic_{H_\beta >0} {e^{-\demi \left<1, H_\beta 1 \right>}\over \sqrt{\det(H_\beta)}}d\beta,
\end{eqnarray}
where the $1$ above denotes the vector in $\R^V$ equal to 1 in each coordinate.
\footnote{Note that our $\beta$ stands for $2\beta$ in \cite{STZ17}.}
The property 
\begin{eqnarray}\label{key_identity_beta}
\int \nu_V^{W}(d\beta) =1
\end{eqnarray}
 is non-trivial and was proved by direct computation in \cite{STZ17}. It is also related to the identity \eqref{key_identity} above in the  case  of the VRJP, where the left-hand side is 1. Indeed, \cite{STZ17}~Theorem~3 states that, for any $i_0\in V$, the random field $t=(t_i)_{i\in V}$ defined by
$$
e^{t_i}={G_\beta(i_0,i)\over G_\beta(i_0,i_0)},\;\;\; \forall i\in V,
$$
where $\beta$ is a random potential distributed according to $\nu_V^{W}$, is the mixing field of the VRJP starting at $i_0$, rooted at $t_{i_0}=0$: as a consequence, $(u_i)_{i\in V}=\left(t_i-\vert V\vert^{-1}\sum_{j\in V}t_j\right)_{i\in V}$ has the law $\mu_{i_0}^{W}$ of definition \eqref{mui0W}. 

Letac and Weso\l{}owski noticed that the distribution~\eqref{VRJP_beta} belongs to a larger and more natural family of measures, which has a remarkable property of stability by restriction and conditioning (proved independently in \cite{SZ19, LW17}). The representation of the VRJP by a random Schr\"odinger operator was instrumental in several directions, in particular it enabled the representation by Sabot and Zeng of the VRJP as a mixture of Markov processes on infinite graphs and the characterization of the recurrence/transience of the VRJP in terms of the existence of a certain delocalized eigenvector of the Schr\"odinger operator at the ground state \cite{SZ19}, and the proof of a monotonicity property \cite{Poudevigne19}. The random Schr\"odinger representation and its consequences were instrumental in proving recurrence in 2D of the ERRW and the VRJP for any constant reinforcement parameter \cite{SZ19,Sabot21,Kozma_Peled21}, diffusive behaviour in dimensions larger or equal to 3 at weak reinforcement \cite{SZ19}, and the uniqueness of the phase transition in dimension $d\ge 3$ \cite{Poudevigne19}.

We generalize below the construction of the random potential $\beta$ of \eqref{VRJP_beta} to the case of the \sVRJPse. The potential $\beta=(\beta_i)_{i\in V}$ lives on the space $\sss$ of $\star$-symmetric functions on $V$, defined in \eqref{S}, 
and we denote below by $\nu_\sss^W(d\beta)$ the corresponding measure, which is introduced later. The main results are the following:
\begin{itemize}
\item
We generalize the key property \eqref{key_identity_beta} to an identity of the type $\int_\sss \nu_\sss^{W}(d\beta)=\int_\aaa \nu_\aaa^{W}(da)$, see \eqref{eq:beta} below, where $\nu_\aaa^{W}(da)$ is a positive measure living on the space $\aaa$ of $\star$-antisymmetric functions and closely related to the law of the initial local time defined in \eqref{nu_a}. The proof works by induction and goes through a family of measures involving both $\beta$ and $a$ variables in complementary $\star$-symmetric subsets. Surprisingly, the initialization of this induction is difficult and uses the Lagrange resolvent method for solving fourth degree polynomial equations.
\item
We relate the distribution $\nu_\sss^W$ to the randomized \sVRJPse, see Section~\ref{sec_VRJP} below: roughly speaking, the random potential $\beta$ describes the law of the asymptotic jump rates of the time-changed \sVRJPse. For any $\star$-symmetric subset $I\subset V$, we identify the law of the randomized \sVRJP conditioned on $\beta_I$ and $A_{I^c}$: this representation is new, even in the standard VRJP case. This enables us to express the mixing law of the \sVRJP $\mu_{i_0}^W$ defined in \eqref{mui0W} in terms of the law of the random potential $\nu_\sss^W$, and in particular it gives a new proof of Theorem~\ref{main} and a purely computational proof of the identity \eqref{key_identity}.
\end{itemize}

\section{The $\beta$-potentials and fundamental properties}
\label{sec_beta_pot}
\subsection{Definition}
Recall the notations and definitions of Section~\ref{s_star-graph}. However, in this Part~\ref{partII}, it is more convenient not to suppose that the $\star$-directed graph $\ggg$ is strongly connected. 
We remind from the beginning of Section~\ref{s-def-svrjp} that $(W_{i,j})_{(i,j)\in E}$ are $\star$-symmetric positive weights, called conductances by abuse of terminology, on the graph $\ggg$ and that $W=(W_{i,j})_{i,j\in V}$ also denotes the matrix of conductances with $W_{i,j}=0$ when $(i,j)\notin E$.

For all $\beta\in \sss$, let $H_\beta$ be the Schr\"odinger operator defined by
$$
H_\beta=\beta-W,
$$
where $\beta$ is the operator of coordinate multiplication by $(\beta_i)_{i\in V}$.
We write $H_\beta>0$ when $H_\beta$ is positive stable, i.e.  all its eigenvalues have positive real parts. When $H_\beta>0$, $H_\beta$ is a non-singular $M$-matrix (see e.g. \cite{berman1994nonnegative} or Appendix~\ref{Append_M_matrices}) and  the inverse matrix 
$$
G_\beta=H_\beta^{-1}
$$
is well-defined and has positive coefficients between any two vertices 
$i$ and $j$ such that there is a directed path in $\ggg$ from $i$ to $j$, and coefficient 0 otherwise,
see \cite{berman1994nonnegative}, Theorem~2.3~$(N_{38})$ or Proposition~\ref{prop-M} in Appendix~\ref{Append_M_matrices}.

Clearly, the matrices $W$, $H_\beta $, for $\beta\in \sss$, are $\left<\cdot,\cdot\right>$-symmetric since $W_{i,j}=W_{j^\star,i^\star}$ and $\beta_{i^\star}=\beta_i$, where $\left<\cdot,\cdot\right>$ is the bilinear form defined in \eqref{product}. 

The random Schr\"odinger representation of the \sVRJP is based on key integral identities on the spaces $\sss$ and $\aaa$. 
The following results give an extension of Theorem~1 in  \cite{STZ17} for the \sVRJPse, and of its generalized form given in Theorem~2.2 of \cite{LW17}.

\begin{definition}\label{Def_beta}
For all $\theta\in (0,\iy)^V$, $\eta\in[0,\iy)^V$, we define $\nu_\sss^{W,\theta,\eta}(d\beta)$ as the measure on $\sss$ defined by
\begin{align}
\label{eq:nu_beta}
\nu_\sss^{W,\theta,\eta}(d\beta)= \left( \prod_{i\in V_0} \theta_i \right)  {\indic_{H_\beta>0}\over \sqrt{2\pi}^{\dim(\sss)}}  \exp\left(-\demi \left<\theta, H_\beta \theta\right>-\demi \left<\eta, G_\beta \eta\right>+\left<\theta,\eta\right>\right){d\beta\over \sqrt{\vert H_\beta \vert}},
\end{align}
where $d\beta=\prod_{i\in \tilde V} d\beta_i$, 
and $\nu_\aaa^{W,\theta,\eta}(d a)$ as the measure on $\aaa$ defined by 
\begin{align}
\label{eq:nu_a}
\nu_\aaa^{W,\theta,\eta}(d a)= {1\over \sqrt{2\pi}^{\dim(\aaa)}} \exp\left(-\demi\left<e^{a^\star}\theta, We^{a^\star}\theta\right>+\demi\left<\theta,W\theta\right>-\left<\eta,e^{a^\star}\theta-\theta\right> \right) da,
\end{align}
where $da=\prod_{i\in V_1} da_i$ and where $e^{a^\star}\theta =(e^{a_{i^\star}}\theta_i)_{i\in V}$. If $V_1=\emptyset$ then $\aaa=\{0\}$ and by convention we set $\nu_\aaa^{W,\theta,\eta}=1$.
When $\theta_i=1$ and $\eta_i=0$ for all $i\in V$ we simply write $\nu_{\sss}^{W}(d\beta)$ and $\nu_{\aaa}^{W}(da)$.
\end{definition}
\begin{theorem}\label{Thm_beta}
For all $\theta\in (0,\iy)^V$, $\eta\in(\R_+)^V$, we have
\begin{align}
\label{eq:beta}
\int_\sss \nu_\sss^{W,\theta,\eta}(d\beta) = \int_\aaa \nu_\aaa^{W,\theta,\eta}(d a).
\end{align}
 We denote by $F^{W,\theta, \eta}=\int_\aaa \nu_{\aaa}^{W,\theta,\eta}(da)=\int_\sss\nu_{\sss}^{W,\theta,\eta}(d\beta)$ the integrals corresponding to \eqref{eq:beta}. We simply write $F^W$ when $\theta=1$ and $\eta=0$. Finally, $F^{W,\theta, \eta}<\infty$ as soon as for any vertex $i\in V$, there exists a directed path in $\ggg$ from $i$ to $i^\star$ or a directed path from $i$ to a vertex $j$ such that $\eta_j>0$.
\end{theorem}
\begin{remark}
Remark that, when $\theta=1$ and $\eta=0$, the measures $\nu_{\aaa}^W$ above and $\nu_{\aaa,i_0}^W$ introduced in \eqref{nu_a} are related by the simple formula: 
$\nu_{\aaa,i_0}^W(da)= e^{a_{i_0^\star}} \nu_{\aaa}^W(da)$. In particular they are equal when $i_0=i_0^\star$, and $F^W=F^W_{i_0}$ in this case, where $F^W_{i_0}$ is defined in \eqref{FWi}. 
\end{remark}
\begin{remark} 
Note that, for the study of the \sVRJP, it is natural to assume that the graph is strongly connected, in which case the condition ensuring that $F^{W,\theta, \eta}<\infty$ is satisfied. 
However, the induction in Theorem~\eqref{GThm_beta} below involves a restriction of the graph which may no longer be strongly connected, even when the graph $\ggg$ is connected, which explains that we do not restrict the definition to that case.
\end{remark} 
 
\subsection{Restriction and conditioning properties}
Theorem~\ref{Thm_beta} stated above is a special case of a more general theorem, that we state below, which involves also a measure in both the $\beta$ and the $a$ variables. This measure appears naturally in the induction step of the proof of Theorem~\ref{Thm_beta}, it is also related to the \sVRJPse, see Theorems~\ref{mixing-Ia} and~\ref{mixing-Ib} below.

 Fix $I$  self-dual  subset of $V$ (i.e. $I^\star=I$), and let $\sss_I$ and $\aaa_I$ be the corresponding $\star$-symmetric and $\star$-antisymmetric  subspaces of  $\R^I$:
$$
\sss_I=\{x\in \R^I,\;\; x_i=x_{i^\star}\}, \;\;\;  \aaa_I=\{x\in \R^I,\;\; x_i=-x_{i^\star}\}.
$$
We also adopt the following notations: given $y\in \R^V$, denote by $y_I=(y_i)_{i\in I}$ the restriction of $y$ to the indices of $I$, and for a $V\times V$ matrix $M$, let 
$$
M_{I,I}, \;\; M_{I,I^c}, \;\; M_{I^c,I}, \;, M_{I^c,I^c},
$$
be the block matrices obtained by the restriction of $M$ to the corresponding subsets.

Given $a_I\in \aaa_I$, $\beta_I\in \sss_I$, $\theta\in (0,\iy)^V$ and $\eta\in \R_+^V$,   define 
 \begin{align}
 \label{theta_A}
  \hat H_\beta&=(H_\beta)_{I,I}=\beta_I -W_{I,I},
  \;\; \hat G_\beta= (\hat H_\beta)^{-1}\\
 \tet^{a}_I&=e^{a^\star_I}\tet_I=(e^{a^\star_i}\tet_i)_{i\in I}\\
 \label{eta_hat}
 \hat\eta_I&=\eta_I+W_{I,I^c}\theta_{I^c}, \,\,\hat\eta^{a}_I=\eta_I+W_{I,I^c}\theta_{I^c}^{a}
\\
\label{W_check}
\check W_{I^c,I^c}&=W_{I^c,I^c}+W_{I^c,I} \hat G_\beta W_{I,I^c},
\\
 \label{eta_check}
 \check \eta_{I^c}&=\eta_{I^c}+W_{I^c,I}\hat G_\beta \eta_{I}
 \end{align}
Note that $\hat \eta_I$ does not depend on $\beta$ but on $\eta_{I}$ and $\tet_{I^c}$, while $\check W_{I^c,I^c}$, $\check\eta_{I^c}$, depends on $\beta_I$ but not on $\beta_{I^c}$. Note also that $\check\eta_{I^c}=0$ if $\eta=0$. The following lemma plays a key role in the proof of Theorems~\ref{Thm_beta} and~\ref{GThm_beta} below.
\begin{lemma}
\label{lem_identities}
For all non-empty subset $I\subsetneq V$, $I^\star=I$,  for $\theta\in (0,\iy)^V$, $\eta\in \R_+^V$, and with the notations above, we have the following equality of measures:
$$\nu_\sss^{W,\theta,\eta}(d\beta)= \nu_{\sss_I}^{W_{I,I},\theta_I,\hat \eta_I}(d\beta_I)  \nu_{\sss_{I^c}}^{\check W_{I^c,I^c},\theta_{I^c},\check \eta_{I^c}}(d\beta_{I^c});
$$
$$
\nu_\aaa^{W,\theta,\eta}(da)= \nu_{\aaa_I}^{W_{I,I},\theta_I,\hat \eta_I^a}(da_I)  \nu_{\aaa_{I^c}}^{W_{I^c,I^c},\theta_{I^c},\hat \eta_{I^c}}(da_{I^c});
$$
$$
\nu_{\sss_I}^{W_{I,I},\theta_I,\hat \eta^a_I}(d\beta_I) \nu_{\aaa_{I^c}}^{W_{I^c,I^c},\theta_{I^c},\hat \eta_{I^c}}(da_{I^c})
=
\nu_{\sss_I}^{W_{I,I},\theta_I,\hat \eta_I}(d\beta_I) \nu_{\aaa_{I^c}}^{\check W_{I^c,I^c},\theta_{I^c},\check \eta_{I^c}}(da_{I^c}).
$$
We denote by $\mathcal{Q}_I^{W,\tet,\eta}(d\beta_{I},da_{I^c})$  the measure on $\sss_I\times\aaa_I$ given by the equivalent expressions above.
\end{lemma}
The theorem below generalizes Theorem~\ref{Thm_beta} above.
\begin{theorem}\label{GThm_beta}

For all $I\subset V$ such that $I^\star=I$,  $\theta\in (0,\iy)^V$ and $\eta\in \R_+^V$, 
 \begin{align}
 \label{integrale-partielle}
  \int_{\sss}  \nu_\sss^{W,\theta,\eta}(d\beta)
=
\int_{\sss_I\times\aaa_{I^c}}
 \mathcal{Q}_I^{W,\tet,\eta}(d\beta_I,da_{I^c})
 = \int_{\aaa}  \nu_\aaa^{W,\theta,\eta}(da)
\end{align}
If $F^{W,\theta,\eta}<\infty$, we denote by $\overline \nu_\sss^{W,\theta,\eta}$,  $\overline \nu_\aaa^{W,\theta,\eta}$ and $\overline{\mathcal{Q}}_I^{W,\theta,\eta}$ the probability measures ${1\over F^{W,\theta,\eta}} \nu_\sss^{W,\theta,\eta}$,  ${1\over F^{W,\theta,\eta}} \nu_\aaa^{W,\theta,\eta}$ and ${1\over F^{W,\theta,\eta}} \mathcal{Q}_I^{W,\tet,\eta}$.
\end{theorem}

Finally, we prove that the laws introduced in Theorems~\ref{Thm_beta} and~\ref{GThm_beta} are stable by conditioning by the values on a subset. This gives a counterpart to the result proved independently in Lemma~5 of \cite{SZ19} and \cite{LW17} Section~4.  
\begin{proposition}\label{conditionning}
Let $I\subsetneq V$ be such that $I^\star=I$,  $\theta\in (0,\iy)^V$ and $\eta\in \R_+^V$. 
Assume that $F^{W,\theta,\eta}<\infty$.

\noindent i) {\it (Conditioning)} We have the equalities of conditioned probabilities:
\begin{itemize}
\item
under $\overline\nu_\sss^{W,\theta,\eta}(d\beta)$, conditioned on $\beta_I$, $\beta_{I^c}$ has law $ \overline \nu_{\sss_{I^c}}^{\check W_{I^c,I^c},\theta_{I^c},\check \eta_{I^c}}$;
\item
under $\overline\nu_\aaa^{W,\theta,\eta}(da)$, conditioned on $a_{I^c}$, $a_{I}$ has law $\overline \nu_{\aaa_{I}}^{W_{I,I},\theta_{I},\hat \eta^a_{I}}$;

\item
under $\overline{\mathcal{Q}}_I^{W,\tet,\eta}(d\beta_I,da_{I^c})$: conditioned on $a_{I^c}$, $\beta_I$ has law $\overline \nu_{\sss_I}^{W_{I,I},\theta_I,\hat \eta^a_I}$ and conditioned on $\beta_I$, $a_{I^c}$ has law $\overline \nu_{\aaa_{I^c}}^{\check W_{I^c,I^c},\theta_{I^c},\check \eta_{I^c}}$.
\end{itemize}

\noindent ii)
{\it (Restriction)} $\beta_I=(\beta_i)_{i\in I}$ has the same law under $\overline\nu_\sss^{W,\theta,\eta}(d\beta)$ and $\overline{\mathcal{Q}}_I^{W,\tet,\eta}(d\beta_I,da_{I^c})$, while $a_{I^c}=(a_i)_{i\in I^c}$ has the same law under $\overline\nu_\aaa^{W,\theta,\eta}(da)$ and $\overline{\mathcal{Q}}_I^{W,\tet,\eta}(d\beta_I,da_{I^c})$.

\end{proposition}
\begin{remark} Implicitly, Theorem \ref{conditionning} implies that if $F^{W,\theta,\eta}<\infty$ for the full graph, then it is also the case for all the integrals that appear on the restrictions to $I$ and $I^c$ with the corresponding parameters $W,\theta,\eta$, for instance in the first statement it implies that $F^{\check W_{I^c,I^c},\theta_{I^c},\check \eta_{I^c}}<\infty$.
\end{remark}
\begin{remark}
Let us point out that Theorem \ref{conditionning} yields a generalization in the case of the \sVRJP of the restriction and conditioning identities stated in Lemma~5 of \cite{SZ19} and \cite{LW17} Section~4, corresponding to the case where $\star$ is the identity: (i) coincides with  Lemma~5 (ii) in \cite{SZ19}, whereas (ii) coincides with  Lemma~5 (i) in \cite{SZ19}: indeed, in this case $\overline{\mathcal{Q}}_I^{W,\tet,\eta}(d\beta_I,da_{I^c})$ is just $\overline\nu_{\sss_I}^{W_{I,I},\theta_I,\hat \eta_I}(d\beta_I)$ since $\aaa_{I^c}=\{0\}$
 and $\nu_{\aaa_{I^c}}^{\check W_{I^c,I^c},\theta_{I^c},\check \eta_{I^c}}=1$. 
\end{remark}

\section{Relation with the \sVRJP}\label{sec_VRJP}
\subsection{Results in the simpler case where $i_0=i_0^\star$}
In the case where the starting point of the \sVRJP is self-dual, i.e. $i_0=i_0^\star$, the relation between the $\beta$-potential and the \sVRJPse, and its mixing measure $\mu_{i_0}^W$, is easier to state.
In the general case where $i_0\neq i_0^\star$, the formulas are slightly more involved and one needs to tilt the measures to take into account the fact the \sVRJP starts from $i_0$ or $i_0^\star$. Hence, for clarity, we start by stating the results in the case of a self-dual starting point and state the general case in the next section.

For $\theta=1$ and $\eta=0$, the measure $\mathcal{Q}_{I}^{W}$ defined in Lemma~\ref{lem_identities} appears as the joint law of the asymptotic jump rates of the randomized \sVRJP on $I$ and the initial random local time $A$ on $I^c$.
\begin{theorem}\label{mixing-Ia}
Assume that $\ggg$ is finite and strongly connected. Let $I\subsetneq V$, $I^\star=I$. Assume that $i_0=i_0^\star$ and $i_0\in I^c$.

i) Under the law of the randomized \sVRJP starting at $i_0$, i.e. under $\overline\P_{i_0}^{W}$, denote by $(U_i)_{i\in V}$ the limiting local time defined in Theorem~\ref{main}~i) and by
$$
\mathcal{B}_i=\sum_{j, i\to j} W_{i,j}e^{U_{j^\star}-U_{i^\star}}, \;\;\; i\in I.
$$
Then, $(\mathcal{B}_I,A_{I^c})$ has the law  ${1\over F^W}\mathcal{Q}_{I}^{W}$.

ii) Under $\overline\P_{i_0}^{W}$, conditionally on $(\mathcal{B}_I,A_{I^c})$, the \sVRJP $X(t)$ has the law of the jump process with jump rates at time $t$, from $i$ to $j$ given by
$$
W_{i,j}e^{T_{i}(t)+T_{i^\star}(t)} e^{V_{j^\star}(t)-V_{i^\star}(t)},
$$
where $V(t)=(V_i(t))_{i\in V}$ is defined by
$$
\begin{cases}
V_i(t)=T_i(t)+A_i, &\hbox {if $i\in I^c$},
\\
(H_{\mathcal{B}}(e^{V^\star(t)}))_{I}=0. \;\;\; 
\end{cases}
$$
\end{theorem}
\noindent
{\bf N.B.:} Note that the equation $(H_{\mathcal{B}}(e^{V^\star(t)}))_{I}=0$ only involves $\mathcal{B}_I$ and not $\mathcal{B}$ on $I^c$.
\begin{remark}
Let us clarify the meaning of the statement ii): 
\begin{itemize}
\item
When the process $X(t)$ is in $I$, then $V(t)$ is constant, as $\beta$ is fixed and the boundary value of $V(t)$ on $I^c$ is constant. Hence, while $X(t)$ belongs to $I$, the process jumps at rate $W_{i,j}e^{T_{i}(t)+T_{i^\star}(t)} e^{v_{j^\star}-v_{i^\star}}$ for a constant function $v$ and, after the change of time given in Theorem~\ref{main}~ii), 
the process jumps with constant rate $W_{i,j}e^{v_{j^\star}-v_{i^\star}}$. In other words, the process behaves as a time-changed Markov jump process during its excursions on the subset $I$, but the transition probabilities may change between two different excursions.
\item
On the contrary, when $i\in I^c$, and $j\in I^c$, then
$$
W_{i,j}e^{T_{i}(t)+T_{i^\star}(t)} e^{V_{j^\star}(t)-V_{i^\star}(t)}= W_{i,j}e^{T_{i}(t)+A_i+T_{j^\star}(t)+A_{j^\star}},
$$
hence the jump rate is the jump rate of the \sVRJP with initial local time $A$.
\end{itemize}
In summary, Theorem~\ref{mixing-Ia}~ii) gives a representation of the randomized \sVRJP by a self-interacting process on $I^c$, and a mixture of Markov jump processes during the excursions on $I$. 
\end{remark}
\begin{remark}
Note that Theorem \ref{mixing-Ia} is also new in the case of the standard VRJP: then $A=0$, and the statement gives the law of the VRJP conditioned on the asymptotic jump rate $\beta$ restricted to a subset $I\subset V$.
\end{remark}
\begin{remark}\label{rk:identity_0_b}
The law of the asymptotic jump rates $\mathcal{B}_i=\sum_{j, i\to j} W_{i,j}e^{U_{j^\star}-U_{i^\star}}$ on the full set of vertices $V$ does not belong to the family of laws $\overline \nu_\sss^{W,\theta,\eta}$, since it is biased by the starting point of the process. In fact, we always have $\det(H_{\mathcal{B}})=0$ while $H_\beta>0$ a.s. under any of the laws $\overline \nu_\sss^{W,\theta,\eta}(d\beta)$. Note that the same happens for the standard VRJP. However, the expression of the law of $\mathcal{B}$ on $V$ appears naturally in the course of the proof of Corollary~\ref{thm:identity_0_b}, see Remark~\ref{rk-nu_I_i0} below.
\end{remark}

Theorem \ref{mixing-Ia} enables us to retrieve the mixing measure of the VRJP, more precisely it gives a different proof of Theorem~\ref{main}.  Indeed, consider the case $I^c=\{i_0\}$, and recall that we have fixed $i_0=i_0^\star$ in this section. Then $A_{i_0}=0$ and 
$$
V_{j^\star}(t)-V_{i^\star}(t)=u_{j^\star}-u_{i^\star}, \;\; \forall t\ge 0,
$$
where $u$ is the solution of $u_{i_0}=0$ and $H_\bbb(e^u)_{I}=0$, i.e.
$$
e^{u_j}={G_\bbb(i_0,j)\over G_\bbb(i_0,i_0)}, \;\;\; \forall j\in V.
$$
Remark that this only involves the values of $\bbb$ on $I$, and in particular not $\bbb_{i_0}$.
It leads to the following corollary, which is a generalization to the case of the \sVRJP of Theorem~3~i) of \cite{STZ17}, and which also gives a different proof of Theorem~\ref{main}. 
\begin{corollary}
\label{thm:identity_0_a}
Assume that $i_0=i_0^\star$ and that the graph is strongly connected. Let $\theta=1$ and $\eta=0$. Then $\int_\aaa \nu_{\aaa}^{W}(da)=F^W=F^W_{i_0}$ and, under the distribution ${1\over F^W} \nu_{\sss}^{W}(d\beta)$,
$$
\left({G_\beta(i_0,j)\over G_{\beta}(i_0,i_0)}\right)_{j\in V}
$$
is distributed as $(e^{{u_j}-u_{i_0}})_{j\in V}$ under the mixing law ${1\over F^W} \mu_{i_0}^{W}(du)$. In particular, under $\overline \P_{i_0}^W$, the randomized \sVRJP in exchangeable time scale $(Z_t)_{t\ge 0}$ (defined in Theorem~\ref{main}~ii) ) is a mixture of Markov jump processes with jump rates from $i$ to $j$
$$
W_{i,j} {G_\beta(i_0,j^\star)\over G_{\beta}(i_0,i^\star)},
$$
where $\beta$ is distributed according to ${1\over F^W} \nu_{\sss}^{W}(d\beta)$.
\end{corollary}


\subsection{Results in the the general case $i_0\neq i_0^\star$}

As explained above, we need to introduce a tilted version of the measures 
$\qqq_I^{W}(d\beta)$ 
in order to take into account the fact that the \sVRJP starts from the vertex $i_0$ rather than from $i_0^\star$. 
\begin{definition-proposition}\label{def_nu_i0}
i) For $i_0\in V$, with the notations of Section~\ref{sec_beta_pot},
we define
$$
\nu_{\aaa,i_0}^{W,\theta,\eta}(da)= \theta_{i_0}^{a}\nu_{\aaa}^{W,\theta,\eta}(da)= \theta_{i_0}e^{a_{i_0^\star}}\nu_{\aaa}^{W,\theta,\eta}(da).
$$
Remark that when $\theta_i=1$ for all $i\in V$ and $\eta=0$, it coincides with Definition \eqref{nu_a}.


ii) Let $I\subsetneq V$, $I=I^\star$, and assume that $i_0\in I^c$. 
Define 
$$
\qqq_{I,i_0}^{W,\theta,\eta}(d\beta_I,da_{I^c})=
  \theta_{i_0}^{a} \qqq_{I}^{W,\theta,\eta}(d\beta_I,da_{I^c}).
$$

Then
 \begin{align}
 \label{integrale-partielle-i0}
\int_{\sss_I\times\aaa_{I^c}}
 \mathcal{Q}_{I,i_0}^{W,\tet,\eta}(d\beta_I,da_{I^c})
 = \int_{\aaa}  \nu_{\aaa,i_0}^{W,\theta,\eta}(da).
\end{align}
Denote by $F^{W,\theta,\eta}_{i_0}$ the corresponding value. When $\theta=1$, $\eta=0$, we simply write $F^W_{i_0}$ and it corresponds to the integral defined in \eqref{FWi}.
\end{definition-proposition}
\begin{remark} We could also define a tilted version of the measure on $\sss$ when $\eta\neq 0$, with the same integral as \eqref{integrale-partielle-i0}, by $\nu_{\sss,i_0}^{W,\theta,\eta}(d\beta)=  (G_\beta \eta)_{i_0}\nu_{\sss}^{W,\theta,\eta}(d\beta)$. However, we do not introduce it formally since we do not really need it here.
\end{remark}
\begin{theorem}\label{mixing-Ib}
Let $\ggg$ be finite and strongly connected. Let $I\subsetneq V$, $I^\star=I$ and $i_0\in I^c$. Then the statement of Theorem~\ref{mixing-Ia}~i) and ii) holds, with a single modification : ${1\over F^W_{i_0}}\mathcal{Q}_{I,i_0}^{W}$ replaces ${1\over F^W}\mathcal{Q}_{I}^{W}$  in i) as the law of $(\mathcal{B}_I,A_{I^c})$.
\end{theorem}

We can deduce from Theorem~\ref{mixing-Ib}  a second proof of the representation of the \sVRJPse, i.e. of Theorem~\ref{main}. 
It is the content of Corollary~\ref{thm:identity_0_b} below. 
Let us fix some notations: in this subsection we take a starting point $i_0\in V$ and set
$$
I=V\setminus\{i_0,i_0^\star\}.
$$
We use notation~\eqref{theta_A} above for this subset $I$, so that for $\beta_I=(\beta_i)_{i\in I}\in \sss_I$, we have
$$
\hat H_\beta =\beta_I -W_{I,I}, \;\;\;\hat G_\beta =(\hat H_\beta)^{-1}.
$$

Theorem~\ref{mixing-Ib}~i) gives us a relation between the law ${1\over F_{i_0}^W}\mu_{i_0}^W$ of Theorem~\ref{main} and the distribution ${1\over F^W_{i_0}}\mathcal{Q}_{I,i_0}^{W}$. The first step is to prove that in the case where $I=V\setminus\{i_0,i_0^\star\}$, there is a bijection between the random variables $(U_i)_{i\in V}$ and the variables $(\bbb_i)_{i\in I}$ which appear in Theorem~\ref{mixing-Ib}, on good domains. With this goal, we introduce some notation. Denote by 
\begin{equation}
\mathcal{D}_{i_0}=\left\{\be_I=(\be_i)_{i\in I}\in \sss_I:\,\hat H_\be>0
\right\}.
\end{equation}
Note that the probability measure $\mathcal{Q}_{I,i_0}^{W}(d\beta_I,da_{I^c})$ is supported on $\mathcal{D}_{i_0}\times \aaa_{I^c}$. Finally, let $\Xi_{i_0}:\mathcal{U}_0^W\mapsto \sss_I$ 
be the map given by
\begin{eqnarray}\label{Xi}
(\Xi_{i_0}(u))_i= \sum_{j, i\to j} W_{ij}e^{u_{j^\star}-u_{i^\star}}, \;\;\; \forall i\in I.
\end{eqnarray}
Note that 
$\Xi_{i_0}(u)\in \sss_I$
since $u\in \uuu_0^W$. 
We have the following lemma.
\begin{lemma}
\label{diff}
For all $i_0\in V$, $\Xi_{i_0}$ is a $C^1$-diffeomorphism from $\mathcal{U}_0^W$ onto $\mathcal{D}_{i_0}$.
\end{lemma}
Applying Theorem~\ref{mixing-Ib} to the case  $I=V\setminus\{i_0,i_0^\star\}$ and the previous lemma, we see that under the law $\overline\P_{i_0}^W$ of the randomized \sVRJP starting at $i_0$, the joint law of $\left((U_i)_{i\in V}, (A_{i_0},A_{i_0^\star})\right)$ is the same as the law of $\left(\Xi_{i_0}^{-1}(\bbb_I), A_{I^c}\right)$, where $(\bbb_I,A_{I^c})$ is distributed according to ${1\over F^W_{i_0}}\mathcal{Q}_{I,i_0}^{W}$.
In order to retrieve Theorem~\ref{main}, it is thus enough to prove that $\left(\Xi_{i_0}\right)^{-1}(\bbb_I)$ has law ${1\over F^W_{i_0}}\mu_{i_0}^{W}$, which is the purpose of the corollary below. 
\begin{corollary}\label{thm:identity_0_b}
Assume that the graph $\ggg$ is strongly connected and let $I=V\setminus\{i_0,i_0^\star\}$. If $(\bbb_I, A_{I^c})$ is distributed according to ${1\over F^W_{i_0}}\mathcal{Q}_{I,i_0}^{W}$, then $\Xi_{i_0}^{-1}(\bbb_I)$ is distributed according to ${1\over F^W_{i_0}}\mu_{i_0}^{W}$. 
\end{corollary}
\begin{remark}
Together with Theorem~\ref{mixing-Ib}, Corollary \ref{thm:identity_0_b} gives a different proof of Theorem~\ref{main}, and in particular of the identity~\eqref{key_identity}. Remark that it is also the general form of Corollary~\ref{thm:identity_0_a} since in the case $i_0=i_0^\star$, we easily see that if we define $U$ by $U= \left(\Xi_{i_0}\right)^{-1}(\bbb_I)$, then $e^{U_j-U_{i_0}}={G_\bbb(i_0,j)\over G_\bbb(i_0,i_0)}$, by definition of $G_\beta$.
\end{remark}

\part{Proofs of the Results}
\section{Preliminary results and proof of the statements in Sections~\ref{s_limiting} and \ref{s_randomization}}
\subsection{Proofs of the results of Section~\ref{s_limiting} and preliminary results concerning the limiting manifold}\label{subsec_limiting}
\begin{proof}[Proof of Lemma~\ref{transversal}]
We denote by
$$
DJ_u(v)=\left({\partial\over \partial v_i} J_u(v)\right)_{i\in V}, \;\;\; D^2J_u(v)=\left({\partial^2\over \partial v_i\partial v_j} J_u(v)\right)_{i, j\in V},
$$
the gradient and Hessian of the function $J_u$.
A direct computation gives 
$$
{\partial\over \partial v_i} J_u(v)= -\sum_{j, i\to j} W_{i,j}^u e^{v_j-v_i}+ \sum_{j, j\to i} W^u_{j,i} e^{v_i-v_j}
$$
$$
{\partial^2\over \partial v_i\partial v_j} J_u(v)= \left\{\begin{array}{ll} -W^u_{i,j}e^{v_j-v_i}-W_{j,i}^u e^{v_i-v_j},& \hbox{ $i\neq j$},
\\
\sum_{l} (W^u_{i,l}e^{v_l-v_i}+W_{l,i}^u e^{v_i-v_l}),& \hbox{ $i= j$}.
\end{array}\right.
$$
From this last formula, we easily deduce that $J_u$ is strictly convex, since for all $b\in \hhh_0$,
$$
\left( b, D^2J_u(v) b\right)=
\sum_{i,j} (W^u_{i,j}e^{v_j-v_i} + W^u_{j,i}e^{v_i-v_j})(b_i-b_j)^2.
$$
%
If $a\in \aaa$, denoting as above $h=u+a$ and using that $W^u_{i,j}e^{a_{j^\star}-a_{i^\star}}= W^{h}_{i,j}$, we deduce
\begin{eqnarray}\label{DJa}
(DJ_u)(a^\star)_i= - \dive(W^{h})(i),\;\;\; \forall i\in V.
\end{eqnarray}

Now consider  $J_u$ restricted to $\aaa$ : by the asymptotic behavior of $J_u$ and its convexity, we know that
$J_u$ admits a unique minimizer on $\aaa$. It implies that there is a unique $a\in \aaa$ such that $(DJ_u(a^\star),b)=0$ for all $b\in \aaa$, which is equivalent to $\dive(W^{u-a})=0$, since $\dive(W^{u-a})\in \aaa$. Hence there is a unique $a\in \aaa$ such that $h=u-a\in \uuu_0^W$.
 Note finally that by \eqref{DJa} $a^\star$ is also the minimizer of $J_u$ on $\hhh_0$.
\end{proof}

We will also need a result on the tangent space to the limiting manifold $\lll_0^W$. 
Given $h\in\R^V$, $K^h$ defined in \eqref{Ku} can  be understood as the operator $K^h : \R^V \mapsto \R^V$ defined by
\begin{eqnarray}\label{Gh}
K^h(g)(i)= \sum_{j, \;i\to j} W_{i,j}^h(g(j)-g(i)).
\end{eqnarray}
In other words, $K^h$ is the infinitesimal generator of the Markov Jump Process with jump rates 
$W_{i,j}^h$. We write $^tK^h$ for the transpose of $K^h$. When $h\in \lll_0^W$,  $\im (^tK^h)\subset \hhh_0$ : indeed, using  $\dive(W^h)=0$, 
$$
\sum_{i\in V}  {^t}K^h(g)(i)= \sum_{i\in V} g(i)\dive(W^h)(i)=0.
$$
Hence, when $h\in \lll_0^W$, $^tK^h$ is bijective on $\hhh_0$ : in the sequel we always understand the inverse of $^tK^h$
as the operator $(^tK^h)^{-1}:\hhh_0\mapsto \hhh_0$, with a slight abuse of notation, 
obtained as the inverse of $^tK^h_{| \hhh_0}$.

\begin{lemma}\label{tangent}
The tangent space of $\lll_0^W$ at point $h$, denoted by $T_h^W$, is equal to
$$
T_h^W=\{x\in \hhh_0, \;\; {}^tK^h x\in \sss_0 \}= ({}^tK^h)^{-1}(\sss_0),
$$
where, as explained above, $(^tK^h)^{-1}$ is understood as the inverse of the restriction of $^tK^h$ to $\hhh_0$.
\end{lemma}
\begin{proof}
By definition, $h\in \uuu_0^W$ if and only if $h\in \hhh_0$ and $\dive(W^h)=0$. Differentiating the last identity, we see that an element $dh$ is in $T_h(\lll_0^W)$ if and only if $dh\in \hhh_0$ and for all $i\in V$, using $\dive(W^h)=0$,
\beq
&&\sum_{j, i\to j} W_{i,j}^h(dh_i+dh_{j^\star})- \sum_{j, j\to i} W_{j,i}^h(dh_{i^\star}+dh_j)=0
\\
&\Leftrightarrow&
\sum_{j, i\to j} W_{j^\star,i^\star}^h(dh_{j^\star}-dh_{i^\star})- \sum_{j, j\to i} W_{j,i}^h(dh_j-dh_i)=0
\\
&\Leftrightarrow&
{}^tK^h(dh)(i)={}^tK^h(dh)(i^\star).
\eeq
\end{proof}
Finally, the following simple property will be used several times throughout the paper.
\begin{proposition}\label{Fa}
For $i\in V$, define $\tilde F^W_i$ by
\begin{eqnarray}\label{tilde_F}\;\;\;\;
\tilde F^W_i=e^{-\demi\sum_{(i,j)\in E} W_{i,j}} F^W_i=
{1\over \sqrt{2\pi}^{\vert V_1\vert}}\int_\aaa e^{a_{i^\star}}e^{-\demi\sum_{(i,j)\in E} W_{i,j}e^{a_i+a_{j^\star}}}(\prod_{i\in  V_1} da_i)
\end{eqnarray}
If $\tilde a\in \aaa$ then $\tilde F^{W^{\tilde a}}_{i_0}= e^{\tilde a_{i_0}} \tilde F^W_{i_0}$.
\end{proposition}
\begin{proof}
We have
$$
\tilde F^{W^{\tilde a}}_{i_0}={1\over \sqrt{2\pi}^{\vert V_1\vert}}\int_\aaa e^{a_{i_0^\star}}e^{-\sum_{(i,j)\in \tilde E} W_{i,j}e^{\tilde a_i+\tilde a_{j^\star}} e^{a_i+a_{j^\star}}}(\prod_{i\in  V_1} da_i)
$$
The change of variables $a'=a+\tilde a$ yields the result.
\end{proof}

\subsection{Proof 
of the results of Section~\ref{s_randomization}}\label{ss_Proof_random}
\subsubsection{Proof of Lemma~\ref{finitness}}
In fact, $F_{i_0}^W<\infty$ is true under the weaker assumption that, for any $i\in V$ there is a directed path in $\ggg$ from $i$ to $i^\star$, which is obviously satisfied when the graph is strongly connected.
Set $w:=\inf_{(i,j)\in E} W_{i,j} >0$. For $i\in V$, let $\aaa_i=\{a\in \aaa, \; a_i\ge \max_{j\in V} \vert a_j\vert\}$, we have 
$$
\int_\aaa \nu_{i_0, \aaa}^{W}(da) \le \sum_{i\in V} \int_{\aaa_i} \nu_{i_0,\aaa}^{W}(da).
$$
Fix $i\in V$. Since there exists a directed path from $i$ to $i^\star$, we denote by $\sigma$ a shortest one. 
Since $a_{i^\star}=-a_i$, for $a\in \aaa_i$ there exists $k$ such that $a_{\sigma_k}-a_{\sigma_{k+1}}\ge {2 a_i\over \vert \sigma \vert}$, where $\vert \sigma \vert$ is the length of the path $\sigma$. Besides, we have
$$
\nu_{i_0,\aaa}^{W}(da)\le C e^{-a_{i_0}} \exp\left(-\demi \sum_{i\to j} W_{i,j} e^{a_i-a_j}\right) da\le C e^{-a_{i_0}} \exp\left(-\demi w e^{2a_i\over \vert \sigma\vert}\right) da,
$$
where $C$ is a constant depending only on $W$. Since $\vert a_j\vert\le a_i$ on $\aaa_i$, we have
$$
\int_{\aaa_i} \nu_{\aaa}^{W,1,\eta}(da)\le C \int_0^\infty \vert 2a_i \vert^{\vert V_1\vert} e^{a_i} \exp\left(-\demi w e^{2 a_i\over \vert \sigma\vert}\right) da_i,
$$
since for all $j\notin \{i,i^\star\}$, we integrate $a_j=-a_{j^\star}$ on the interval $[-a_i,a_i]$ if $j\neq j^\star$ or $a_j=0$ if $j=j^\star$. The last expression is finite. This concludes the proof.

\subsubsection{Proof of Lemma~\ref{flow}}
Under $\overline \P_{i_0}^W$, conditionally on $A$, the process $X$ is a \sVRJP with conductances $W^A_{i,j}= W_{i,j} e^{A_i+A_{j^\star}}$. Hence, conditionally on $A$, the probability that the process $X$ at time $t$ has performed $n$ jumps in infinitesimal time intervals $[t_i, t_i+dt_i)$, $0<t_1<\cdots <t_n<t$, following the trajectory
$\sigma_{0}=i_0, \sigma_1, \ldots, \sigma_n=j_0$ is equal to
\begin{eqnarray} \label{proba_traject}
\exp\left(-\int_0^t \sum_{j: X_s\to j} W_{X_s,j}^A e^{T_{X_s}(s)+T_{j^\star}(s)}ds\right)
\left(\prod_{l=1}^{n} W_{\sigma_{l-1}, \sigma_l}^A e^{T_{\sigma_{l-1}}(t_l)+T_{\sigma_l^\star}(t_l)} dt_l\right)
\end{eqnarray}
Indeed, the exponential term accounts for the holding probability since, at time $s$, the process $X$ leaves the position $X_s$ at a rate $\sum_{j: X_s\to j} W_{X_s,j}^A e^{T_{X_s}(s)+T_{j^\star}(s)}$, and the product terms accounts for the probabilities to jump from $\sigma_{l-1}$ to  $\sigma_l$ in time intervals $[t_{l}, t_l+dt_l]$.
Remark now that 
\begin{align*}
&
d\left(\sum_{(i,j)\in E} W_{i,j}^A (e^{T_{i}(t)+T_{j^\star}(t)}-1)\right)
\\
=&\sum_{j: X_t\to j} W_{X_t,j}^A e^{T_{X_t}(t)+T_{j^\star}(t)}dt+\sum_{j, j^\star\to X_t^\star} W_{j^\star,X_t^\star}^A e^{T_{j^\star}(t)+T_{X_t}(t)}dt
\\
=&
\; 2\sum_{j: X_t\to j} W_{X_t,j}^A e^{T_{X_t}(t)+T_{j^\star}(t)}dt.
\end{align*}
This implies that,
$$
\int_0^t \sum_{j, X_s\to j} W_{X_s,j}^A e^{T_{X_s}(s)+T_{j^\star}(s)}ds=
\demi \sum_{(i,j)\in E} W_{i,j}^A (e^{T_{i}(t)+T_{j^\star}(t)}-1).
$$
Since $A\in \aaa$, we have $W^A_{i,j}=W_{i,j}e^{A_i+A_{j^\star}}=W_{i,j}e^{A_{j^\star}-A_{i^\star}} $, and it leads to the following expression for the probability \eqref{proba_traject}
\begin{eqnarray}\label{density-A}
\;\;\;\;\;\;
\left(\prod_{l=1}^{n} W_{\sigma_{l-1}, \sigma_l} e^{{T_{\sigma_{l-1}}(t_l)+T_{\sigma_l^\star}(t_l)}} dt_l\right) e^{A_{j_0^\star}-A_{i_0^\star}}
e^{-\demi\sum_{(i,j)\in E} W_{i,j}e^{A_{j^\star}-A_{i^\star}} (e^{T_{i}(t)+T_{j^\star}(t)}-1)}.
\end{eqnarray}

Hence, if $\phi$ is a test function, 
\begin{eqnarray*}
&&
\overline\E^W_{i_0}\left( \phi(A)\;|\; \sigma(X_u, \; u\le t)\right)
\\
&=&
{\int_{\aaa} \phi(a) e^{a_{X_t^\star}-a_{i_0^\star}} \exp\left(-\demi\sum_{(i,j)\in E} W_{i,j}e^{a_{j^\star}-a_{i^\star}} (e^{T_{i}(t)+T_{j^\star}(t))}-1)\right) \nu_{\aaa, i_0}^W(da)
\over \int_{\aaa} e^{a_{X_t^\star}-a_{i_0^\star}} \exp\left(-\demi\sum_{(i,j)\in E} W_{i,j}e^{a_{j^\star}-a_{i^\star}} (e^{T_{i}(t)+T_{j^\star}(t)}-1)\right) \nu_{\aaa, i_0}^W(da)}
\\
&=&
{\int_{\aaa} \phi(a) e^{a_{X_t^\star}} \exp\left(-\demi\sum_{(i,j)\in E} W^{T(t)}_{i,j}e^{a_{j^\star}-a_{i^\star}} \right) da
\over \int_{\aaa} e^{a_{X_t^\star}} \exp\left(-\demi\sum_{(i,j)\in E} W_{i,j}^{T(t)}e^{a_{j^\star}-a_{i^\star}} \right)da}
\\
&=&
\int_{\aaa} \phi(a) {1\over F^{W^{T(t)}}_{X_t}} \nu^{W^{T(t)}}_{\aaa,X_t} (da).
\end{eqnarray*}
\subsubsection{Proof of Corollary~\ref{cor_jump_rate}}
Under $\overline \P^W_{i_0}$, conditionnally on $A$ and on $\fff^X_t$,  if $X_t= i$, then $X$ jumps from $i$ to $j$, $i\to j$, at a rate $W^{A+T(t)}_{i,j}=W^{T(t)}_{i,j} e^{A_{j^\star}-A_{i^\star}} $. 
By Lemma \ref{flow}, conditionally on the past at time $t$, $A$ has law ${1\over F^{W^{T(t)}}_{X_t}} \nu_{\aaa,X_t}^{W^{T(t)}}$. 

Hence, conditionally on the past at time $t$, $X$ jumps from $i$ to $j$ at rate 
$$
W^{T(t)}_{i,j} \int e^{a_{j^\star}-a_{i^\star}}{1\over F^{W^{T(t)}}_i} \nu^{W^{T(t)}}_{\aaa,i}(da)= W^{T(t)}_{i,j} {F^{W^{T(t)}}_j\over F^{W^{T(t)}}_i}.
$$


\subsubsection{Proof of Lemma~\ref{randomized-mixture}}
We will prove that, if $\alpha$ satisfies the divergence condition \eqref{divergence}, if $(W_e)_{e\in \tilde E}$ are independent gamma random variables with parameters $(\alpha_e)_{e\in \tilde E}$, 
and, 
if, conditionally on $W$, $A$ is distributed according to $\nu_{\aaa,i_0}^W$, then $W^A\stackrel{law}{=} W$. This will conclude the proof, using Lemma \ref{lem_ERRW-VRJP}.

Let $\phi$ be any positive measurable test function, and let $C_\alpha:= \prod_{(i,j)\in \tilde E}  \Gamma(\alpha_{i,j})$.
 Then
\beq
&&\E\left( \phi(W^A)\right)
\\
&=&
{C_\alpha^{-1}}\int_{\R_+^{\tilde E}} \left( \int_{\aaa}  \phi(W^a){1\over F^W_{i_0}} \nu_{\aaa,i_0}^W(da)\right) \prod_{(i,j)\in \tilde E}  W_{i,j}^{\alpha_{i,j}-1} e^{-W_{i,j}}  dW_{i,j}
\\
&=&
C_\alpha^{-1}\int_{\R_+^{\tilde E} \times\aaa} \phi(W^a) \left( \prod_{(i,j)\in \tilde E}  W_{i,j}^{\alpha_{i,j}-1} e^{-W_{i,j}}  dW_{i,j} \right)
\frac{e^{a_{i_0^\star}-\demi\sum_{(i,j)\in E} W_{i,j}e^{a_i+a_{j^\star}}}}{ \tilde F^W_{i_0}}(\prod_{i\in  V_1} da_i),
\eeq
where we recall that ${ \tilde F^W_{i_0}}$ is defined in proposition \ref{Fa}.
Let us perform the change of variables
$$
(W,a)\to (\tilde W=W^a, a);
$$
then
\beq
&&
\prod_{(i,j)\in \tilde E}  W_{i,j}^{\alpha_{i,j}-1} e^{-W_{i,j}}  dW_{i,j}
\\
&=&
\left( \prod_{(i,j)\in \tilde E}  \tilde W_{i,j}^{\alpha_{i,j}-1}  d\tilde W_{i,j}\right)
e^{-\demi \sum_{(i,j)\in E} \alpha_{i,j}(a_i-a_j)} e^{- \sum_{(i,j)\in \tilde E} \tilde W_{i,j}e^{-a_i-a_{j^\star}}} 
\\&=&
\left( \prod_{(i,j)\in \tilde E}  \tilde W_{i,j}^{\alpha_{i,j}-1}  d\tilde W_{i,j}\right)
e^{-\demi \sum_{i\in V} a_i \dive(\alpha)(i)} e^{- \demi \sum_{(i,j)\in E} \tilde W_{i,j}e^{-a_i-a_{j^\star}}} 
\\
&=&
e^{-a_{i_0^\star}} \left( \prod_{(i,j)\in \tilde E}  \tilde W_{i,j}^{\alpha_{i,j}-1}  d\tilde W_{i,j}\right)
e^{- \demi \sum_{(i,j)\in E} \tilde W_{i,j}e^{-a_i-a_{j^\star}}}. 
\eeq
where we use the divergence condition~\eqref{divergence} in the last equality.

Using also that, by proposition \ref{Fa}, $e^{a_{i_0}} \tilde F^W_{i_0}= \tilde F^{W^a}_{i_0}$, we deduce that
\beq
&&\E\left( \phi(W^A)\right)
\\
&=&
C_\alpha^{-1}\int_{\R_+^{\tilde E} \times\aaa} \phi(\tilde W) \left( \prod_{(i,j)\in \tilde E}  \tilde W_{i,j}^{\alpha_{i,j}-1} e^{-\tilde W_{i,j}}  d\tilde W_{i,j} \right)
{ e^{-a_{i_0^\star}}\over \tilde F^{\tilde W}_{i_0}}  e^{-\sum_{(i,j)\in \tilde E} \tilde W_{i,j}e^{-a_i-a_{j^\star}}}(\prod_{i\in  V_1} da_i).
\eeq
Changing to variables $a'=-a$ and integrating on $a'$, we deduce
\beq
\E\left( \phi(W^A)\right)
=
C_\alpha^{-1}\int_{\R_+^{\tilde E}} \phi(\tilde W) \left( \prod_{(i,j)\in \tilde E}  \tilde W_{i,j}^{\alpha_{i,j}-1} e^{-\tilde W_{i,j}}  d\tilde W_{i,j} \right).
\eeq

\subsubsection{Proof of Proposition~\ref{exchangeability}}
From \eqref{density-A}, integrating on the initial random time $A\sim \nu_{i_0,\aaa}^W$, the probability that, under $\overline\P^W_{i_0}$, the randomized \sVRJP $X$ at time $t$ has performed $n$ jumps in infinitesimal time intervals $[t_l, t_l+dt_l)$, $l=1,\cdots, n$ with $0<t_1<\cdots <t_n<t=t_{n+1}$, following the trajectory
$\sigma_{0}=i_0, \sigma_1, \ldots, \sigma_n=j_0$ is equal to
\begin{eqnarray}\label{density-*0}
\left(\prod_{l=1}^{n} W_{\sigma_{l-1}, \sigma_l} e^{{T_{\sigma_{l-1}}(t_l)+T_{\sigma_l^\star}(t_l)}} dt_l\right)
{\tilde F^{W^{T(t)}}_{j_0}\over \tilde  F^W_{i_0}}
\end{eqnarray}
where we remind that $\tilde F^W_i$ is defined in \eqref{tilde_F}.

Remark that 
\beq
&&\prod_{l=1}^{n} W_{\sigma_{l-1}, \sigma_l} e^{{T_{\sigma_{l-1}}(t_l)+T_{\sigma_l^\star}(t_l)}}
\\
&=&
\left(\prod_{l=1}^{n} W_{\sigma_{l-1}, \sigma_l} e^{{T_{\sigma_{l-1}}(t_l)+T_{\sigma_{l-1}^\star}(t_l)}}\right)
\exp\left(\sum_{l=1}^n \left({T_{\sigma_{l}^\star}(t_l)-T_{\sigma_{l-1}^\star}(t_l)}\right) \right),
\eeq
and that
\begin{eqnarray}
\label{cancellation}
\sum_{l=1}^n \left({T_{\sigma_{l}^\star}(t_l)-T_{\sigma_{l-1}^\star}(t_l)}\right)
=-\sum_{l=0}^{n-1} \left({T_{\sigma_{l}^\star}(t_{l+1})-T_{\sigma_{l}^\star}(t_l)}\right)+T_{j_0^\star}(t).
\end{eqnarray}
Next, we observe that:
\begin{itemize}
\item
If $\sigma_l\not\in V_0$, then $T_{\sigma_l^\star}(t_l)=T_{\sigma_l^\star}(t_{l+1})$, since the process is at site $\s_l$ between times $t_l$ and $t_{l+1}$.
\item
If $\sigma_l=\sigma_l^\star=i\in V_0$, then the sum of $T_{\sigma_{l}}(t_{l+1})-T_{\sigma_{l}}(t_l)$ restricted on the indices $l$ such that $\sigma_l=i$ is equal to $T_i(t)$ (indeed, the local time $T_i$ is constant when the walk is not at vertex $i$),
\end{itemize}
which imply that \eqref{cancellation} equals $T_{j_0^\star}(t)-\sum_{i\in V_0} T_i(t)$. 

This implies that the expression \eqref{density-*0} is equal to
 \begin{eqnarray}
\label{density-*}
\exp\left(T_{j_0^\star}(t)-\sum_{i\in V_0} T_i(t)\right)\left(\prod_{l=1}^{n} W_{\sigma_{l-1}, \sigma_l} e^{{T_{\sigma_{l-1}}(t_l)+T_{\sigma_{l-1}^\star}(t_l)}}dt_l\right)
{\tilde F^{W^{T(t)}}_{j_0}\over \tilde F^W_{i_0}}.
\end{eqnarray}
Setting $M_i(t)=\demi(T_i(t)+T_{i^\star}(t))$, and using Proposition~\ref{Fa},  \eqref{density-*} is equal to
\begin{eqnarray}
\label{density-*-M}
{e^{M_{j_0^\star}(t)}\over \prod_{i\in V_0} e^{M_i(t)}}
\left(\prod_{l=1}^{n} W_{\sigma_{l-1}, \sigma_l} e^{{T_{\sigma_{l-1}}(t_l)+T_{\sigma_{l-1}^\star}(t_l)}}dt_l\right)
{\tilde F^{W^{M(t)}}_{j_0}\over \tilde F^W_{i_0}}.
\end{eqnarray}
Changing to time $s=C(t)$, we have
\begin{align}\label{ds_dt}
ds = e^{T_{X_t}(t)+T_{X_t^\star}(t)} dt.
\end{align}
We deduce that
\begin{eqnarray}
\nonumber
e^{T_{i}(t)+T_{i^\star}(t)}-1&=&\int_0^t e^{T_i(u)+T_{i^\star}(u)}(\indic_{X_u=i} +\indic_{X_u=i^\star}) du
\\
\label{formule_chgt_tps}
&=&
\int_0^s (\indic_{Z_v= i}+ \indic_{Z_v=i^\star}) dv
=\ell_i^Z(s)+\ell_{i^\star}^Z(s).
\end{eqnarray}
where $\ell^Z$ is the local time of the process $Z$. Hence, 
\begin{align}\label{M-vs-l}
e^{M_i(t)}=\sqrt{1+ \ell^Z_{i}(s)+\ell^Z_{i^\star}(s)},
\end{align}
Changing to time $s$ in the expression \eqref{density-*-M}, the probability that, under $\overline\P^W_{i_0}$, the time-changed randomized \sVRJP $Z$ at time $s$ has performed $n$ jumps in infinitesimal time intervals $[s_l, s_l+ds_l)$, $l=1,\cdots, n$ with $0<s_1<\cdots <s_n<s$, following the trajectory
$\sigma_{0}=i_0, \sigma_1, \ldots, \sigma_n=j_0$ is equal to
\beq
{\sqrt{1+\ell_{j_0}^Z(s)+\ell_{j_0^\star}^Z(s)}\over \prod_{i\in V_0} \sqrt{1+2\ell_i^{Z}(s)}}
{\tilde F^{\tilde W(\ell^Z(s))}_{j_0}\over \tilde F^W_{i_0}}
\left(\prod_{l=1}^{n} W_{\sigma_{l-1}, \sigma_l} ds_l\right),
\eeq
where $\tilde W_{i,j}(\ell^Z(s))=W_{i,j}\sqrt{1+\ell^Z_{i}(s)+\ell^Z_{i^\star}(s)}\sqrt{1+\ell^Z_{j}(s)+\ell^Z_{j^\star}(s)}$.

Hence, the law of the process $Z$ on a time interval $[0,s]$
depends only on the number of crossings of edges and on the local time at final time $s$.  This implies that 
$Z$ is partially exchangeable in the sense of \cite{Freedman}, see~\cite{Zeng16} Proposition~1. 
\begin{remark}\label{rmk_M_B_1} If one performs a similar computation under the non-randomized law $\P_{i_0}^W$, then the probability of a path depends on the whole  local time $(T_i(t))$, not only on its projection $(M_i(t))$ on $\sss$, contrary to the outcome for the randomized VRJP, under law $\overline \P_{i_0}^W$. Indeed, that local time $(T_i(t))$ cannot be expressed in terms of $(\ell_i^Z(s))$: in fact we can prove that 
\begin{equation}
\label{ilim}
e^{T_i(t)}=\sqrt{1+\ell_{i}^Z(s)+\ell_{i^\star}^Z(s)}e^{B^1_i(s)},
\end{equation} 
see Remark~\ref{rmk_M_B_2} below.
Here $B^1_i(s)$ is the functional defined in \eqref{def_A}, which depends on the trajectory of $Z$ up to time $s$, not just on its final local time.
\end{remark}

\section{Proof of Theorems~\ref{main} and \ref{main_bis}}
\label{proof:main}
In this section, we first prove Theorem~\ref{main} (the randomized case), from which we  deduce Theorem~\ref{main_bis}. 

\subsection{Notation}

We remind that $(X(t))_{t\ge 0}$ represents the canonical process on $\ddd([0,\infty),V)$, and that $\overline\P^{W}_{i_0}$ is the law of the randomized \sVRJP defined in Section~\ref{s_randomization}. Besides, $Z=X\circ C^{-1}$ is the time changed process defined in Section~\ref{s_randomization}. Under $\overline\P^{W}_{i_0}$, $Z$ is the randomized \sVRJP in exchangeable time-scale.

For $u\in \uuu^W_0$, let $P^{W,u}_{i_0}$ be the law on $\ddd([0,\infty),V)$, such that under $P^{W,u}_{i_0}$, $(X(t))$ is the process starting from $i_0$ which, conditioned on the past at time $t$, jumps from $i$ to $j$ at rate
$$
W_{i,j}e^{u_{j^\star}-u_{i^\star}}e^{T_i(t)+T_{i^\star}(t)}.
$$
The following simple lemma shows that the definition of $P^{W,u}_{i_0}$ is consistent with the definition given in Section~\ref{s_non-randomized}. 
\begin{lemma}\label{canonical-time-change}
Under $P^{W,u}_{i_0}$, $(Z_s)_{s\ge0}=(X_{C^{-1}(s)})_{s\ge0}$ is the Markov jump process starting at $i_0$ with jump rates $W_{i,j}e^{u_{j^\star}-u_{i^\star}}$.
\end{lemma}
\begin{proof}
As in the proof of Proposition~\ref{exchangeability}, changing to time $s=C(t)$ we deduce that
$$
ds = e^{T_{X(t)}(t)+T_{X(t)^\star}(t)} dt,
$$
where we note that, in the definition of $C(t)$ in Proposition \ref{exchangeability}, $e^{T_i(t)+T_{i^\star}(t)}$ appears twice if $i\in V_1$. This  implies  
$
W_{i,j}e^{u_{j^\star}-u_{i^\star}}e^{T_i(t)+T_{i^\star}(t)}dt= W_{i,j}e^{u_{j^\star}-u_{i^\star}}ds.
$ 
\end{proof}

Next, define the process $\hat X(t)=(X(t),T(t))$, which is the joint process of position and local time. By Corollary~\ref{cor_jump_rate}, under $\overline\P^{W}_{i_0}$,
$(\hat X(t))$ is a Markov process with generator
\begin{eqnarray}\label{Gen}
\Gen^W f(i,t)={\partial \over \partial t_i}f(i,t) + \sum_{j,i\to j} W^t_{i,j} {F^{W^t}_j\over F^{W^t}_i}(f(j,t)-f(i,t)).
\end{eqnarray}
Denote by $\overline\P^{W}_{i_0,t^0}$ the law of the process $(\hat X(t))$ with generator $\Gen^W$,  starting from initial value $(i_0,t^0)$.
 
Let $\Gen^{W,u}$ be the generator of $(\hat X(t))$ under the law $P^{W,u}_{i_0}$. Then
\begin{eqnarray}\label{Gen-u}
\Gen^{W,u} f(i,t)={\partial \over \partial t_i}f(i,t) + \sum_{j,i\to j} W_{i,j} e^{t_i+t_{i^\star}+u_{j^\star}-u_{i^\star}} (f(j,t)-f(i,t)).
\end{eqnarray}
Also denote by $P^{W,u}_{i_0,t^0}$ the law of the Markov process with generator $\Gen^{W,u}$, starting from initial value $(i_0,t^0)$.

\subsection{Proof of Theorem~\ref{main}~(i)}
\hfill\break

\noindent{\it Step 1: Feynman-Kac identity}


From now on, we fix $\varphi:\uuu^W_0\mapsto [0,\infty)$ a positive bounded measurable function with compact support in $\uuu^W_0$. We denote by $\ccc_0\subset \uuu_0^W$ its support. The main result of the step 1 of the proof is the following key Feynman-Kac identity.
\begin{lemma}\label{Feynman-Kac}
For all $(j,\tau)\in V\times \R_+^V$, 
define the fonction
\begin{eqnarray}\label{Laplace}
\Psi (j,\tau)= \int_{\lll^{W^\tau}_0} \varphi\left(u+\overline \tau\right) {1\over F^{W^\tau}_j} \mu^{W^\tau}_{j}(du),
\end{eqnarray}
where $\overline\tau=(\overline\tau_i)_{i\in V}$, $\overline \tau_i=\tau_i-{1\over \vert V\vert}\sum_{j\in V} \tau_j$.
Then, for any starting point $i_0\in V$,
\begin{eqnarray}\label{FK-prop}
\Psi (i_0,0)= \overline\E_{i_0}^W\left(\Psi(X_t,T(t))\right).
\end{eqnarray}
\end{lemma}
\noindent N.B.: By an elementary computation we have that $u+\overline \tau\in \uuu_0^W$ if $u\in \uuu_0^{W{^\tau}}$ so that $\varphi(u+\overline \tau)$ is well defined in \eqref{Laplace}.

We could prove that lemma by direct computation, and directly verify that ${\mathcal L}^W(\psi)=0$. We use a more constructive approach in Lemma~\ref{lem_sol_inv}, as this also gives an insight on how the mixing measure has to be related to the law of the \sVRJP, in order to satisfy the Feynman-Kac identity. Our approach is also useful in the second part of the proof of Theorem~\ref{main}.


Define, for all $u\in \uuu^W_0$, $R^{W,u}:V\times \R_+^V\to \R_+$ by
$$
R^{W,u}(i_0, \tau)={F^{W^\tau}_{i_0}} e^{-\demi\sum_{i, j: i\to j} W_{i,j} (e^{\tau_i+\tau_{j^\star}}-e^{\tau_i+\tau_{i^\star}} e^{u_{j^\star}-u_{i^\star}})}
{e^{\tau_{i_0^\star}-u_{i_0^\star}}\over \prod_{i\in V_0}e^{{\tau_i}}}.
$$
\begin{lemma}\label{lem_sol_inv}
(i) 
For all $(i_0,\tau)\in V\times \R_+^V$, the Radon-Nykodym derivative of the randomized \sVRJP $\hat X$ under the law $\overline\P^W_{i_0,\tau}$, with respect to the law of the Markov Jump Process $P^{W,u}_{i_0,\tau}$ on time interval $[0,t]$, is given by
$$
\left({d\overline\P^W_{i_0,\tau}\over dP^{W,u}_{i_0,\tau}}\right)_{|[0,t]}={{R^{W,u}(\hat X(t))}\over R^{W,u}(i_0,\tau)}.
$$


\noindent
(ii) Let $\tau\in  \R_+^V$, 
 $i_0, j_0\in V$. For any positive measurable test function $\phi$, we have
\begin{eqnarray*}
&&
\int_{\lll_0^W} \phi\left(u\right) {R^{W,u}(i_0,0)\over R^{W,u}(j_0,\tau) }
{1\over F^W_{i_0}} \mu^W_{i_0}(du)
\\
&=& \int_{\lll_0^{W^\tau}} \phi\left(\left(\tilde u_i+\tau_i-{1\over \vert V\vert} \sum_{j\in V}\tau_j\right)_{i\in V}\right)
{1\over F^{W^\tau}_{j_0}} \mu_{j_0}^{W^\tau}(d \tilde u).
\end{eqnarray*}
\end{lemma}
\begin{proof}[Proof of Lemma~\ref{lem_sol_inv}]
{\it i)} By a direct adaptation of  \eqref{density-*} to the case where the initial local time is $T(0)=\tau$, the probability under $\overline\P^W_{i_0,\tau}$ that the randomized \sVRJP $X$ at time $t$ has performed $n$ jumps at times in $[t_l, t_l+dt_l)$, $l=1,\cdots, n$ with $0<t_1<\cdots <t_n<t$, following the trajectory
$\sigma_{0}=i_0, \sigma_1, \ldots, \sigma_n=j_0$ is equal to
\begin{eqnarray}
\nonumber
&&
\exp\left(T_{X(t)^\star}(t)-\tau_{i_0}-\sum_{i\in V_0} (T_i(t)-\tau_i)\right)\left(\prod_{l=1}^{n} W_{\sigma_{l-1}, \sigma_l} e^{{T_{\sigma_{l-1}}(t_l)+T_{\sigma_{l-1}^\star}(t_l)}}dt_l\right)
{\tilde F^{W^{T(t)}}_{j_0}\over \tilde F^{W^\tau}_{i_0}}
\end{eqnarray}
which can be written, using definition \eqref{tilde_F}, as
\begin{eqnarray} 
\label{density-*2}
&&
\left(\prod_{l=1}^{n} W_{\sigma_{l-1}, \sigma_l} e^{{T_{\sigma_{l-1}}(t_l)+T_{\sigma_{l-1}^\star}(t_l)}}dt_l\right)
\\
\nonumber
&\boldsymbol{\cdot}&
{e^{T_{X(t)^\star}(t)-\tau_{i_0}}
\over \prod_{i\in V_0} e^{T_i(t)-\tau_i}}
e^{-\demi\sum_{i,j: i\to j} W_{i,j}(e^{T_i(t)+T_{j^\star}(t)}-e^{\tau_i+\tau_j^\star})} {F^{W^{T(t)}}_{j_0}\over F^{W^\tau}_{i_0}}.
\end{eqnarray}
On the other hand, the probability that, under $P^{W,u}_{i_0}$, $X$ follows the same path is equal to
\begin{eqnarray}\label{path_Pwu_1}
\exp\left(-\int_0^t \sum_{j,X(s)\to j} W_{X(s),j} e^{u_{j^\star}-u_{X(s)^\star}}e^{T_{X(s)}(s)+T_{X(s)^\star}(s)}ds\right)
\\ \nonumber
\;\;\;\;\;\;
\times\left(\prod_{l=1}^n W_{\sigma_{l-1},\sigma_{l}}e^{T_{\sigma_{l-1}}(t_l)+T_{\sigma_{l-1}^\star}(t_l)}e^{u_{\sigma_{l}^\star}-u_{\sigma_{l-1}^\star}}dt_l\right).
\end{eqnarray}
Indeed, the first expression comes from the fact that under $P^{W,u}_{i_0}$, $X$ jumps from $i$ to $j$, $i\to j$, at a rate $W_{i,j}e^{T_i(t)+T_{i^\star}(t)}e^{u_{j\star}-u_{i\star}}$, conditionally on the past at time $t$, so that the jump rate at time $t$ is  
$$\sum_{j,X(t)\to j} W_{X(t),j} e^{u_{j^\star}-u_{X(t)^\star}}e^{T_{X(t)}(t)+T_{X(t)^\star}(t)}.$$ This explains the integral in the first exponential term. The product comes from the probability to jump in time interval $[t_l, t_l+dt_l]$.
Now note that
\beq
&&
{d\over dt}
\left(\sum_{i,j: i\to j} W_{i,j} e^{u_{j^\star}-u_{i^\star}}e^{T_i(t)+T_{i^\star}(t)}\right)=
\\
&&\sum_{j, X(t)\to j} W_{X(t),j} e^{u_{j^\star}-u_{X(t)^\star}}e^{T_{X(t)}(t)+T_{X(t)^\star}(t)}+
\sum_{j, X(t)^\star\to j^\star} W_{X(t)^\star,j^\star} e^{u_j-u_{X(t)}}e^{T_{X(t)}(t)+T_{X(t)^\star}(t)}.
\eeq
But, since $u\in \uuu_0^W$, we have, for all $i\in V$
$$
\sum_{j, i^\star\to j^\star} W_{i^\star,j^\star} e^{u_j-u_i}=\sum_{j, j\to i} W_{j,i} e^{u_j-u_i}= \sum_{j, i\to j} W_{i,j} e^{u_{j^\star}-u_{i^\star}}
$$
This implies that 
$$
\int_0^t \sum_{j,X(s)\to j} W_{X(s),j} e^{u_{j^\star}-u_{X(s)^\star}}e^{T_{X(s)}(s)+T_{X(s)^\star}(s)}ds
=
\demi\sum_{i,j: i\to j} W_{i,j} e^{u_{j^\star}-u_{i^\star}}(e^{T_i(t)+T_{i^\star}(t)}-e^{\tau_i+\tau_{i^\star}}).
$$
Hence, \eqref{path_Pwu_1} is equal to
\begin{equation}\label{density-Wu}
e^{u_{j_0^\star}-u_{i_0^\star}}
e^{-\demi\sum_{i,j: i\to j} W_{i,j} e^{u_{j^\star}-u_{i^\star}}(e^{T_i(t)+T_{i^\star}(t)}-e^{\tau_i+\tau_{i^\star}})}
\left(\prod_{l=1}^n W_{\sigma_{l-1},\sigma_{l}}e^{T_{\sigma_{l-1}}(t_l)+T_{\sigma_{l-1}^\star}(t_l)}dt_k\right).
\end{equation}
The expression of the Radon-Nykodim derivative follows from taking the ratio of \eqref{density-*2} and \eqref{density-Wu}.

\noindent (ii)
From the definition of $\mu_{i_0}^W$ and $R^{W,u}$, for $i_0,j_0\in V$, $\tau\in \R_+^V$,
\begin{eqnarray*}
&&{R^{W,u}(i_0,0)\over R^{W,u}(j_0,\tau) } {1\over F^W_{i_0}}\mu^W_{i_0}(du)
\\
&=&
{\sqrt{\vert V\vert} \sqrt{2}^{-\vert V_1\vert}\over F^{W^\tau}_{j_0}}
e^{\demi\sum_{i,j: i\to j} W_{i,j} (e^{\tau_i+\tau_{j^\star}}-e^{\tau_i+\tau_{i^\star}} e^{u_{j^\star}-u_{i^\star}})}
{e^{-\tau_{j_0^\star}}\over \prod_{i\in V_0}e^{-{\tau_i}}} 
{e^{u_{j_0^\star}-\sum_{i\in V_0} u_i} \over\sqrt{2\pi}^{\vert V_0\vert-1}}  {\sqrt{D(W^u)}\over \det_{\aaa}(-K^u)}\sigma_{\lll_0^W}(du)
\\
&=&
{\sqrt{\vert V\vert} \sqrt{2}^{-\vert V_1\vert} \over F^{W^\tau}_{j_0}}
e^{-\demi\sum_{i,j: i\to j}  W^\tau_{i,j} (e^{u_{j^\star}-u_{i^\star}-\tau_{j\star}+\tau_{i^\star}}-1)} {e^{u_{j_0^\star}-\tau_{j_0^\star}-\sum_{i\in V_0} (u_i-\tau_i)} \over\sqrt{2\pi}^{\vert V_0\vert-1}}   {\sqrt{D(W^u)}\over \det_{\aaa}(-K^u)}\sigma_{\lll_0^W}(du).
\end{eqnarray*}
Changing from variables $(u)_{i\in V}$ to $(\tilde u)_{i\in V}$, given by
$$
\tilde u_i:= u_i-\tau_i+{1\over \vert V\vert} \sum_{j\in V} \tau_j,\;\;\;\forall i\in V,
$$
 we deduce:
\begin{eqnarray*}
&&e^{-\demi\sum_{i,j: i\to j}  W^\tau_{i,j} (e^{u_{j^\star}-u_{i^\star}-\tau_{j\star}+\tau_{i^\star}}-1)} e^{u_{j_0^\star}-\tau_{j_0^\star}-\sum_{i\in V_0} (u_i-\tau_i)}
\\
&=&e^{-\demi\sum_{i,j: i\to j}  W^\tau_{i,j} (e^{\tilde u_{j^\star}-\tilde u_{i^\star}}-1)} e^{\tilde u_{j_0^\star}-\sum_{i\in V_0} \tilde u_i+{\vert V_0\vert-1\over \vert V\vert}\sum_{i\in V}\tau_i}
\end{eqnarray*}
and
$$
{\sqrt{D(W^u)}}= e^{-{\vert V\vert -1\over \vert V\vert}\sum_{i\in V} \tau_i} {\sqrt{D(W^{\tau+\tilde u})}},
\;\;\;
{\det}_{\aaa}( K^u)=e^{-{2\dim(\aaa)\over \vert V\vert} \sum_{i\in V} \tau_i}{\det}_{\aaa}( -K^{\tau+\tilde u}).
$$
Moreover, $u\in \uuu_0^W$ iff $\tilde u\in \uuu_0^{W^\tau}$ and $\sigma_{\lll_0^W}(du)= \sigma_{\lll_0^{W^\tau}}(d\tilde u)$, see \eqref{volume_bis}. Since $\vert V_0\vert=\vert V\vert -2\dim(\aaa)$, this concludes the proof.
\end{proof}
\begin{proof}[Proof of Lemma~\ref{Feynman-Kac}]
By Lemma~\ref{lem_sol_inv}~(ii) applied to $\phi=\varphi$, we have
\begin{align*}
&\overline\E_{i_0}^W\left(  \int_{\uuu^{W^{T(t)}}_0} \varphi(u+\overline T(t))
{1\over F^{W^{T(t)}}_{X_t}} 
\mu_{X_t}^{W^{T(t)}}(du)\right)
\\
=&
\overline\E_{i_0}^W\left( \int_{\uuu^{W}_0}\varphi(u) {R^{W,u}(i_0,0)\over R^{W,u}(X(t),T(t))}
{1\over F^W_{i_0}} 
\mu_{i_0}^{W}(du)\right)
\\
=&
 \int_{\uuu^{W}_0} \varphi(u) \overline\E_{i_0}^{W}\left( {R^{W,u}(i_0,0)\over R^{W,u}(X(t),T(t))}\right)
 {1\over F^W_{i_0}} \mu_{i_0}^{W}(du)
 \\
=&
 \int_{\uuu^{W}_0} \varphi(u) E_{i_0}^{W,u}\left(1\right)
 {1\over F^W_{i_0}} \mu_{i_0}^{W}(du) 
\\
=&
\Psi(i_0,0),
\end{align*}
where we used Lemma~\ref{lem_sol_inv}~(i) in the penultimate identity.
\end{proof}

\noindent{\it Step 2: Asymptotic Gaussian estimates.}

The strategy is now to prove that $\psi(X_t,T(t))$ converges {a.s.} to $\varphi(U)$ where $U$ is the limit defined in Theorem~\ref{main}~i), and to obtain a good bound on $\psi(X_t,T(t))$ to apply dominated convergence. Remind that, at this stage of the proof, we do not know that ${1\over F_{i_0}^W}\mu_{i_0}^W$ is a probability measure, so that we do not have an obvious bound on $\psi(X_t,T(t))$, even though $\varphi$ is bounded. 

Following the notation in Lemma~\ref{Feynman-Kac}, we set
$$
\overline T(t)=T(t)-t/ \vert V\vert \in \hhh_0 \;\; \hbox{ and } \;\; \overline H(t)=p_{\lll_0^W}(\overline T(t)).
$$
the projection of $\overline T(t)$ on the limiting manifold $\lll_0^W$. Besides we set $H(t)=\overline H(t)+t/\vert V\vert$.
Let 
$A(t):=\overline H(t)-\overline T(t)$ which is in $\aaa$ by Lemma~\ref{transversal}. We also have $A(t)=H(t)-T(t)$. Lemma \ref{Convergence} yields, with the notation of Theorem \ref{main} i), that
$$\lim_{t\to \infty} A+\overline T(t)=U.
$$ 
This implies subsequently, using $U\in \lll_0^W$, that 
$$\lim_{t\to \infty} p_{\lll_0^W}(A+\overline T(t))=\lim_{t\to \infty} \overline H(t)=U,
$$ 
and therefore that $\lim_{t\to \infty}A(t)=A$. 

If $u\in \uuu_0^W$, set
\begin{eqnarray}\label{etaWU}
\eta(W^u)=
{\sqrt{\vert V\vert} \sqrt{2}^{-\vert V_1\vert}\over \sqrt{2\pi}^{\vert V_0\cup V_1\vert -1}}{\sqrt{D(W^u)}\over \det_{\aaa}( K^{W^u} )}.
\end{eqnarray}
With this notation, for all  $(i_0,\tau)\in V\times \R_+^V$, we have, using $\uuu_0^{W^{\tau}}=\uuu_0^{W^{\overline\tau}}$,
\beq
&&\Psi (i_0,\tau)
\\
&=&{1\over F^{W^{\tau}}_{i_0}}\int_{\lll^{W^\otau}_0} \varphi(u+\overline \tau)  e^{u_{i^\star_0}} e^{-\sum_{i\in V_0} u_{i}}e^{-\demi \sum_{(i,j)\in E}W^\tau_{i,j}(e^{u_{j^\star}-u_{i^\star}}-1)}\eta(W^{\tau+u}){d\sigma_{\lll_0^{W^\otau}}({u})}
\\
&=&
{1\over \tilde F^{W^\tau}_{i_0}}\int_{\lll^{W^\otau}_0} \varphi(u+\overline \tau) e^{u_{i^\star_0}} e^{-\sum_{i\in V_0} u_{i}}e^{-\demi \sum_{(i,j)\in E}W^\tau_{i,j}e^{u_{j^\star}-u_{i^\star}}}\eta(W^{\tau+u}){d\sigma_{\lll_0^{W^\otau}}({u})}.
\eeq
Using that
$$
\tilde F^{W^{T(t)}}_{i_0}= e^{A_{i^\star_0}(t)} \tilde F^{W^{H(t)}}_{i_0},
$$
(see Proposition \ref{Fa}), and changing to variable $\tilde u_i= u_i -A_i(t)$, we have $\tilde u\in \uuu_0^{W^{\overline H(t)}}$ and
\begin{align}\label{tilde_psi}
&\;\;\;\;\;\;\; \Psi (i_0,T(t))
\\
\nonumber &={1\over \tilde F^{W^{H(t)}}_{i_0}} 
\int_{\lll^{W^{\overline H(t)}}_0} \varphi(\tilde u+\overline H(t)) e^{\tilde u_{i^\star_0}-\sum_{j\in V_0} \tilde u_{j}}e^{-\demi \sum_{(i,j)\in E}W^{H(t)}_{i,j}e^{\tilde u_{j^\star}-\tilde u_{i^\star}}}\eta(W^{H(t)+\tilde u}){d\sigma_{\lll_0^{W^{\overline H(t)}}}({\tilde u})}
\\
\nonumber &=
{1\over F^{W^{H(t)}}_{i_0}}
\int_{\lll^{W^{\overline H(t)}}_0} \varphi(\tilde u+\overline H(t)) e^{\tilde u_{i^\star_0}-\sum_{j\in V_0} \tilde u_{j}}e^{-\demi \sum_{(i,j)\in E}W^{H(t)}_{i,j}(e^{\tilde u_{j^\star}-\tilde u_{i^\star}}-1)}\eta(W^{H(t)+\tilde u}){d\sigma_{\lll_0^{W^{\overline H(t)}}}({\tilde u})}
\\
\nonumber &:=
{\tilde \Psi(i_0,H(t))\over  F^{W^{H(t)}}_{i_0}},
\end{align}
where, by definition, $\tilde \Psi(i_0,H(t))$ is the integral term in the penultimate expression.

Set $w=\inf_{(i,j)\in E}W_{i,j}$, and 
\begin{align}\label{c_0}
c_0=\max_{u\in \ccc_0} \max_{(i,j)\in E} \vert \nabla u_{i^\star,j^\star}\vert,
\end{align}
where $\nabla u_{i^\star,j^\star}=u_{j^\star}-u_{i^\star}$,
which is finite since $\ccc_0$ is compact.

We now introduce the compact set $\ccc_1$ of all $u\in\uuu_0^W$ such that 
$$ \max_{(i,j)\in E} \vert \nabla u_{i^\star,j^\star}\vert \le c_0+1.$$
Note that the interior of $\ccc_1$ contains $\ccc_0$, and that, for all $u'\in \ccc_1^c$,  
\begin{eqnarray}\label{separation}
\max_{(i,j)\in E} \vert \nabla u'_{i^\star,j^\star}\vert > c_0+1.
\end{eqnarray}

The strategy is to treat separately the case $\overline H(t)\in \ccc_1$, where we use the compactness of $\ccc_1$ to have uniform estimates, and the case $\overline H(t)\in \ccc_1^c$, where we use that $\hbox{supp}(\varphi)=\ccc_0$ to prove a uniform convergence to 0 of  $\psi(X(t),T(t))$.

Recall that $\ppp_{\sss}$ and $\ppp_\aaa$ are the orthogonal projections on subspaces $\sss$ and
$\aaa$, and that $K^u$ was introduced in \eqref{Ku}. For $u\in \lll_0^W$,  $y\in \R^V$, 
$${}^t(K^u)(y^\star)(i^\star)=\sum_{j:j\to i^\star}W_{j,i^\star}^u(y_{j^\star}-y_i)=\sum_{j:i\to j}W_{i,j}^u(y_j-y_i)=K^u(y)(i),$$ hence
$$
\ppp_\sss K^u \ppp_\sss = \ppp_\sss(\demi(K^u +{}^t(K^u)))\ppp_\sss, \;\;\; \ppp_\aaa K^u \ppp_\aaa = \ppp_\aaa(\demi(K^u +{}^t(K^u)))\ppp_\aaa,
$$
which implies that $\ppp_\sss K^u \ppp_\sss$ and $\ppp_\aaa K^u \ppp_\aaa$ are symmetric. For simplicity, we denote as before by ${\det}_{\sss_0}((-K^u)^{-1})$ the determinant of the operator $\ppp_{\sss_0}(-K^u)^{-1}\ppp_{\sss_0}$ restricted to $\sss_0$, and by ${\det}_{\aaa}(-K^u)$ the determinant of the operator $\ppp_{\aaa}(-K^u)\ppp_{\aaa}$ restricted to $\aaa$.
\begin{lemma}\label{Gaussian-asymptotic}
We have, when $t\to \infty$,
$$
\indic_{\overline H(t)\in \ccc_1} \tilde \psi(X(t),H(t)) \sim \indic_{\overline H(t)\in \ccc_1} e^{-{t\over N} \vert V_1\vert } \sqrt{2\pi}^{\vert V_0\cup V_1\vert -1} \varphi(\overline H(t)) \eta(W^{\overline H(t)}) \sqrt{ {\det}_{\sss_0}\left(-(K^{\overline H(t)})^{-1}\right)},
$$
$$
\indic_{\overline H(t)\in \ccc_1}  F^{W^{H(t)}}_i
\sim \indic_{\overline H(t)\in \ccc_1}  \sqrt{2}^{-\vert V_1\vert}e^{-{t\over N}\vert V_1\vert} 
 {\det}_{\aaa}(-K^{\overline H(t)})^{-\demi},
$$
with a uniform control of the same order for each term.
\end{lemma}
\begin{lemma}\label{out_C1}
There exist positive constants $c$ and $c'$ such that
\begin{eqnarray}\label{bound_out_C1}
\indic_{\overline H(t)\in \ccc_1^c} \psi(X(t),T(t))\le c \exp\left( - c' e^{t/\vert V\vert}\right).
\end{eqnarray}
\end{lemma}
\begin{proof}[Proof of Lemma~\ref{Gaussian-asymptotic}]

When $t$ tends to $\infty$, $W^{H(t)}=e^{2t/N}W^{\overline H(t)}$ is equivalent to $e^{2t/N}W^U$. 
The main exponential term in the integrand of $\tilde \Psi_\lambda(i,H(t))$ is
$$
-e^{2t/N} 
\demi \sum_{(i,j)\in E}W^{\overline H(t)}_{i,j}(e^{u_{j^\star}-u_{i^\star}}-1).
$$
The maximum of the last expression is obtained for $u=0$ since $H(t)\in \lll_0^W$, by Lemma \ref{transversal}.
This implies that the first order Taylor expansion cancels out. A second order expansion yields 
$$
 \sum_{(i,j)\in \tilde E}W^{\overline H(t)}_{i,j}(e^{u_{j^\star}-u_{i^\star}}-1)
= \demi \sum_{(i,j)\in \tilde E}W^{\overline H(t)}_{i,j}(u_{j^\star}-u_{i^\star})^2 + o(\| u\|^2).
$$
Besides, the remainder term $o(\| u\|^2)$ is uniform for $\overline H(t)$ in $\ccc_1$ and $u$ such that $\overline H(t)+u\in \ccc_0$. 

We recall from Lemma~\ref{tangent} that $T_0^{W^{\overline H(t)}}=\left({}^tK^{\overline H(t)}\right)^{-1}(\sss_0)$ is the tangent space at $0$ of $\uuu_0^{W^{\overline H(t)}}$.
We denote by $\sigma_{ T_0^{W^{\overline H(t)}}}$ the volume measure induced by $\sigma_{\uuu_0^{W^{\overline H(t)}}}$ so that, by \eqref{volume}, $\sigma_{T_0^{W^{\overline H(t)}}}(B)=\lambda_{\sss_0}(\ppp_{\sss}(B))$. 
For large $t$, the integral in $\tilde \Psi_\lambda(i,H(t))$ concentrates around $u=0$. 

Since we have a uniform quadratic estimate of the exponential term, we can localize the integral in a ball of size $B(0, e^{(1+\epsilon) t/N})$, change to variables $x_i= e^{t/N} \tilde u_i$, and integrate on the tangent plane $T_0^{W^{\overline H(t)}}$ (details are easy and left to the reader).
Using that $\lll_0^W$ has dimension $\vert V_0\vert +\vert V_1\vert -1$ and that $\eta$ is $(\vert V_0\vert -1)/2$ homogeneous,
 $\indic_{\overline H(t)\in \ccc_1} \tilde \Psi (i,H(t))$ is equivalent to
\begin{eqnarray*}
 && e^{-{t\over N} \vert V_1\vert }  \indic_{\overline H(t)\in \ccc_1} \varphi(\overline H(t)) \eta(W^{\overline H(t)}) \int_{T_0^{W^{\overline H(t)}}} e^{-\frac{1} {4} \sum_i \sum_{j, i\to j} W^{\overline H(t)}_{i,j} (x_{i^\star}-x_{j^\star})^2}
d\sigma_{ T_0^{W^{\overline H(t)}}}(x)
 \\
 &=&  e^{-{t\over N} \vert V_1\vert }  \indic_{\overline H(t)\in \ccc_1} \varphi(\overline H(t)) \eta(W^{\overline H(t)}) \int_{T_0^{W^{\overline H(t)}}} e^{-\frac{1} {4}  \sum_i \sum_{j, j\to i} W^{\overline H(t)}_{j,i} (x_{i}-x_{j})^2}
d\sigma_{ T_0^{W^{\overline H(t)}}}(x)
 \\
 &=&
 e^{-{t\over N} \vert V_1\vert }  \indic_{\overline H(t)\in \ccc_1} \varphi(\overline H(t)) \eta(W^{\overline H(t)})
 \int_{T_0^{W^{\overline H(t)}}} e^{\demi ({}^tK^{\overline H(t)} x, x)}
d\sigma_{ T_0^{W^{\overline H(t)}}}(x).
 \eeq
 By Lemma \ref{tangent}, we use the change of variables $y={}^t(K^{\overline H(t)})(x)\in \sss_0$, which yields that the previous expression is equal to
  \begin{align}
  \nonumber
 &e^{-{t\over N} \vert V_1\vert }  \indic_{\overline H(t)\in \ccc_1} \varphi(\overline H(t)) \eta(W^{\overline H(t)}) {\det}_{\sss_0}\left( -(K^{\overline H(t)})^{-1}\right) \int_{\sss_0} e^{\demi\left((K^{\overline H(t)})^{-1} y, y\right)}
d\lambda_{\sss_0}\left(dy \right).
 \\
 \label{estimate-tilde-psi}
 \;\;\;\;\,\,\,\,&= e^{-{t\over N} \vert V_1\vert } \indic_{\overline H(t)\in \ccc_1} \sqrt{2\pi}^{\vert V_0\cup V_1\vert -1} \varphi(\overline H(t)) {\eta(W^{\overline H(t)})} {\det}_{\sss_0}\left(-(K^{\overline H(t)})^{-1}\right)^{\demi}
 \end{align}
 since $\overline H(t)\to U$ and since $\varphi$ is supported on $\ccc_0\subset \ccc_1$.
 
We can apply a very similar reasoning to $F^{W^{H(t)}}_i$: changing  variables $a$ to $\tilde a = e^{t/N} a$, we deduce
\begin{align}
\nonumber
\indic_{\overline H(t)\in \ccc_1} F^{W^{H(t)}}_i&\sim &
{1\over \sqrt{2\pi}^{\vert V_1\vert}} e^{-{t\over N}\vert V_1\vert} \indic_{\overline H(t)\in \ccc_1} \int_{\aaa} e^{-\demi\sum_i \sum_{j} W^{\overline H(t)}_{i,j} (\tilde a_{j}-\tilde a_i)^2}  d\tilde a 
\\
\nonumber
&=&
{1\over \sqrt{2\pi}^{\vert V_1\vert}}e^{-{t\over N}\vert V_1\vert} \indic_{\overline H(t)\in \ccc_1}
\int_{\aaa} e^{\demi \left(\tilde a^\star,K^{\overline H(t)} \tilde a\right)} d\tilde a
\end{align}
where as before $(\cdot,\cdot)$ is the usual scalar product on $\R^V$.  Let us denote by $\lambda_\aaa$ the Euclidian volume measure on $\aaa$. then we have $d\lambda_\aaa=\sqrt{2}^{\vert V_1\vert} \prod_{i\in V_1} da_i$, hence 
\begin{eqnarray}
\nonumber
\indic_{\overline H(t)\in \ccc_1} F^{W^{H(t)}}_i
&\sim&{\sqrt{2}^{-\vert V_1\vert}\over \sqrt{2\pi}^{\vert V_1\vert}}e^{-{t\over N}\vert V_1\vert} \indic_{\overline H(t)\in \ccc_1}
\int_{\aaa} e^{\demi \left(\tilde a^\star,K^{\overline H(t)} \tilde a\right)} d\lambda_{\aaa}(a)
\\
\label{estimate-F}
&=&
\sqrt{2}^{-\vert V_1\vert}e^{-{t\over N}\vert V_1\vert} 
 \indic_{\overline H(t)\in \ccc_1} {\det}_{\aaa}(-K^{\overline H(t)})^{-\demi}.
 \end{eqnarray}
\end{proof}
\begin{proof}[Proof of Lemma~\ref{out_C1}.]

In this proof, $c$, $c'$, $c''$ are positive constants (depending only on the parameters of the model and on $\varphi$), whose values can change from line to line. We first show an exponential upper bound for $\indic_{\overline H(t)\in \ccc_1^c} \tilde \psi(i,H(t))$ valid for any $i\in V$, hence in particular for $X_t$. 

We start by observing that, $\varphi$ having compact support $\ccc_0$, we have a uniform bound 
$$
\indic_{u+\overline H(t)\in \ccc_0} \varphi(u+\overline H(t))\eta(u+\overline H(t))\le c \indic_{u+\overline H(t)\in \ccc_0}.
$$
Besides, in $\tilde \psi(i,H(t))$, we integrate on the compact set $\ccc_0-\overline H(t)$, which has measure $\lambda_{\sss_0}(\ppp_{\sss}(\ccc_0-\overline H(t)))= \lambda_{\sss_0}(\ppp_{\sss}(\ccc_0))$, see \eqref{volume}. Hence 
\begin{align*}
\tilde \psi(i,H(t))&\le 
c 
\max_{u\in \uuu_0^W, \, u+\overline H(t)\in \ccc_0} e^{ u_{i^\star_0}-\sum_{j\in V_0} u_{j}}e^{-\demi e^{2t/\vert V\vert}\sum_{(i,j)\in E}W^{\overline H(t)}_{i,j}(e^{ u_{j^\star}- u_{i^\star}}-1)}
\\
&\le c 
e^{(\vert V_0\vert +1)\max_{i\in V} \vert \overline H_i(t)\vert} \max_{u\in \uuu_0^W, \, u+\overline H(t)\in \ccc_0} e^{-\demi e^{2t/\vert V\vert}\sum_{(i,j)\in E}W^{\overline H(t)}_{i,j}(e^{ u_{j^\star}- u_{i^\star}}-1)}.
\end{align*}
Also note that 
\begin{eqnarray}\label{bound_H}
\max_{i\in V} \vert \overline H_i(t)\vert\le c't+c.
\end{eqnarray}
Indeed, $T_i(t)\le t$ implies $\vert \overline T_i(t)\vert \le t$. Also, as $\overline H(t)=\overline T(t) + A(t)$ and $A(t)\in \aaa$, we have $\vert \overline H_i(t)+\overline H_{i^\star}(t)\vert=\vert \overline T_i(t)+\overline T_{i^\star}(t)\vert\le 2t$. Besides,  Lemma~\ref{transversal} implies 
$$
\sum_{i,j\in E} W_{i,j}e^{\overline H_i(t)+\overline H_{j^\star}(t)}\le \sum_{i,j\in E} W_{i,j}e^{\overline T_i(t)+\overline T_{j^\star}(t)},
$$
thus, for $(i,j)\in E$,  $\overline H_i(t)+\overline H_{j^\star}(t)\le 2 t+c$ for $t$ large enough, which implies subsequently that for $(i,j)\in E$, $\overline H_{j^\star}(t)-\overline H_{i^\star}(t)=\overline H_i(t)+\overline H_{j^\star}(t)-(\overline H_i(t)+\overline H_{i^\star}(t))\le 4 t+c$. We deduce \eqref{bound_H} since $\sum_{i\in V} \overline H_i(t)=0$, and since the graph $\ggg$ is strongly connected.

Hence
\begin{align}\label{first_bound}
\tilde \psi(i,H(t))
&\le c 
e^{c' t} \max_{u\in \uuu_0^W, \, u+\overline H(t)\in \ccc_0} e^{-\demi e^{2t/\vert V\vert}\sum_{(i,j)\in E}W^{\overline H(t)}_{i,j}(e^{ u_{j^\star}- u_{i^\star}}-1)}.
\end{align}
The aim is now to prove that
\begin{eqnarray}\label{minoration}
\min_{\stackrel{\overline H \in \ccc_1^c}
{u\in \uuu_0^W, u+\overline H \in \ccc_0} }
\left( \sum_{(i,j)\in E} W^{\overline H}_{i,j}(e^{ u_{j^\star}- u_{i^\star}}-1)\right) >0.
\end{eqnarray}

 We now fix $\overline H \in \ccc_1^c$ and $u\in \uuu_0^W$ such that $u+\overline H \in \ccc_0$.
Since $\overline H\in \uuu_0^W$, we have
\begin{eqnarray*}
\sum_{(i,j)\in E}W^{\overline H}_{i,j}(e^{u_{j^\star}-u_{i^\star}}-1)&=& 
\sum_{(i,j)\in E}W^{\overline H}_{i,j}(e^{\nabla u_{i^\star,j^\star}}- \nabla u_{i^\star,j^\star} -1).
\end{eqnarray*}
Remark that $e^s-s-1$ is positive convex and minimal at $s=0$, and $e^s-s-1\ge \varepsilon_0$ for $\vert s\vert \ge 1$ and that $e^s-s-1\ge \varepsilon_1 e^{s}$ for $s \ge 1$, for some constants $\varepsilon_0>0$ and $\varepsilon_1>0$. 

Consider the set 
$$
E'=\{ (i,j)\in E, \hbox { such that } 
\vert \nabla \overline H_{i^\star,j^\star} \vert \le c_0 + 1 \hbox { and } \vert \nabla \overline H_{j,i}\vert \le c_0 + 1
\}.
$$
Remark that $E'$ can be considered as a subset of $\tilde E$ since $(i,j)\in E'$ implies $(j^\star,i^\star)\in E'$.  By \eqref{separation}, since $\overline{H}\in\ccc_1^c$, we know that $E'$ is strictly contained in $E$.

Using $\sum_{i\in V} \overline H_i=0$, we can find $i_0\in V$ such that $\overline H_{i_0}+\overline H_{i^\star_0}\ge 0$. 
Consider now the set $V'$ of vertices which can be reached from $i_0$ by a directed path in $E'$. If $j_0\in V'$ then $H_{j_0}+H_{j_0^\star}\ge -(c_0+1) 2 \vert E\vert$.  

Since $E'$ is strictly included in $E$ there exist $j_0\in V'$ and $j_1\in V$ such that $(j_0,j_1)\in E\setminus E'$. If $\vert \nabla \overline H_{j_0^\star,j_1^\star}\vert > c_0 + 1$, then $\vert \nabla u_{j_0^\star,j_1^\star}\vert > 1$, using $u+\overline{H}\in\ccc_0$ and \eqref{c_0}. 

We consider two cases: if $H_{j_0}+H_{j_1^\star}\ge -(c_0+1)(2\vert E\vert +1)$, then
$$
W_{j_0,j_1}^{\overline H}(e^{\nabla u_{j_0^\star,j_1^\star}}- \nabla u_{j_0^\star,j_1^\star} -1)\ge \varepsilon_0 w e^{-(c_0+1)(2\vert E\vert +1)},
$$
while if $H_{j_0}+H_{j_1^\star} < -(c_0+1)(2\vert E\vert +1)$ then $\nabla \overline H_{j_0^\star,j_1^\star}\le -(c_0+1)$ since, which implies subsequently that $ \nabla u_{j_0^\star,j_1^\star}\ge 1$ since $\vert \nabla (u+\overline{H})_{j_0^\star,j_1^\star}\vert \le c_0$, using $u+\overline{H}\in\ccc_0$. This implies that
$$
W_{j_0,j_1}^{\overline H}(e^{\nabla u_{j_0^\star,j_1^\star}}- \nabla u_{j_0^\star,j_1^\star} -1)\ge \varepsilon_1 W_{j_0,j_1} e^{\overline H_{j_0}+\overline H_{j_0^\star}} e^{\nabla(u+\overline H)_{j_0^\star,j_1^\star}} \ge \varepsilon_1 w e^{-(c_0+1)2\vert E\vert -c_0}.
$$

On the other hand, if $\vert \nabla \overline H_{j_0^\star,j_1^\star}\vert \le c_0 + 1$ and $\vert \nabla \overline H_{j_1,j_0}\vert > c_0 + 1$ we have 
$$
e^{H_{j_0}+H_{j_1^\star}}=e^{H_{j_0}+H_{j_0^\star}+\nabla\overline H_{j_0^\star,j_1^\star}}\ge e^{-(c_0+1)} e^{H_{j_0}+H_{j_0^\star}},
$$
hence, considering the term associated with the edge $(j_1^\star,j_0^\star)$,
$$
W_{j_1^\star,j_0^\star}^{\overline H}(e^{\nabla u_{j_1,j_0}}- \nabla u_{j_1,j_0} -1)
\ge \varepsilon_0 W_{j_0,j_1}^{\overline H}\ge \demi W_{j_0,j_1} e^{H_{j_0}+H_{j_0^\star}} e^{-(c_0+1)} \ge\varepsilon_0 w e^{-(c_0+1)(2\vert E\vert +1)}.
$$
This concludes the proof of the estimate \eqref{minoration}.

Combining \eqref{first_bound} and \eqref{minoration}, we deduce that
\begin{align}\label{bound_tilde_psi}
\indic_{\overline H(t)\in \ccc_1^c} \tilde \psi(i,H(t))
&\le c 
e^{c' t}  e^{-c'' e^{2t/\vert V\vert}} .
\end{align}

We now need a lower bound for $F_i^{W^{H(t)}}$. Since $\overline H(t)\in \uuu_0^W$, the exponential term of the integrand can be rewritten as follows: 
$$
-e^{2t/\vert V\vert} 
\demi \sum_{(i,j)\in E}W^{\overline H(t)}_{i,j}(e^{a_{j^\star}-a_{i^\star}}-1)= 
-e^{2t/\vert V\vert} 
\demi \sum_{(i,j)\in E}W^{\overline H(t)}_{i,j}(e^{\nabla a_{i^\star,j^\star}}- \nabla a_{i^\star,j^\star} -1).
$$
By \eqref{bound_H} we know that $e^{t/\vert V\vert}e^{\overline H_i(t)+\overline H_{j^\star}(t)}\le e^{2ct}$. We set $\aaa(t)=\{a\in \aaa, \; a_i\le e^{-ct}\}$ with the same constant $c$. Using that $e^s-s-1\le 2 s^2$ for $s\le 1$, we deduce
$$
F_i^{W^{H(t)}}
\ge c'\int_{\aaa(t)} e^{- \sum_{i,j} W_{i,j}} da = c'\left(2 e^{-ct} \right)^{\vert V_1\vert} e^{-\sum_{i,j} W_{i,j}}.
$$
Combining with \eqref{bound_tilde_psi}  concludes the proof.
\end{proof}
\noindent{\it Step 3: final computations.}

From \eqref{FK-prop} and \eqref{tilde_psi}, and since  $\overline H(t)\to U$, we can apply the dominated convergence theorem to deduce that
\begin{eqnarray*}
\psi(i_0,0)&=&\lim_{t\to \infty} \E_{i_0}^W\left( \psi(X(t),T(t))\right)
\\
&=&\lim_{t\to \infty} \E_{i_0}^W\left( \indic_{\overline H(t)\in \ccc_1} \psi(X(t),T(t))\right)
\\
&=&
\E_{i_0}^W\left( {\sqrt{2}^{\vert V_1\vert}} \varphi(U) {\eta(W^U)} \sqrt{{\det}_{\sss_0}(-(K^U)^{-1}) {\det}_{\aaa}(-K^U)} \right).
\end{eqnarray*}
Indeed,  we have a uniform control in Lemma~\ref{Gaussian-asymptotic}, since the functions involved are continuous in $\overline H(t)$. This implies that $\indic_{\overline H(t)\in \ccc_1} \psi(X(t),T(t))$ is bounded, and  Lemma~\ref{out_C1} yields a control outside $\ccc_1$.

Coming back to the definition of $\eta(W^U)$, see \eqref{etaWU}, we have
$$
{\sqrt{2}^{\vert V_1\vert}}  {\eta(W^U)} \sqrt{{\det}_{\sss_0}(-(K^U)^{-1}) {\det}_{\aaa}(-K^U)}=\sqrt{\vert V\vert} \sqrt{D(W^U)} \sqrt{{\det}_{\sss_0}(-(K^{U})^{-1})\over {\det}_{\aaa}(-K^U)}.
$$
For simplicity, we simply write $K$ for $K^U$ below. Firstly, we have
\begin{align}\label{ratio_det}
{{\det}_{\aaa}(-K)\over {\det}_{\sss_0}(-K^{-1})}= {\det}_{\hhh_0} (-K).
\end{align}
Indeed, if we write $K$ by blocks on $\hhh_0= \aaa\oplus \sss_0$ as
$$
K=\left(\begin{array}{ll}  \ppp_\aaa K\ppp_\aaa &\ppp_\aaa K\ppp_{\sss_0} \\ \ppp_{\sss_0} K\ppp_{\aaa} &\ppp_{\sss_0}K\ppp_{\sss_0}
\end{array}\right).
$$
we note that
$$
\left(\begin{array}{cc}  \ppp_\aaa K\ppp_\aaa &\ppp_\aaa K\ppp_{\sss_0} \\0&\Id_{\sss_0}
\end{array}\right) K^{-1} = \left(\begin{array}{cc}  \Id_\aaa &0\\
\ppp_{\sss_0} K^{-1}\ppp_{\aaa} & \ppp_{\sss_0}K^{-1} \ppp_{\sss_0}
\end{array}\right).
$$
Next, we prove that
\begin{align}\label{N_D_K}
\vert V\vert \D(W^U)= {\det}_{\hhh_0} (-K).
\end{align}
In order to compute the determinant of the operator $K_{|\hhh_0}$, we choose a basis of $\hhh_0$: a convenient one in this context is the basis $(e_i=\delta_i-\delta_{i_0})_{i\neq i_0}$. In this basis, we have
$$
K(e_i)=\sum_{j\neq i_0} (K_{i,j}-K_{i_0,j})e_j.
$$
We deduce
$$
 {\det}_{\hhh_0} (-K)=\det(- \left(K_{i,j}-K_{i_0,j}\right)_{i\neq i_0,j\neq i_0}).
$$
Since the sum of columns of $K$ is null, summing all columns in a column $i_1$, now note that
\begin{align*}
&\det(- \left(K_{i,j}-K_{i_0,j}\right)_{i\neq i_0,j\neq i_0})=\det(- \left(K_{i,j}-K_{i_0,j}\right)_{i\neq i_0\,j\neq i_0,i_1},(|V|K_{i_0,j})_{i\neq i_0\,j=i_1})\\
&=|V|\det(- \left(K_{i,j}-K_{i_0,j}\right)_{i\neq i_0\,j\neq i_0,i_1},(-K_{i_1,j})_{i\neq i_0\,j=i_1})\\
&=\vert V\vert \det(- \left(K_{i,j}\right)_{i\neq i_0,j\neq i_0})=\vert V\vert D(W^U).
\end{align*}
We conclude that $\psi(i_0,0)=\E_{i_0}^W\left(\varphi(U) \right)$, which implies that ${1\over F_{i_0}^W}\mu_{i_0}^W$ is the law of $U$, hence in particular that it must be a probability measure. 


\subsection{Proof of Theorem~\ref{main}~(ii)}

The statement~(ii) is equivalent to the following equality of probabilities,
$$
\int_{\uuu_{0}^W} P_{i_0}^{W,u}(\cdot) {1\over F^W_{i_0}}\mu_{i_0}^W(du)=\overline \P_{i_0}^W(\cdot),
$$
since under $P_{i_0}^{W,u}$, $(Z_s)=(X_{C^{-1}(s)})$ is a Markov jump process with jump rate $W_{i,j}e^{u_j-u_i}$.

Let $\phi((X_v)_{v\le t})$ be a positive measurable test function of the trajectory of $X$ up to time $t$. By Lemma~\ref{lem_sol_inv}~(i), we have
\begin{align*}
\int_{\uuu_{0}^W} E_{i_0}^{W,u}(\phi((X_v)_{v\le t}))  {1\over F^W_{i_0}}\mu_{i_0}^W(du)
&=
\int_{\uuu_{0}^W} \overline \E_{i_0}^{W}\left({R^{W,u}(i_0,0)\over R^{W,u}(\hat X_t)}\phi((X_v)_{v\le t})\right)  {1\over F^W_{i_0}} \mu_{i_0}^W(du)
\\
&=
\overline \E_{i_0}^{W}\left(\phi((X_v)_{v\le t})  \int_{\uuu_{0}^W} {R^{W,u}(i_0,0)\over R^{W,u}(\hat X_t)}  {1\over F^W_{i_0}} \mu_{i_0}^W(du)\right).
\end{align*}
Using  Lemma~\ref{lem_sol_inv}~(ii),
we obtain
$$
 \int_{\uuu_{0}^W} {R^{W,u}(i_0,0)\over R^{W,u}(\hat X_t)}  {1\over F^W_{i_0}} \mu_{i_0}^W(du)=  
 \int_{\uuu_{0}^{W^{T(t)}}}   {1\over F^{W^{T(t)}}_{X_t}} \mu_{X_t}^{W^{T(t)}}(d\tilde u)=1
 $$
 since $  {1\over F^{W^{T(t)}}_{X_t}} \mu_{X_t}^{W^{T(t)}}$ is a probability measure by Theorem~\ref{main}~(i). Hence
$$
\int_{\uuu_{0}^W} E_{i_0}^{W,u}(\phi((X_v)_{v\le t}))  {1\over F^W_{i_0}} \mu_{i_0}^W(du)= 
\overline \E_{i_0}^{W}\left(\phi((X_v)_{v\le t}) \right).
$$
Finally, if $u\in \uuu_0^W$, then under $P^{W,u}_{i_0}$, $(Z_s)$ is Yaglom reversible since $\pi_i:=e^{u_i+u_{i^\star}}$ obviously satisfies $\pi_i W_{i,j} e^{u_{j^\star}-u_{i^\star}}= \pi_{j^\star}W_{j^\star,i^\star}e^{u_i-u_{j}}$ and is an invariant measure, using $u\in \uuu_0^W$. 

\subsection{Proof of Proposition~\ref{Prop_convergence_B} and Theorem~\ref{main_bis}}
\begin{proof}[Proof of Proposition~\ref{Prop_convergence_B}]
Since $u\in \uuu_0^W$, we know that under the law $P^{W,u}_{i_0}$, $(Z_s)$ is the Markov process with jump rates $W_{i,j}e^{u_{j^\star}-u_{i^\star}}$, which has invariant measure 
$$
\pi_i(u)={e^{u_i+u_{i ^\star}}\over \sum_{j\in V} e^{u_j+u_{j^\star}}},\;\;\;\forall i\in V.
$$
Hence  $\lim_{s\to \infty}{1\over s}( \ell_i(\s)-\ell_{i^\star(s)})=0$ a.s. and, by the central limit theorem, the speed of convergence is of order ${1\over \sqrt{s}}$. In particular, we deduce that  ${1\over s} (\ell_i(\s)-\ell_{i^\star(s)})=o(s^{-{1\over 4}})$ a.s.. Using $(\indic_{Z_s=i}-\indic_{Z_s={i^\star}}) ds=d\left(\ell_i(s)-\ell_{i^\star}(s)\right)$, an integration by parts yields
\begin{eqnarray*}
B^\theta_i(s)&=&\demi \int_0^s {d\left(\ell_i(u)-\ell_{i^\star}(u)\right)\over \theta_i+\ell_i(u)+\ell_{i^\star}(u)}
\\
&=&\demi {\ell_i(s)-\ell_{i^\star}(s)\over \theta_i+\ell_i(s)+\ell_{i^\star}(s)}
+\demi \int_0^s {(\ell_i(u)-\ell_{i^\star}(u))(\indic_{Z_u=i}+\indic_{Z_u={i^\star}})\over (\theta_i+\ell_i(u)+\ell_{i^\star}(u))^2}du.
\end{eqnarray*}
The first term converges to 0, and the second is a convergent integral by previous considerations.

The formula \eqref{add} is a direct consequence of the Markov property of $(Z_s)_{s\ge0}$ under $P^{W,u}_{i_0}$, and it is also true for a stopping time instead of fixed time $s$.

Moreover, if $\tau$ is the stopping time defined as the first time the Markov Process $Z_s$ leaves the last point it has visited in $V$, then $B_i^\theta(\tau)$ has a density on $\R^V$, since the jump rates are exponential. By \eqref{add}, $B_i^\theta(\infty)$ is the convolution of  $B_i^\theta(\tau)$ under $P^{W,u}_{i_0}$ and $B^{\theta+\ell(\tau)+\ell^\star(\tau)}(\infty)$ under $P^{W,u}_{Z_\tau}$, hence it has a density.
\end{proof}

Theorem~\ref{main_bis} will be a consequence of the following Proposition \ref{prop:cond}.
\begin{proposition}[Limit theorem for the $\star$-VRJP conditioned on $A$]\label{prop:cond}

(i) Under $\overline \P^{W}_{i_0}$, conditionally on $A$, $U$ is distributed according to
\begin{align*}
{ f^{W,u}_{i_0}\left(\demi(u-u^\star)-A\right)e^{-A_{i_0^\star}}e^{\demi\sum_{(i,j)\in E} W_{i,j}(e^{A_i+A_{j^\star}}-1)}} \mu_{i_0}^W(du).
\end{align*}

(ii) Under $\overline \P^{W}_{i_0}$, conditionally on $A$ and $U$,
  $(Z_s)$ has the law of a conditioned Markov jump process, more precisely:
$$
\overline\P^{W}_{i_0}\left(\;\cdot\;\vert \; A,U\right)=P^{W,U}_{i_0}\left(\; \cdot\; \vert \; B^1(\infty)=\demi(U-U^\star)-A \right) \;\;\; \hbox{ a.s..}
$$
\end{proposition}

Proposition \ref{prop:cond} implies Theorem \ref{main_bis} by the following argument. Since $\nu_{\aaa,i_0}^W$ has a density on $\aaa$, (i) means that for almost all $a\in \aaa$ under  $\P^{W^a}_{i_0}$, the law of the non-ramdomized $\star$-VRJP at rates $W^a$, the law of $U$ is given by the formula above (with $a$ instead of $A$). Hence, for a.e. $W$, conditionally on $a=0$ (which corresponds to the non-randomized $\star$-VRJP $\P^{W}_{i_0}$), the law of $U$ is
$$
f_{i_0}^{W,u} \left((u-u^\star)/2\right) \cdot \mu_{i_0}^W(du)
$$
and, if we condition further  on $U$, the \sVRJP has the law $P^{W,U}_{i_0}\left(\; \cdot\; \vert \; B^1(\infty)=(U-U^\star)/2 \right)$.

\begin{proof}[Proof of Proposition~\ref{prop:cond}]
By definition of $U=(U_i)_{i\in V}$ in Theorem~\ref{main}, we have
$$
\demi (U_i-U_{i^\star})= A_i+\lim_{t\to\infty} \demi(T_i(t)-T_{i^\star}(t)).
$$
Recall the time change defined by $s=C(t)$ defined in Proposition~\ref{exchangeability}. By \eqref{ds_dt}, 
$$
ds= e^{T_i(t)+T_{i^\star}(t)}dt.
$$
Besides, \eqref{M-vs-l} implies
$$
e^{T_i(t)+T_{i^\star}(t)}=1+\ell_i^Z(s)+\ell_{i^\star}^Z(s),
$$
hence
$$
dt= {1\over 1+\ell_i^Z(s)+\ell_{i^\star}^Z(s)} ds.
$$
Changing from time $t$ to time $s$, we deduce
$$
\demi(T_i(t)-T_{i^\star}(t))=\demi\int_{0}^t \left( \indic_{X_v=i}-\indic_{X_v=i^\star}\right) dv= \demi\int_{0}^s {\indic_{Z_u=i}-\indic_{Z_u=i^\star}\over 1+\ell_i^Z(u)+\ell_{i^\star}^Z(u)} du= B_i^1(s).
$$
\begin{remark}\label{rmk_M_B_2}
Note that this proves the formula \eqref{ilim}, since $$e^{T_i(t)}= e^{\demi (T_i(t)+T_{i^\star}(t))}e^{\demi(T_i(t)-T_{i^\star}(t))}=
\sqrt{1+\ell_i^Z(s)+\ell_{i^\star}^Z(s)} e^{B_i(s)},$$ with the time change $s=C(t)$.
\end{remark}
This yields 
\begin{align}\label{lim_B}
A= \demi(U-U^\star)-B^1(\infty).
\end{align}
Under $\overline \P^W_{i_0}$, $A$ is distributed according to ${1\over F_{i_0}^W} \nu_{\aaa,i_0}^W$, and, conditionally on $U$, $(Z_t)$ has law $P^{W,U}_{i_0}$.
Now, under $P^{W,U}_{i_0}$, $B^1(\infty)$ has density  $f^{W,U}_{i_0}(b)$. Hence, under $\overline\P^{W}_{i_0}$ and  conditionally on $U$, $A$ has a density
$$
f^{W,U}_{i_0}(\demi(U-U^\star)-a).
$$
Next, apply Bayes formula : under $\overline\P^{W}_{i_0}$, $A$ is distributed according to ${1\over F_{i_0}^W} \nu_{i_0}^W$, and $U$ is distributed according to ${1\over F^W_{i_0}}\mu_{i_0}^W$. We deduce that, under $\overline\P^{W}_{i_0}$ and conditionally on $A$, $U$ has distribution
\begin{align*}
{ f^{W,u}_{i_0}\left(\demi(u-u^\star)-A\right)e^{-A_{i_0^\star}}e^{\demi\sum_{(i,j)\in E} W_{i,j}(e^{A_i+A_{j^\star}}-1)}} \mu_{i_0}^W(du).
\end{align*}
\end{proof}



\section{Proof of the results of Section~\ref{sec_beta_pot}: Lemma~\ref{lem_identities}, Theorems~\ref{Thm_beta} and~\ref{GThm_beta} and Proposition~\ref{conditionning}}

\label{pr}
Equality~\eqref{eq:beta} in Theorem~\ref{Thm_beta} is a special case of Theorem~\ref{GThm_beta}. The proof of Theorem~\ref{GThm_beta} works by induction on $\dim(\sss)$. 
The section is organized as follows: we start by proving the initialization step of the induction, then we prove the key identities of Lemma~\ref{lem_identities} and deduce Theorem~\ref{GThm_beta} and Proposition~\ref{conditionning}. At the end of the section we prove the statement of Theorem~\ref{Thm_beta} about the finiteness of the integrals.

\subsection{Initialization of the induction, step 1: reduction of the problem}


We start by proving Theorem~\ref{GThm_beta} in the case of a self-dual point or a pair of dual points. Surprisingly, the proof is rather difficult and relies on the Lagrange resolvent method to solve the polynomial equations of degree 4 (\cite{Lagrange} Section 30 and 31, or \cite{Cox} Section~12.C and \cite{Galuzzi}, Section~3.1 for modern references).
\begin{lemma}\label{beta_d=2}
Theorem~\ref{GThm_beta} is true for $V=\{i\}$, $i=i^\star$, or for a pair of dual points, $\vert V\vert=2$ and $V=\{i,i^\star\}$.  
\end{lemma}
\begin{proof}[Step 1 of the proof of Lemma~\ref{beta_d=2}]
By a simple change of variables $\tilde \beta_i=\theta_i\theta_i^\star\beta_i$ we obtain $\int_\sss \nu_\sss^{W,\theta,\eta}(d\beta)=\int_\sss \nu_\sss^{\tilde W,1,\tilde \eta}(d\beta)$ with  $\tilde W_{i,j}=\theta_{i}^\star\theta_j W_{i,j}$, $\tilde \eta_i=\theta_i^\star\eta_i$. On the other hand, from the definition, we have that $\int_\aaa \nu_\aaa^{W,\theta,\eta}(da)=\int_\aaa \nu_\aaa^{\tilde W,1,\tilde \eta}(da)$. Hence it is enough to prove the Lemma in the case $\theta=1$, which is what we assume below.
Besides, note that we have $W_{i,i}=W_{i^\star,i^\star}$. Changing $\beta_i$ into $\beta_i-W_{i,i}$ in $\nu_\sss^{W,\theta,\eta}$, we can always assume that $W$ is null on the diagonal, since on $\nu_\aaa^{W,\theta,\eta}$ the diagonal terms of $W$ cancel directly. 

Assume first that $V=\{i\}$ with $i=i^\star$ a self-dual point, the proof is  simple in this case:
\begin{align*}
\int_\sss \nu_\sss^{W,1,\theta}(d\beta)&={1\over \sqrt{2\pi}} \int_{0}^\infty {d\beta\over \sqrt{\beta}}\exp(-\demi(\beta+{\eta^2\over\beta})+\eta)
\\
&=
{\sqrt{\eta}e^{\eta}\over \sqrt{2\pi}} \int_{-\infty}^\infty \exp(-\eta\cosh(y))e^{\demi y} dy
\\
&=
{\sqrt{\eta}e^{\eta}\over \sqrt{2\pi}} 2 K_{\demi}(\eta)=1,
\end{align*}
where in the third equality we made the change of variable ${\beta\over \eta}=e^y$, and where $K_{\demi}(\eta)={\sqrt{\pi\over 2\eta}}e^{-\eta}$ is the modified Bessel function of the second kind with order $\demi$. On the other hand $\aaa=\emptyset$ and $\nu_\aaa^{W,1,\theta}$ is the constant 1.

Assume now that $V=\{i,i^\star\}$ is a pair of dual points, with $i\neq i^\star$. For simplicity we write $\tilde \nu_\sss^{W,1,\theta}(d\beta)=e^{-\demi\left<1,W 1\right>-\left<\eta,1\right>} \nu_\sss^{W,1,\eta}(d\beta)$ and $\tilde \nu_\aaa^{W,1,\eta}(da)=e^{-\demi\left<1,W 1\right>-\left<\eta,1\right>} \nu_\aaa^{W,1,\eta}(da)$. We first simplify the integral of $\tilde \nu_{\sss}^{W,1,\eta}(d\beta)$.
Set $w_1=W_{i,i^\star}$, $w_2=W_{i^\star,i}$ and $\eta_1=\eta_i$, $\eta_2=\eta_{i^\star}$ (and remind that we assume $\theta_i=\theta_{i^\star}=1$ and $W_{i,i}=W_{i^\star,i^\star}=0$), and $\beta_i=\beta_{i^\star}$ which we simply denote $\beta$. With these notations, we have 
$$
G_{\beta}={1\over \beta^2-w_1w_2} 
\left(\begin{matrix} \beta &w_1
\\ w_2&\beta
\end{matrix}\right);\;\; \left<\eta, G_\beta\eta\right>={2\eta_1\eta_2\beta+w_2\eta_1^2+w_1\eta_2^2\over \beta^2-w_1w_2}.
$$ 
With the notation
\begin{eqnarray}\label{hAB}
\;\;\; h=\sqrt{w_1w_2}, \;\; C={\eta_1\over \sqrt{w_1h}},\;\; D={\eta_2\over \sqrt{w_2 h}},\;\; E=\exp\left(\demi(w_1+w_2)+\eta_1+\eta_2\right),
\end{eqnarray}
changing to variable $x={\beta\over \sqrt{w_1w_2}} $, we obtain that
$$
\int_\sss \tilde\nu_\sss^{W,1,\theta}(d\beta)={E\over \sqrt{2\pi}} \int_1^\infty {dx\over \sqrt{x^2-1}}
\exp\left(-h\left(x+CD{x\over x^2-1}+\demi(C^2+D^2){1\over x^2-1}\right)\right).
$$
Writing 
\begin{eqnarray}\label{pol_P}
P(x)=x^3+(CD-1)x+\demi(C^2+D^2),
\end{eqnarray}
the previous integral is equal to
$$
{1\over \sqrt{2\pi}}\int_1^\infty {dx\over \sqrt{x^2-1}}\exp\left(-h{P(x)\over x^2-1}\right).
$$

Consider now the integral $\tilde\nu_{\aaa}^{W,1,\eta}(da)$. With the notation above, we have
\begin{eqnarray*}
\int_\aaa \tilde\nu_{\aaa}^{W,1,\eta}(da)&=&E\int_\R {da\over \sqrt{2\pi} } \exp\left( -\demi\left(w_1 e^{2a}+w_2 e^{-2a}\right)-\left(\eta_1 e^a+\eta_2 e^{-a}\right)\right)
\\
&=&
E\int_\R {da\over \sqrt{2\pi} } \exp\left( -h\left( \demi\left(\sqrt{{w_1\over w_2}} e^{2a}+\sqrt{{w_2\over w_1}} e^{-2a}\right)-\left({\eta_1\over h} e^a+{\eta_2\over h} e^{-a}\right)\right)\right).
\end{eqnarray*}
Changing to variable $\tilde a$ such that $\sqrt{{w_1\over w_2}} e^{2 a}=e^{2{\tilde a}}$, we have ${\eta_1\over h} e^a =C e^{\tilde a}$ and ${\eta_2\over h} e^{-a}=De^{-\tilde a}$, so that the previous integral equals
$$
\int_\R {d\tilde a\over \sqrt{2\pi} } \exp\left( -h\left( \demi\left( e^{2\tilde a}+e^{-2\tilde a}\right)-\left(C e^{\tilde a}+De^{-\tilde a}\right)\right)\right).
$$
Finally, changing to coordinate $x=e^{\tilde a}$, we obtain that
\begin{align}\label{int_Q}
\int_\aaa \tilde\nu_{\aaa}^{W,1,\eta}(da)={E\over \sqrt{2\pi}}\int_0^\infty \exp \left(-h{Q(x)\over 2 x^2}\right) {dx\over x},
\end{align}
where
\begin{eqnarray}\label{pol_Q}
Q(x)=x^4+2Cx^3+2Dx+1.
\end{eqnarray}
Hence the lemma is equivalent to the following equality of integrals
\begin{align}\label{integrales_Lagrange}
\int_1^\infty {dx\over \sqrt{x^2-1}}\exp\left(-h{P(x)\over x^2-1}\right)=\int_0^\infty \exp \left(-h{Q(x)\over 2 x^2}\right) {dx\over x}.
\end{align}
We note that the left hand side is the Laplace transform of the image of the measure $\indic_{x>1} {dx\over \sqrt{x^2-1}}$ by the function ${P(x)\over x^2-1}$, and the right hand side is the Laplace transform of the image of the measure $\indic_{x>0} {dx\over x}$ by the function ${Q(x)\over 2 x^2}$. The equality is equivalent to the fact that these measures are equal. Hence, the strategy is now to make the change of variable $u={P(x)\over x^2-1}$ or $u={Q(x)\over 2 x^2}$ in each of these equations. This leads to polynomial equations of degree 3 and 4. A difficulty is that theses change of variables are not bijective on their domain but 2 to 1. Remarkably, the  equation of degree 3 is the Lagrange resolvent of the equation of degree 4. This leads to relations between the roots of these two equations, which are at the heart of the argument. The details of the computation are given in next Section~\ref{Lagrange}.
\end{proof}
\subsection{Initialization of the induction, step 2: proof of equality \eqref{integrales_Lagrange} by Lagrange resolvent method}\label{Lagrange}
Let us state the equality as a self-contained lemma. The statements and arguments of this Section~\ref{Lagrange} are completely self-contained, and we freely give a different meaning to symbols already used in the rest of the paper. In particular, $u$, $a$, $\beta$ which appear below have nothing to do with the $u$, $a$ and $\beta$  in the rest of the manuscript.
\begin{lemma}\label{lem_Lagrange}
Let $C\ge 0$, $D\ge 0$ be two reals, and $P(x)$ and $Q(x)$ be the two polynomials
$$
P(x)=x^3+(CD-1)x+\demi(C^2+D^2), 
 \;\;\; Q(x)= x^4+2Cx^3+2Dx+1,
$$
Then
\begin{align}\label{integrales_Lagrange_app}
\int_0^\infty \exp \left(-h{Q(x)\over 2 x^2}\right) {dx\over x}= \int_1^\infty {dx\over \sqrt{x^2-1}}\exp\left(-h{P(x)\over x^2-1}\right).
\end{align}
for all $h>0$.
\end{lemma}
\begin{proof}
If $C=D=0$ the proof is simple:  ${Q(x)\over  2 x^2}=\demi(x^2+{1\over x^2})$, changing to variable $u$ such that $e^u=x^2$ we obtain that
$$
\int_0^\infty \exp \left(-h{Q(x)\over 2 x^2}\right) {dx\over x}=\demi \int_\R e^{-h\cosh(u)}du=  \int_0^\infty e^{-h\cosh(u)}du.
$$
(which is the modified Bessel function $K_0(h)$). On the other hand, ${P(x)\over x^2-1}=x$ and changing to variable $u$ such that $x=\cosh(u)$, we obtain 
$$
\int_1^\infty {dx\over \sqrt{x^2-1}}\exp\left(-h{P(x)\over x^2-1}\right)= \int_0^\infty e^{-h\cosh(u)}du.
$$

Let us now suppose that $C>0$ or $D>0$. 
As explained above, the strategy is to make the change of variables $u={Q(x)\over x^2}$ and $u={P(x)\over x^2-1}$ in each of the integrals, and to use that $P$ and $Q$ are related by the Lagrange resolvent method. 

Let us start by $Q$. The equation $u={Q(x)\over 2 x^2}$ is equivalent to $Q(x,u)=0$ with
\begin{eqnarray}\label{Qu}
Q(x,u)= x^4+2Cx^3-2ux^2+2Dx+1.
\end{eqnarray}
We have $Q(0,u)=1$ and $Q(+\infty,u)=+\infty$. Moreover, if $x>0$, then $Q(-x,u)\le Q(x,u)$ since $C\ge 0$, $D\ge 0$ and $Q(x,u)$ is decreasing in $u$. It implies that there exists $\underline u>0$ such that $Q(x,u)$ has two positive simple roots and two negative simple roots for $u>\underline u$,  no positive root for $u<\underline u$, and a double positive root for $u=\underline u$.  For $u>\underline u$ we denote by $a(u),b(u),c(u),d(u)$ (we often simply write $a,b,c,d$) the roots of the polynomial $x\to Q(x,u)$ and we choose the roots such that $0<a<b$, and $c<d<0$. Note that the infimum
$$
\underline u=\inf_{x>0} {Q(x)\over 2 x^2}
$$
 is reached at a unique value $\underline x>0$, and that $x\to {Q(x)\over 2 x^2}$ is bijective from the interval $(0,\underline x)$ (resp. $(\underline x,+\infty)$) onto $(\underline u,+\infty)$, with inverse $a(u)$ (respectively $b(u)$). Performing the change of variable $u={Q(x)\over 2 x^2}$ on each of these intervals, we obtain that
\begin{eqnarray}\label{L_chgt_var_1}
\int_0^\infty \exp \left(-h{Q(x)\over 2 x^2}\right) {dx\over x}=
\int_{\underline u}^\infty 
e^{-hu}\left( -{a'(u)\over a(u)}+{b'(u)\over b(u)}\right) du.
\end{eqnarray}

Let us now consider the polynomial $P$. The equation $u={P(x)\over x^2-1}$ is equivalent to $P(x,u)=0$, where
$$
P(x,u)= x^3-ux^2+(CD-1)x +\demi(C^2+D^2)+u.
$$
We remark that $P(1,u)=\demi(C+D)^2>0$, hence $P(x,u)$ has at least one root on $(-\infty,1)$. Besides, for $x\in (1,\infty)$, $P(x,u)$ is decreasing in $u$, hence there is $\underline u' >0$ such that $x\to P(x,u)$ has two simple roots (resp. no root) on $(1,\infty)$ for $u>\underline u'$ (resp. $u<\underline u'$) and hence a double root on $(1,\infty)$ for $u=\underline u'$. (We will show later that $\underline u'=\underline u$).

The key relation between $Q$ and $P$ is that, up to a simple change of variable, $P(x,u)$ is the Lagrange resolvent of $Q(x,u)$. Let us briefly recall what we need on Lagrange resolvents (cf \cite{Lagrange}, Section seconde, 30 and 31, or \cite{Cox} Section~12.C and \cite{Galuzzi} Section~3.1 for modern references): if 
$$
x^4+ mx^3+nx^2+px+q
$$
is a polynomial of degree 4, with roots $a,b,c,d$, its Lagrange resolvent is the polynomial of degree 3, given by
$$
y^3-ny^2+(mp-4q)y-(m^2-4n)q-p^2,
$$
which has roots 
$$
ab+cd,\;\; ac+bd,\;\; ad+bc.
$$

Coming back to our question, the Lagrange resolvent of $Q(x,u)$ is
$$
R(y)=y^3+2uy^2+y(4CD-4)-8(\demi(C^2+D^2)+u)=8P(-\demi y, u).
$$
In particular, it means that $P(x,u)$ has roots
$$
\gamma(u):=-\demi (ab+ cd), \;\; \alpha(u):=-\demi(ac+bd), \;\; \beta(u):= -\demi(ad+bc).
$$
We remark that for $u>\underline u$, $\gamma<0$ and $\beta>1$, $\alpha>1$ since $a>0$, $b>0$ and $c<0$, $d<0$ and $abcd=1$ (using the arithmetico-geometric inequality), so that $x\to P(x,u)$ has two roots on $(1,\infty)$. For $u=\underline u$, $a=b$ is a double root of $Q(x,u)$, hence $\alpha=\beta$ is double root of $P(x,u)$ on $(1,\infty)$. In particular it implies that $\underline u=\underline u'$. Moreover, for $u>\underline u$, we have $1<\alpha<\beta$: indeed, we have  $d<c<0<a<b$, hence  $(a-b)(d-c)=ad+bc-(bd+ac)>0$. Proceeding as for $Q$, for $u>\underline u$, we consider $\underline x'$ which is the unique point such that ${P(x,u)\over x^2-1}$ reaches its minimum on $(1,\infty)$ and make the change of variables $u={P(x,u)\over x^2-1}$ on $(1,\underline x')$ and on $(\underline x',+\infty)$. This leads to the following equality
\begin{align}\label{L_chgt_var_2}
\int_1^\infty {dx\over \sqrt{x^2-1}}\exp\left(-h{P(x)\over x^2-1}\right)
=
\int_{\underline u}^\infty e^{-hu}\left({\beta'(u)\over \sqrt{\beta(u)-1}}-{\alpha'(u)\over \sqrt{\alpha(u)-1}}\right).
\end{align}

In the last step, we use a key argument in the Lagrange method to solve the 4th degree polynomial equation: since $-2\beta=ac+bd$ and $abcd=1$, $ac$ and $bd$ are solutions of the equation of degree 2:
$$ t^2+2\beta t+1=0.
$$
Since $d<c<0<a<b$ we have that $bd<ac<0$, hence $ac$ is the solution $ac=-\beta+\sqrt{\beta^2-1}$, hence differentiating in $u$, we deduce
$$
ab({a'\over a}+{c'\over c})=-\beta'{-\beta+\sqrt{\beta^2-1}\over \sqrt{\beta^2-1}}, \hbox{ hence } \;\; {a'\over a}+{c'\over c}={ -\beta'\over \sqrt{\beta^2-1}}.
$$

Similarly, $ad$ and $bc$ are solutions of the equation $t^2+2\alpha t+1=0$. Moreover, we have $ad<bc<0$: indeed, since for $u>\underline u>0$ we have $Q(-x,u)<Q(u,x)$ for $x>0$, we deduce $Q(-b,u)<0$, $Q(-a,u)<0$ which gives $d<-b<-a<c$, hence $\vert ad\vert <\vert bc\vert$. Thus $bc$ is the root $bc=-\alpha+\sqrt{\alpha^2-1}$. Finally,
$$
{b'\over b}+{c'\over c}={ -\alpha'\over \sqrt{\alpha^2-1}}.
$$
Taking the difference of the two identities we have obtained, we can conclude that
$$
{a'\over a}-{b'\over b}={\alpha'\over \sqrt{\alpha^2-1}}-{ \beta'\over \sqrt{\beta^2-1}}.
$$
This concludes the proof with \eqref{L_chgt_var_1} and \eqref{L_chgt_var_2}.
\end{proof}

\subsection{Proof of Lemma~\ref{lem_identities}}
By hypothesis, $I$ is a non-empty subset such that $I^\star=I$ and $I\subsetneq V$, so that we can write 
$$
V=I\sqcup I^c,
$$
with $I^c$ non-empty and $(I^c)^\star=I^c$. 
The first part of the proof generalizes some computations done in the proof of lemma 4 in \cite{SZ19} or in section 4 \cite{LW17}. 
We can write \(H_{\beta}\) as the block matrix
 \begin{equation}
 \label{block}
 H_{\beta}=
    \begin{pmatrix}
   (H_\beta)_{I,I} & -W_{I,I^{c}} \\ -W_{I^{c},I} & (H_\beta)_{I^{c},I^{c}} 
    \end{pmatrix}
\end{equation}
and set
\begin{equation}
\label{Gbe}
\hat H_\beta=(H_\beta)_{I,I},\,\,
\hat G_\beta= \hat H_\beta^{-1},\,\,\check W= W_{I^c, I^c}+W_{I^c, I}\hat G_\beta W_{I,I^c},
\end{equation}
\begin{align}\label{Hcheck}
\check H_\beta= \be_{I^c}-\check W= (H_{\beta})_{I^c, I^c}-W_{I^c, I}\hat G_\beta W_{I,I^c}, \;\;\; \check G_\beta= \check H_\beta^{-1}.
\end{align}
Note that $\check H_\beta$ corresponds to the Schur complement of the matrix $H_\beta$ on the sub-block $I^c\times I^c$.  For this reason, classically, we have
$$
\check G_\beta=(G_\beta)_{I^c,I^c}.
$$
More precisely, with the previous notations we have
  \begin{equation}
    \label{eq-hbeta-productof3}
H_{\beta}=
    \begin{pmatrix}
      \Id_{I,I} & 0 \\ -W_{I^{c},I}\hat{G}_\beta & \Id_{I^c,I^c}
    \end{pmatrix}
    \begin{pmatrix}
      \hat H_\beta  & 0 \\ 0 & \check H_\beta
    \end{pmatrix}
    \begin{pmatrix}
      \Id_{I,I} & -\hat{G}_\beta W_{I,I^{c}} \\ 0 & \Id_{I^c,I^c}
    \end{pmatrix},
\end{equation}
and, subsequently,
 \begin{equation}
    \label{eq-schur-pour-gbeta}
    G_{\beta}=
               \begin{pmatrix}
                 \Id_{I,I} & \hat{G}_\beta W_{I,I^{c}} \\ 0 & \Id_{I^c,I^c}
               \end{pmatrix}                                                           \begin{pmatrix}
\hat{G}_\beta & 0\\ 0 & \check{G}_\beta                                                               \end{pmatrix}
                          \begin{pmatrix}
                            \Id_{I,I} & 0 \\ W_{I^{c},I}\hat{G}_\beta & \Id_{I^c,I^c}
                          \end{pmatrix}\\
  \end{equation}
 
 By \eqref{eq-hbeta-productof3}, we have
  \begin{equation}
    \label{eq-1hbeta1}
    \begin{aligned}
    &\left< \theta,H_{\beta}\theta \right>
    \\=&\left< \theta_{I^c} ,\check H_\beta \theta_{I^c} \right>+\left< \theta_{I},\hat H_{\beta}\theta_{I} \right>+
    \left< \theta_{I^c},W_{I^{c},I}\hat{G}_\beta W_{I,I^{c}} \theta_{I^c}\right>-2\left< \theta_{I},W_{I,I^{c}} \theta_{I^c}\right>
        \end{aligned}
 \end{equation}
On the other hand,
\[ \begin{pmatrix}
                            \Id_{I,I} & 0 \\ W_{I^{c},I}\hat{G}_\beta & \Id_{I^c,I^c}
                          \end{pmatrix}
                          \begin{pmatrix}
                            \eta_{I} \\ \eta_{I^{c}}
                          \end{pmatrix}=
                          \begin{pmatrix}
                            \eta_{I}\\ \check{\eta}_{I^c}
                          \end{pmatrix}
.\]

    We deduce from  \eqref{eq-schur-pour-gbeta} that
  \begin{equation}
    \label{eq-1gbeta1}
\left< \eta,G_{\beta}\eta \right>=\left< \eta_{I},\hat{G}_\beta \eta_{I} \right>+\left< \check \eta, \check{G}_\beta \check\eta \right>.
\end{equation}

Set $\hat\eta=\hat\eta_I$ and  $\check{\eta}=\check{\eta}_{I^c}$ for simplicity. Observe that $W$
and $\hat{G}_\beta$ are symmetric with respect to $\left<.,.\right>$, so that the adjoint of 
$$\begin{pmatrix}
      \Id_{I,I} & -\hat{G}_\beta W_{I,I^{c}} \\ 0 & \Id_{I^c,I^c}
    \end{pmatrix}$$
    with respect to that bilinear form is
$$    \begin{pmatrix}
      \Id_{I,I} & 0 \\ -W_{I^{c},I}\hat{G}_\beta & \Id_{I^c,I^c}
    \end{pmatrix}.$$
     
 Hence, using \eqref{eq-hbeta-productof3}, we deduce
 \begin{equation} 
 \label{adhelp}
 \left< \theta,H_{\beta}\theta \right>=\left<\zeta,D_H\zeta\right>,
 \end{equation}
 with 
 $$D_H=\begin{pmatrix}
      \hat H_\beta  & 0 \\ 0 & \check H_\beta
    \end{pmatrix},
   \,\,\zeta= \begin{pmatrix}
      \Id_{I,I} & -\hat{G}_\beta W_{I,I^{c}} \\ 0 & \Id_{I^c,I^c}
    \end{pmatrix} \tet.
    $$
    
Combining \eqref{eq-1hbeta1}, \eqref{adhelp} and \eqref{eq-1gbeta1}, and noting that $\left<\eta,\tet\right>=\left<\hat\eta,\tet_I\right>+\left<\check{\eta},\tet_{I^c}\right>$, we deduce
\begin{align}\label{term_exp}
 \left< \theta,H_{\beta}\theta \right>+\left< \eta,G_{\beta}\eta \right>-2\left<\eta,\theta\right>
  =&\left< \theta_{I^c},\check{H}_\beta \theta_{I^c} \right>+ \left< \check{\eta},\check{G}_\beta \check{\eta}  \right>-2\left<\check\eta,\theta_{I^c}\right>\\
\nonumber   &+\left< \theta_{I},\hat H_\beta\theta_{I} \right>+\left< \hat\eta,\hat{G}_\beta \hat\eta \right>-2\left<\hat\eta,\theta_{I}\right>
\end{align}

Note that, by  \eqref{eq-hbeta-productof3}, 
  \begin{align}\label{prod_det}
  \det H_{\beta}=\det \hat H_\beta \det \check{H}_\beta,\ \ \mathds{1}_{H_{\beta}>0}=\mathds{1}_{\hat H_{\beta}>0}\mathds{1}_{\check{H}_{\beta}>0}.
  \end{align}
Combining \eqref{term_exp} and \eqref{prod_det}, we deduce
  \begin{eqnarray*}
\nu_{\sss}^{W,\theta,\eta}(d\beta)&=& {\prod_{i\in V_0} \theta_i\over \sqrt{2\pi}^{\vert\sss\vert}} e^{-\frac{1}{2}\left< \theta,H_{\beta}\theta \right>-\frac{1}{2}\left< \eta,G_{\beta}\eta \right>+\left<\eta,\theta\right>}\frac{\mathds{1}_{H_{\beta}>0}}{\sqrt{\det H_{\beta}}}d\beta
    \\
    &=&{\prod_{i\in V_0\cap I} \theta_i\over \sqrt{2\pi}^{\vert\sss_I\vert}} e^{-\frac{1}{2} \left< \theta_I,\hat H_{\beta}\theta_I \right>-\demi\left<\hat \eta,\hat{G}_\beta\hat\eta \right>+\left<\hat\eta,\theta_{I} \right>  }\frac{\mathds{1}_{\hat H_{\beta}>0}}{\sqrt{\det \hat H_{\beta}}}d\beta_I\\
    &&\;\;\;\;\cdot {\prod_{i\in V_0\cap I^c} \theta_i\over \sqrt{2\pi}^{\vert\sss_{I^c}\vert}}  e^{-\frac{1}{2}\left< \theta_{I^c},\check{H}\theta_{I^c} \right>
   -\demi \left< \check{\eta},\check{G}_{\beta} \check{\eta}  \right>+\left<\check\eta,\theta_{I^c}\right>}\frac{\mathds{1}_{ \check{H}_{\beta}>0}}{\sqrt{\det \check{H}_{\beta}}}d\beta_{I^c}
   \\
   &=&
\nu_{\sss_I}^{W_{I,I},\theta_I,\eta_I}(d\beta_I) \nu_{\sss_{I^c}}^{\check W_{I^c,I^c},\theta_{I^c},\check\eta_{I^c}}(d\beta_{I^c}).
  \end{eqnarray*}
Hence, this proves the first identity of Lemma~\ref{lem_identities}. 

 Next we prove the third identity of Lemma~\ref{lem_identities}. Using that $\check W_{I^c,I^c}=W_{I^c,I^c}+W_{I^c,I}\hat G_\beta W_{I,I^c}$, and that $\check \eta_{I^c}=\eta_{I^c} +W_{I^c,I}\hat G_\beta \theta_{I^c}$, and using the $\left<\cdot,\cdot \right>$-symmetry of the matrix $W$, we deduce:
\begin{align*}
& 
-\demi\left<\theta^{a}_{I^c}, \check W \theta^{a}_{I^c}\right>
+\demi\left<\theta_{I^c},\check W \theta_{I^c}\right>
-\left<\check\eta,(\theta^{a}_{I^c}-\theta_{I^c})\right>
-\demi\left<\hat \eta,\hat{G}_\beta\hat\eta \right>
\\
=& -\demi \left<\theta^{a}_{I^c}, W_{I^c,I^c} \theta^a_{I^c}\right>+\demi \left<\theta_{I^c},W_{I^c,I^c} \theta_{I^c}\right>-
\left<\eta_{I^c},\theta^a_{I^c}-\theta_{I^c}\right>
\\
&
-\demi \left<W_{I,I^c} \theta^{a}_{I^c},\hat G_\beta W_{I,I^c}\theta^{a}_{I^c}\right>+\demi \left<W_{I,I^c} \theta_{I^c},\hat G_\beta W_{I,I^c}\theta_{I^c}\right>
\\
&
-\left<\eta_I,\hat G_\beta W_{I, I^c} \left(\theta_{I^c}^a-\theta_{I^c}\right)\right>-\demi\left<\hat \eta,\hat{G}_\beta\hat\eta \right>.
\end{align*}
Remark now that, using the definitions of $\hat \eta$ and $\hat \eta^a$,
$$
\left<\hat \eta,\hat{G}_\beta\hat\eta \right>=\left<\hat \eta^a,\hat{G}_\beta\hat\eta^a \right>+ \left<W_{I,I^c} \theta_{I^c},\hat G_\beta W_{I,I^c}\theta_{I^c}\right>+2 \left<\eta_{I^c},\theta_{I^c}-\theta^a_{I^c}\right>-\left<W_{I,I^c} \theta^{a}_{I^c},\hat G_\beta W_{I,I^c}\theta^{a}_{I^c}\right>.
$$
Combined with the previous equality, it gives
\begin{align*}
&
-\demi\left<\theta^{a}_{I^c}, \check W \theta^{a}_{I^c}\right>
+\demi\left<\theta_{I^c},\check W \theta_{I^c}\right>
-\left<\check\eta,(\theta^{a}_{I^c}-\theta_{I^c})\right>
-\demi\left<\hat \eta,\hat{G}_\beta\hat\eta \right>
\\
=&
 -\demi \left<\theta^{a}_{I^c}, W_{I^c,I^c} \theta^a_{I^c}\right>+\demi \left<\theta_{I^c},W_{I^c,I^c} \theta_{I^c}\right>-
\left<\eta_{I^c},(\theta^{a}_{I^c}-\theta_{I^c})\right>-\demi \left<\hat{\eta}_I^{a},\hat G_\beta \hat{\eta}_I^{a}\right>.
\end{align*}
Coming back to the definitions, this yields
$$
\nu_{\sss_I}^{W_{I,I},\theta_I,\hat \eta_I}(d\beta_I) \nu_{\aaa_{I^c}}^{\check W_{I^c,I^c},\theta_{I^c},\check \eta_{I^c}}(da_{I^c})=
\nu_{\sss_I}^{W_{I,I},\theta_I,\hat \eta^a_I}(d\beta_I) \nu_{\aaa_{I^c}}^{W_{I^c,I^c},\theta_{I^c},\hat \eta_{I^c}}(da_{I^c})
=\mathcal{Q}_I^{W,\tet,\eta}(d\beta_{I},da_{I^c}),
$$ 
where the last equality comes from the definition.


The last step is to prove the second equality of Lemma~\ref{lem_identities}, i.e. that
$$
\nu_{\aaa_I}^{W_{I,I},\theta_I,\hat \eta^a_I}(da_I) \nu_{\aaa_{I^c}}^{W_{I^c,I^c},\theta_{I^c},\hat \eta_{I^c}}(da_{I^c})  
=\nu_{\aaa}^{W,\theta,\eta}(da).
$$
This comes from the fact that 
\begin{align*}
&
-\left<\hat{\eta}_I^{a}, \theta_I^a-\theta_I\right>=  -\left<{\eta}_I, \theta_I^a-\theta_I\right> -\left<W_{I,I^c}{\theta}_{I^c}^{a}, \theta_I^a-\theta_I\right>
\\
&
-
\left<\hat \eta_{I^c},\theta^a_{I^c}-\theta_{I^c}\right>
=-
\left<\eta_{I^c},\theta^a_{I^c}-\theta_{I^c}\right>-
\left<W_{I^c,I} \theta_{I},\theta^a_{I^c}-\theta_{I^c}\right>,
\end{align*}
hence
  \begin{align*}
  &\nu_{\aaa_I}^{W_{I,I},\theta_I,\hat \eta^a_I}(da_I) \nu_{\aaa_{I^c}}^{W_{I^c,I^c},\theta_{I^c},\hat \eta_{I^c}}(da_{I^c}) 
\\
=&e^{ -\demi\left<\tet^a_I, W_{I,I}\theta_I^{a}\right> +\demi \left<\theta_I, W_{I,I}\theta_{I}\right> -\left<\hat{\eta}_I^{a}, \theta_I^a-\theta_I\right>} 
e^{ -\demi \left<\theta^{a}_{I^c}, W_{I^c,I^c} \theta^a_{I^c}\right>+\demi \left<\theta_{I^c},W_{I^c,I^c} \theta_{I^c}\right>-
\left<\hat \eta_{I^c},\theta^a_{I^c}-\theta_{I^c}\right>} da
\\
=&
 e^{-\demi\left<\theta^{a}, W\theta^a\right>+\demi\left<\theta,W\theta\right>-\left<\eta,\theta^{a}\right>+\left<\eta,\theta\right>} da= \nu_{\aaa}^{W,\theta,\eta}(da),
\end{align*}
which concludes the proof of Lemma~\ref{lem_identities}.
\subsection{Proof of Theorem~\ref{GThm_beta} and Proposition~\ref{conditionning}}
The proof works by induction on $\dim(\sss)=\vert V_0\sqcup V_1\vert$. The initialization $\dim(\sss)=1$ has already been done. Assume that the statement is true for $\dim(\sss)\le n$, and consider a graph $\ggg$ such that $\dim(\sss)=n+1$. Consider $I\subsetneq V$ as in Lemma~\ref{lem_identities} so that $\dim(\sss_I)\le n$ and $\dim(\sss_{I^c})\le n$.
From the first equality in Lemma~\ref{lem_identities}, we deduce
$$\nu_\sss^{W,\theta,\eta}(d\beta)=\nu_{\sss_I}^{W_{I,I},\theta_I,\eta_I}(d\beta_I) \nu_{\sss_{I^c}}^{\check W_{I^c,I^c},\theta_{I^c},\check\eta_{I^c}}(d\beta_{I^c}).$$
Note that, by definition, the parameters $\check W_{I^c,I^c}$ and $\check\eta_{I^c}$ only depend on $\beta_I$ and not on $\beta_{I^c}$. We now integrate on $\beta_{I^c}$, while leaving $\beta_I$ fixed, hence considering $\check W_{I^c,I^c}$ and $\check\eta_{I^c}$ as fixed parameters. Applying the induction formula on the subgraph with vertices $I^c$,
we obtain
 \begin{align*}
&
\int_{\sss_{I^c} } 
\nu_{\sss_{I^c}}^{\check W_{I^c,I^c},\theta_{I^c},\check\eta_{I^c}}(d\beta_{I^c})  
=
\int_{\aaa_{I^c}} \nu_{\aaa_{I^c}}^{\check W_{I^c,I^c},\theta_{I^c},\check\eta_{I^c}}(da_{I^c}).  
  \end{align*}
Integrating now on $d\beta_I$, and by definition of $\mathcal{Q}_I^{W,\tet,\eta}(d\beta_{I},da_{I^c})$ (see Lemma~\ref{lem_identities}), we obtain
$$
\int_{\sss} \nu_{\sss}^{W,\theta,\eta}(d\beta)=\int_{\sss_{I}\times \aaa_{I^c}} \mathcal{Q}_I^{W,\tet,\eta}(d\beta_{I},da_{I^c}),
$$
which proves the first part of the equality of Theorem~\ref{GThm_beta}. Remark that, by the same argument, we also prove the first part of Proposition~\ref{conditionning}~ii).

By the third equality of Lemma~\ref{lem_identities}, using the other expression of $\mathcal{Q}_I^{W,\tet,\eta}$, we also have
$$
\mathcal{Q}_I^{W,\tet,\eta}(d\beta_{I},da_{I^c})
=
\nu_{\sss_I}^{W_{I,I},\theta_I,\hat \eta^a_I}(d\beta_I) \nu_{\aaa_{I^c}}^{W_{I^c,I^c},\theta_{I^c},\hat \eta_{I^c}}(da_{I^c}).
$$
Remark that, on the right-hand side in $\nu_{\aaa_{I^c}}^{W_{I^c,I^c},\theta_{I^c},\hat \eta_{I^c}}(da_{I^c})$, the parameters do not depend on $\beta_I$, while in $\hat\eta^a_I$ only depends on $a_{I^c}$. Hence, applying the recurrence hypothesis on the subgraph with vertices $I$, conditioned on $a_{I^c}$, we deduce
$$
\int_{\sss_I} \nu_{\sss_I}^{W_{I,I},\theta_I,\hat \eta^a_I}(d\beta_I) = \int_{\aaa_I} \nu_{\aaa_I}^{W_{I,I},\theta_I,\hat \eta^a_I}(da_I).
$$
Integrating now on $a_{I^c}$, and using the second equality of Lemma~\ref{lem_identities}, this implies
$$
\int_{\sss_{I}\times \aaa_{I^c}} \mathcal{Q}_I^{W,\tet,\eta}(d\beta_{I},da_{I^c})=
\int_{\aaa} \nu_{\aaa_I}^{W_{I,I},\theta_I,\hat \eta^a_I}(da_I) \nu_{\aaa_{I^c}}^{W_{I^c,I^c},\theta_{I^c},\hat \eta_{I^c}}(da_{I^c})= \int_{\aaa}  \nu_\aaa^{W,\theta,\eta}(da).
$$
This proves the second part of the equality of Theorem~\ref{GThm_beta} and, by a similar argument, the second part of Proposition~\ref{conditionning}~ii).

The conditioning properties, Proposition~\ref{conditionning}~i), are direct consequences of each of the identities of Lemma~\ref{lem_identities}.

\subsection{Proof of the integrability condition of Theorem~\ref{Thm_beta}}
Under the assumption of  Theorem~\ref{Thm_beta}, we will prove that $\int_\aaa \nu_{\aaa}^{W,\theta,\eta}(da) <\infty$, which is equivalent to $F^{W,\theta,\eta} <\infty$ by \eqref{eq:beta} and definition. Changing $W_{i,j}$ into $\tilde W_{i,j}=\theta_{i}^*\theta_j W_{i,j}$ and $\eta_i$ into $\tilde \eta_i=\theta_i^\star\eta_i$, we can always assume that $\theta_i=1$ for all $i\in V$, which is what we do in the sequel. Set $w:=\inf_{(i,j)\in E} W_{i,j} >0$. For $i\in V$, let $\aaa_i=\{a\in \aaa, \; a_i\ge \max_{j\in V} \vert a_j\vert\}$, we have 
$$
\int_\aaa \nu_{\aaa}^{W,1,\eta}(da) \le \sum_{i\in V} \int_{\aaa_i} \nu_{\aaa}^{W,1,\eta}(da).
$$
Fix $i\in V$. If there exists a directed path from $i$ to $i^\star$, then the proof is essentially the same as the proof of Part~\ref{partI}, Lemma~\ref{finitness}. 

Suppose now that the other assumption is satisfied for $i$, i.e. that there exists $j\in V$ such that there is a directed path from $i$ to $j$ and $\eta_{j}>0$. Let $\sigma$ be a shortest path from $i$ to $j$, then there exists $k$ such that $\sigma_k-\sigma_{k+1}\ge {a_i-a_j\over \vert \sigma \vert}$. Hence
$$
\nu_{\aaa}^{W,1,\eta}(da)\le C_{W,\eta} \exp\left(-\demi \sum_{i\to j} W_{i,j} e^{a_i-a_j}- \eta_{j} e^{a_j} \right) da \le \exp\left(-\demi w e^{a_i-a_j\over \vert \sigma\vert} - \eta_{j} e^{a_j}\right) da.
$$
We easily obtain $\max({a_i-a_j\over \vert \sigma\vert}, a_j)\ge {a_i\over \vert \sigma\vert+1}$, hence, by the same argument as in the proof of Part~\ref{partI}, Lemma~\ref{finitness},
$$
\int_{\aaa_i} \nu_{\aaa}^{W,1,\eta}(da)\le C_{W,\eta}\int_0^\infty \vert 2a_i \vert^{\vert V_1\vert}  \exp\left(- \min(\demi w,\eta_j) e^{a_i\over \vert \sigma\vert+1}\right) da_i,
$$
which is finite. This concludes the proof.
\section{Proof of Theorem~\ref{mixing-Ib}}

The strategy of the proof of Theorem~\ref{mixing-Ib} goes as follows: we pick $(\bbb_I,A_{I^c})$ according to the law ${1\over F^W_{i_0}}\qqq^W_{I,i_0}$ and run the process defined in Theorem~\ref{mixing-Ib}~ii). Then we prove that the law of this process is that of the randomized \sVRJP $(X_t)$, i.e. $\overline \P_{i_0}^W$. To complete the proof, we show that the random
variables $(\bbb_I,A_{I^c})$ appear asymptotically as functions of the path of the process and coincide with the corresponding values defined in Theorem~\ref{mixing-Ib}~i).

Let us fix some notation: we denote by $\tilde \P_{I,i_0}^W$ the joint law $\left((X_t),  (\bbb_I, A_{I^c})\right)$ where $(\bbb_I,A_{I^c})$ is distributed according to the law ${1\over F^W_{i_0}}\qqq^W_{I,i_0}$ and conditionally on $(\bbb_I,A_{I^c})$, the process $(X_t)$ is defined in Theorem~\ref{mixing-Ib}~ii). By abuse of notation, we sometimes consider $\tilde \P_{I,i_0}^W$ as the law of its marginal $(X_t)$.

By definition, under  $\tilde \P_{I,i_0}^W$ and
conditioned on $(\bbb_I,A_{I^c})$, the jump rate of $X(t)$ at time $t$ is
$$
\sum_{j, X_u\to j} W_{X_u,j} e^{T_{X_u}(u)+T_{X_u^\star}(u)}e^{V_{j^\star}(u)-V_{X_u^\star}(u)}.
$$
By a classical computation, under $\tilde \P_{I,i_0}^W$  and conditioned on $(\bbb_I,A_{I^c})$, the probability that the canonical process $X$ at time $t$ has performed $n$ jumps in infinitesimal time intervals $[t_i, t_i+dt_i)$, $0<t_1<\cdots <t_n<t$, following the trajectory
$\sigma_{0}=i_0, \sigma_1, \ldots, \sigma_n=j_0$ is equal to
\begin{eqnarray} \label{proba_traject-2}
&&\exp\left(-\int_0^t \sum_{j, X_u\to j} W_{X_u,j} e^{T_{X_u}(u)+T_{X_u^\star}(u)}e^{V_{j^\star}(u)-V_{X_u^\star}(u)}du\right)
\\
\nonumber
&&\cdot \left(\prod_{l=1}^{n} W_{\sigma_{l-1}, \sigma_l} e^{T_{\sigma_{l-1}}(t_l)+T_{\sigma_{l-1}^\star}(t_l)} e^{V_{\sigma_l^\star}(t_l)-V_{\sigma_{l-1}^\star}(t_l)} dt_l\right).
\end{eqnarray}
Due to several simplifications detailed just below we have
$$
\prod_{l=1}^{n}  e^{V_{\sigma_l^\star}(t_l)-V_{\sigma_{l-1}^\star}(t_l)}
= e^{V_{j_0^\star}(t)-V_{i_0^\star}(0)-\sum_{i\in V_0\cap I} T_i(t)}.
$$
Indeed:
\begin{itemize}
\item
We know that $V(t)$ is constant when $X(t)$ is in $I$. Hence when $\sigma_l$ and $\sigma_l^\star$ are in $I$, the term $e^{V_{\sigma_l^\star}(t_l)}$ cancels with the next term $e^{-V_{\sigma_l^\star}(t_{l+1})}$ since $V_{\sigma_l^\star}$ is constant on the interval $[t_l, t_{l+1}]$. 
\item
When $\sigma_l$ and $\sigma_l^\star$ are in $I^c$, we have $V_{\sigma_l^\star}(t_l)=T_{\sigma_l^\star}(t_l)+A_{l^\star}$ and we have a similar simplification as in the proof of Part~\ref{partI} Proposition~\ref{exchangeability}~equation~\eqref{density-*}. More precisely,
\begin{itemize}
\item
If $\sigma_l\neq\sigma_l^\star$, the local time $T_{\sigma_l^\star}$ does not change between the time $s_k$ at which the process jumps to $\sigma_l$ and the time $s_{l+1}$ at which it leaves $\sigma_l$. Hence, $T_{\sigma_l^\star}(s_l)$ cancels with $T_{\sigma_l^\star}(s_{l+1})$.
\item
If $\sigma_l=\sigma_l^\star$, the local time $T_{\sigma_l^\star}$ does not change between the time $s_{l+1}$ at which the process leaves $\sigma_l$ and the first time after $s_{l+1}$ at which it comes back to $\sigma_l$. Hence, there is a cancellation at each return time to a self-dual point. Since the initial local time is 0, it leaves the contribution $-\sum_{i\in V_0\cap I} T_i(t)$ in the formula above.
\end{itemize} 
\end{itemize}
Hence, integrating on the distribution ${1\over F^W_{i_0}}\qqq^W_{I,i_0}$ of the random variables $(\bbb_I,A_{I^c})$, we see that under $\tilde \P_{I,i_0}^W$, the probability that the process $X$ at time $t$ has performed $n$ jumps in infinitesimal time intervals $[t_i, t_i+dt_i)$, $0<t_1<\cdots <t_n<t$, following the trajectory
$\sigma_{0}=i_0, \sigma_1, \ldots, \sigma_n=j_0$ is equal to
\begin{align}\label{proba-traject-2-integree}
&\;\;\;\;\;\;\;\;
{1\over F^W_{i_0}}
\left(\prod_{l=1}^{n} W_{\sigma_{l-1}, \sigma_l} e^{T_{\sigma_{l-1}}(t_l)+T_{\sigma_{l-1}^\star}(t_l)} dt_l\right)e^{-\sum_{i\in V_0\cap I^c} T_i(t)}
\\
\nonumber
&
\!\!\!\!\!\!\!\cdot\int_{\sss_I\times\aaa_{I^c}}
e^{V_{j_0^\star}(t)-V_{i_0^\star}(0)-\int_0^t \sum_{j, X_u\to j} W_{X_u,j} e^{T_{X_u}(u)+T_{X_u^\star}(u)}e^{V_{j^\star}(u)-V_{X_u^\star}(u)}du}
\;\qqq^W_{I,i_0}(d\beta_I,da_{I^c}).
\end{align}
where in the previous formula, $(V_i(u))$ is constructed from the variables $(\beta_I, a_{I^c})$ instead of $(\bbb_I,A_{I^c})$, i.e. we have
$$
\begin{cases}
V_i(t)=T_i(t)+a_i, &\hbox {if $i\in I^c$},
\\
(H_{\beta}(e^{V^\star(t)}))_{I}=0. \;\;\; 
\end{cases}
$$


Let $$\theta (t)= e^{T^\star(t)}.$$ Then \eqref{theta_A} and the definition of $V(t)$ imply that 
\begin{equation}
\label{trid}
\theta^a_{I^c}(t)= (e^{T^\star(t)+a^\star})_{I^c}=(e^{V^\star(t)})_{I^c}.
\end{equation}
 Also,
\begin{eqnarray*}
&&
{\partial \over \partial t} \left(
\demi\left<\theta_I(t), \beta_I \theta_I(t)\right> +\demi \left<\theta^{a}_{I^c}(t),W_{I^c,I}\hat G_\beta W_{I,I^c} \theta^{a}_{I^c}(t)\right> +\demi \left<\theta^{a}_{I^c}(t), W_{I^c,I^c}\theta^a_{I^c}(t)\right>
\right)
\\
&
=&
\sum_{j, X_t\to j} W_{X_t,j} e^{T_{X_t}(t)+T_{X_t^\star}(t)}e^{V_{j^\star}(t)-V_{X_t^\star}(t)}.
\end{eqnarray*}
Indeed, we have
\begin{eqnarray*}
{\partial \over \partial t} 
\demi\left<\theta_I(t), \beta_I \theta_I(t)\right>
=\indic_{X_t\in I} e^{T_{X_t}(t)+T_{X_t}^\star(t)}\beta_{X_t}=\indic_{X_t\in I} e^{T_{X_t}(t)+T_{X_t}^\star(t)}\sum_{j, X_t\to j} W_{X_t,j} e^{V_{j^\star}(t)-V_{X_t^\star}(t)}
\end{eqnarray*}
since, for $i\in I$, $\beta_i=\sum_{j,i\to j} W_{i,j}e^{V_{j^\star}(t)-V_{i^\star}(t)}$, for all time $t\ge 0$. Besides, 
\begin{eqnarray*}
{\partial \over \partial t} 
\demi \left<\theta^{a}_{I^c}(t),W_{I^c,I}\hat G_\beta W_{I,I^c} \theta^{a}_{I^c}(t)\right>
&=&
\indic_{X_t\in I^c} \theta_{X^\star_t}^{a}(t) \sum_{j\in I, X_t\to j} W_{X_t,j} \left( \hat G_\beta W_{I,I^c} \theta^{a}_{I^c}(t)\right)_{j}
\\
&=&
\indic_{X_t\in I^c} e^{T_{X_t}(t)+T_{X^\star_t}(t)} \sum_{j\in I, X_t\to j} W_{X_t,j} e^{V_{j^\star}(t)-V_{X_t^\star}(t)}
\end{eqnarray*}
since by the definition of $V(t)$ in Theorem~\ref{mixing-Ib}~ii) we have, for $i\in I^c$ and $j\in I$,
\begin{equation*}
\theta_{i^\star}^{a}(t)=e^{V_{i}(t)}= e^{T_{i}(t)}e^{a_i}= e^{T_{i}(t)+T_{i^\star}(t)} e^{-V_{i^\star}(t)}
\end{equation*}
and, using  $H_\be(e^{V^\star(t)})_I=0$, or equivalently $W_{II^c}(e^{V^\star(t)})_{I^c}=\hat H_\beta(e^{V^\star(t)})_{I}$,  
\begin{equation}
\label{eq:pot}
\left( \hat G_\beta W_{I,I^c} \theta^{a}_{I^c}(t) \right)_{j}=\left( \hat G_\beta W_{I,I^c}(e^{V^\star(t)})_{I^c} \right)_{j}=e^{V_{j^\star}(t)}.
\end{equation}
Finally, for the last term we have
\begin{eqnarray*}
{\partial \over \partial t} 
\demi \left<\theta^{a}_{I^c}(t), W_{I^c,I^c}\theta^a_{I^c}(t)\right>
&=&
\indic_{X_t\in I^c} \theta_{X_t^\star}^{a}(t) \sum_{j\in I^c, X_t\to j} W_{X_t,j} \theta^a_{j}(t)
\\
&=&
\indic_{X_t\in I^c} e^{T_{X_t}(t)+T_{X^\star_t}(t)} \sum_{j\in I^c, X_t\to j} W_{X_t,j} e^{V_{j^\star}(t)-V_{X_t^\star}(t)},
\end{eqnarray*}
using \eqref{trid}.

Hence, recalling notation \eqref{W_check},
$
\check W_{I^c,I^c}= W_{I^c,I^c}+ W_{I^c,I}\hat G_\beta W_{I,I^c},
$
we have
\begin{eqnarray*}
&&\int_0^t \sum_{j, X_u\to j} W_{X_u,j} e^{T_{X_u}(u)+T_{X_u^\star}(u)}e^{V_{j^\star}(u)-V_{X_u^\star}(u)}du
\\
&=&\demi\left(\left<\theta_I(t), \beta_I \theta_I(t)\right> + \left<\theta^a_{I^c}(t), \check W_{I^c,I^c}\theta^a_{I^c}(t)\right>
-
\left<1_I, \beta_I 1_I\right>  - \left<e^{a}_{I^c}, \check W_{I^c,I^c}e^a_{I^c}\right>
\right).
\end{eqnarray*}
From the definition of $\qqq^{W,\theta,0}_{I}(d\beta_I,da_{I^c})$, for $\theta\in (0,\iy)^V$, $\eta=0$, we see that
\begin{eqnarray*}\label{FQ-tildeF}
&&e^{-\demi\left<\theta, W\theta\right>} \qqq^{W,\theta,0}_{I}(d\beta_I,da_{I^c})
\\
&=&
\frac{\indic_{\hat H_\beta >0}}{\sqrt{2\pi}^{\vert\aaa_{I^c}\vert+\vert\sss_I\vert}} 
\frac{\prod_{i\in V_0\cap I}\theta_i}{\sqrt{\det \hat H_{\beta}}} 
e^{-\demi\left<\theta_I, \beta_I \theta_I\right> -\demi \left<\theta^{a}_{I^c},W_{I^c,I}\hat G_\beta W_{I,I^c} \theta^{a}_{I^c}\right> -\demi \left<\theta^{a}_{I^c}, W_{I^c,I^c}\theta^a_{I^c}\right>
}
da_{I^c} d\beta_I
\\
&=&
\frac{\indic_{\hat H_\beta >0}}{\sqrt{2\pi}^{\vert\aaa_{I^c}\vert+\vert\sss_I\vert}} 
\frac{\prod_{i\in V_0\cap I}\theta_i}{\sqrt{\det \hat H_{\beta}}} 
e^{-\demi\left<\theta_I, \beta_I \theta_I\right> -\demi \left<\theta^{a}_{I^c},\check W_{I^c,I} \theta^{a}_{I^c}\right>
}
da_{I^c} d\beta_I.
\end{eqnarray*}
Hence 
\begin{align}
\label{exp_int}
&\exp\left({-\int_0^t \sum_{j, X_u\to j} W_{X_u,j} e^{T_{X_u}(u)+T_{X_u^\star}(u)}e^{V_{j^\star}(u)-V_{X_u^\star}(u)}du}\right)
e^{-\demi\left<1, W 1\right>}\qqq^W_{I}(d\beta_I,da_{I^c})
\\
\nonumber
&=
{1\over \prod_{i\in V_0\cap I} \theta_i(t)}e^{-\demi\left<\theta(t), W\theta(t)\right>}\qqq^{W,\theta(t)}_{I}(d\beta_I,da_{I^c}).
\end{align}
Besides, we have the following simple proposition.
\begin{proposition}
\label{proptrid}
For all $t \ge 0$,
$$
\qqq^{W,\theta(t)}_{I,i_0}(d\beta_I,da_{I^c})=e^{V_{i_0^\star}(t)}\qqq^{W,\theta(t)}_{I}(d\beta_I,da_{I^c})
$$
\end{proposition}
\begin{proof}
Using \eqref{trid} and \eqref{eq:pot}, we have $(e^{V^\star(t)})_{I^c}=(\theta^a(t))_{I^c}$ and 
$
e^{V^\star(t)}_{I}=\hat G_\beta \hat\eta_I^a(t)
$
with  $\hat\eta_I^a(t)=W_{I,I^c}(\theta^a(t))_{I^c}$ corresponding to \eqref{eta_hat} with $\theta(t)$ and $\eta=0$. By Definition-Proposition~\ref{def_nu_i0}, this yields the identity.
\end{proof}
Set $$\tilde F^{W,\theta}_{i_0}=e^{-\demi \left<\theta, W, \theta \right>} F^{W,\theta}_{i_0}.$$
Proposition \ref{proptrid} applied at times $0$ and $t>0$, combined with \eqref{exp_int}, implies that 
\begin{eqnarray*}
&&
e^{-\int_0^t \sum_{j, X_u\to j} W_{X_u,j} e^{T_{X_u}(u)+T_{X_u^\star}(u)}du}
e^{V_{j_0^\star}(t)-V_{i_0^\star}(0)-\sum_{i\in V_0\cap I^c} T_i(t)}{1\over F^W_{i_0}}\qqq^W_{I,i_0}(d\beta_I,da_{I^c})
\\
&=&
e^{-\int_0^t \sum_{j, X_u\to j} W_{X_u,j} e^{T_{X_u}(u)+T_{X_u^\star}(u)}du}
e^{V_{j_0^\star}(t)-\sum_{i\in V_0\cap I^c} T_i(t)}{e^{-\demi \left<1,W 1\right>} \over \tilde F^W_{i_0}}\qqq^W_{I}(d\beta_I,da_{I^c})
\\
&=&
{e^{-\sum_{i\in V_0\cap I^c}T_i(t)}\over \prod_{i\in V_0\cap I} \theta_i(t)}{e^{-\demi \left<\theta(t),W\theta(t)\right>}  \over \tilde F^W_{i_0}}
e^{V_{j_0^\star}(t)} \qqq^{W,\theta(t)}_{I}(d\beta_I,da_{I^c})
\\
&=&
e^{-\sum_{i\in V_0} T_i(t)}{\tilde F^{W, \theta(t)}_{j_0}  \over \tilde F^W_{i_0}}
{1\over F^{W, \theta(t)}_{j_0}}\qqq^{W,\theta(t)}_{I,j_0}(d\beta_I,da_{I^c}).
\end{eqnarray*}

Hence, using \eqref{proba-traject-2-integree} and integrating the last expression on $\sss_I\times\aaa_I$, we obtain that under $\tilde \P_{I,i_0}^W$ of the process in Theorem \ref{mixing-Ia}, the probability that the process $X$ at time $t$ has performed $n$ jumps in infinitesimal time intervals $[t_i, t_i+dt_i)$, $0<t_1<\cdots <t_n<t$, following the trajectory
$\sigma_{0}=i_0, \sigma_1, \ldots, \sigma_n=j_0$, is equal to
\begin{eqnarray}\label{prob_sigma}
\left(\prod_{l=1}^{n} W_{\sigma_{l-1}, \sigma_l} e^{T_{\sigma_{l-1}}(t_l)+T_{\sigma_{l-1}^\star}(t_l)} dt_l\right)e^{-\sum_{i\in V_0} T_i(t)}
{\tilde F^{W,\theta(t)}_{j_0} \over \tilde F^W_{i_0}}.
\end{eqnarray}
From the definitions of $\theta(t)$ and of $F^{W,\theta(t)}_{j_0}$, we deduce that 
$$\tilde F^{W,\theta(t)}_{j_0}=\theta_{j_0}(t)\tilde F^{W^{T(t)}}_{j_0}=e^{T_{j_0^\star}(t)}\tilde F^{W^{T(t)}}_{j_0}.
$$ 
Hence the expression \eqref{prob_sigma} is 
equal to the expression~\eqref{density-*} in Part~\ref{partI}, which means that the probability of trajectories of $(X_t)$ are the same under $\overline \P^{W}_{i_0}$ and $\tilde \P^{W}_{I,i_0}$, and thus  $(X_t)$ has the same law under $\overline \P^{W}_{i_0}$ and $\tilde \P^{W}_{I,i_0}$.

Let us now deduce i) of Theorem~\ref{mixing-Ib}. Under $\overline \P_{i_0}^W$, the initial local time $(A_i)_{i\in V}$ is distributed according to 
$\nu_{\aaa,i_0}^W$ and $(U_i)_{i\in V}$ is defined in Theorem~\ref{main}~i) by
\begin{align}
\label{U-A-T}
U_i:=\lim_{t \to \infty} A_i+T_i(t)-t/N.
\end{align}
Note that the almost sure convergence of the previous limit is not a consequence of Theorem~\ref{main} but of a more elementary Lemma~\ref{Convergence} of Part~\ref{partI}. Besides, $U\in \uuu_{0}^W$, which implies
\begin{eqnarray}\label{U-pU}
U_i:=\lim_{t \to \infty}  p_{\uuu_0^W}(A+T(t)-t/N)=\lim_{t \to \infty}  p_{\uuu_0^W}(T(t)-t/N)
\end{eqnarray}
since $A\in \aaa$ (see Lemma~\ref{lem:proj} of Part~\ref{partI}). Hence, a.s. $U$ can be retrieved from the infinite trajectory of the process $(X(t))$ (i.e. it is a.s. equal to a measurable function of the path). Moreover, consider the time changed process $Z_s=X_{C^{-1}(s)}$ defined in Theorem~\ref{main}~ii). By an easy computation, conditionally on $A$, the jump rate of $Z$ at time $s$ from $i$ to $j$ is
$$
W_{i,j}e^{T_{j^\star}(C^{-1}(s))-T_{i^\star}(C^{-1}(s))+A_{j^\star}-A_{i^\star}}.
$$
By \eqref{U-pU} it means that 
the asymptotic jump rate of $Z$ from $i$ to $j$ is a.s.
\begin{align}\label{jumping_rate_bis}
W_{i,j}e^{U_{j^\star}-U_{i^\star}}.
\end{align}
Note that, since the graph is strongly recurrent and the \sVRJP visits infinitely often each vertex, the asymptotic jump rate of $Z$ is a measurable function of the path of $(X_t)_{t\ge 0}$.
This property will be useful to identify the limit of $(V_i(t))$ in terms of $(U_i)$.
From \eqref{U-A-T}, we also obtain
\begin{eqnarray}\label{retrieve_A}
A=\demi(U-U^\star)-\lim_{t\to \infty} \demi(T(t)-T^\star(t)).
\end{eqnarray}
It means that $(A_i)$ can also a.s. be retrieved from the trajectory of $(X_t)$. Besides, in Theorem~\ref{mixing-Ib}~i), we have for all $i\in I$
\begin{eqnarray}\label{retrieve_B}
\bbb_i=\sum_{j, i\to j} W_{i,j}e^{U_{j^\star}-U_{i^\star}}.
\end{eqnarray}

Consider now  $(X_t)$ under the law $\tilde\P^{W}_{I,i_0}$ as defined at the beginning of 
the section. The process is defined as a mixture of processes with $(\bbb_I,A_{I^c})$ distributed according to $\qqq^{W}_{I,i_0}$. The aim is to prove that  $(\bbb_I,A_{I^c})$ can be retrieved a.s. by the same measurable functionals of the path of $(X_t)$ as in \eqref{retrieve_A} and \eqref{retrieve_B}.  Let $U$ be defined from $(X_t)_{t\ge 0}$ by the last term of formula \eqref{U-pU}: as we have proved that $(X_t)$ has the same law under $\tilde\P^{W}_{I,i_0}$ and $\overline\P^{W}_{i_0}$, clearly the limit exists and is finite a.s. since it is the case under the law $\overline \P_{i_0}^W$. 
By definition, for all $i, j\in V$ such that $i\to j$, the jump rate of the process $(X_t)$ from $i$ to $j$ is equal to
$$
e^{T_{i}(t)+T_{i^\star}(t)}e^{V_{j^\star}(t)-V_{i^\star}(t)},
$$
hence for the time changed process $Z$, at time $s$ the jump rate is equal to 
$$
e^{V_{j^\star}(C^{-1}(s))-V_{i^\star}(C^{-1}(s))}.
$$ 
Since the asymptotic jump rate is a measurable function of the path, $(X_t)$ has the same law under $\tilde\P^{W}_{I,i_0}$ and $\overline\P^{W}_{i_0}$, and by \eqref{jumping_rate_bis}, this implies that
$$
\lim_{t\to\infty} V_{j^\star}(t)-V_{i^\star}(t)= U_{j^\star}-U_{i^\star}.
$$
Since the graph is strongly connected, we deduce that the previous equality is true for all $i,j\in V$. Since $V_i(t)=T_{i}(t)+A_i$  for all $i\in I^c$,  we have
$$
A_i=\lim_{t\to \infty} \left(\demi(V_i(t)-V_{i^\star}(t))-\demi(T_i(t)-T_{i^\star}(t))\right)= \demi(U_i-U_{i^\star})-\lim_{t\to\infty} \demi(T_i(t)-T_{i^\star}(t)), 
$$
which matches with \eqref{retrieve_A} on $I^c$. Besides, by definition, we have $\bbb_i=\sum_{j, i\to j}W_{i,j} e^{V_{j^\star}(t)-V_{i^\star}(t)}$, for all $i\in I$ and for all time $t\ge 0$. Letting $t$ go to infinity, we obtain $\bbb_i=\sum_{j, i\to j}W_{i,j} e^{U_{j^\star}-U_{i^\star}}$, for all $i\in I$ which matches \eqref{retrieve_B}. This concludes the proof.


\section{Proof of Lemma~\ref{diff}, Corollary~\ref{thm:identity_0_a} and Corollary~\ref{thm:identity_0_b}}
\subsection{Proof of Lemma~\ref{diff}
}
\label{sec:diff}

Fix $i_0\in V$, and let $I=V\setminus\{i_0,i_0^\star\}$. We first prove that $\Xi_{i_0}(\mathcal{U}_0^W)\subset \ddd_{i_0}$. Given $u\in\mathcal{U}_0^W$, we denote by $\be=(\beta_i)_{i\in V}$ the vector 
defined by 
\begin{eqnarray}\label{eq:be}
\beta_i=\sum_{j, i\to j} W_{i,j} e^{u_{j^\star}-u_{i^\star}}, \;\;\; \forall i \in V,
\end{eqnarray}
so that $\beta_I=\Xi_{i_0}(u)$, see \eqref{Xi}. Then $u\in\mathcal{U}_0^W$ implies $\be_{i}=\be_{i^\star}$ for all $i\in V$, so that $\beta_I\in \sss_I$. Furthermore, 
$H_\be e^{u^\star}=0$, where $e^{u^\star}=(e^{u_{i^\star}})_{i\in V}$. Let 
$$\hat H_\beta= \beta_I-W_{I,I},
$$ 
be defined as in \eqref{Gbe}: using the block matrix decomposition \eqref{block}, we have
\begin{equation}
\label{eq:res}
\left(\hat H_\be e^{u^\star}\right)_{I}=W_{I,I^c}\left(e^{u^\star}\right)_{I^c}.
\end{equation}
Assumption (2) of Proposition \ref{prop-M} in Appendix~\ref{Append_M_matrices} is satisfied with $A=\hat H_\be$, $x=\left( e^{u^\star}\right)_{I}$ and $y=W_{I,I^c}\left(e^{u^\star}\right)_{I^c}$, using that $\ggg$ is strongly connected: therefore, $\hat H_\be$ is a non-singular $M$-matrix in the sense of Definition \ref{def-M}, i.e. $\hat H_\be>0$. Hence $\beta_I\in \ddd_{i_0}$.

Next, our goal is to show that, given $(\be_i)_{i\in I}\in\mathcal{D}_{i_0}$, there exists a unique $u\in\mathcal{U}_0^W$ such that $(\beta_i)_{i\in I}=\Xi_{i_0}(u)$, 
and that $u$ is a differentiable function of $(\be_i)_{i\in I}$.
First assume that such an $u\in\mathcal{U}_0^W$ exists, and define $\beta$ by \eqref{eq:be}. As before we have $\hat G_\beta$, $\check W$ and $\check H_\beta$ as in \eqref{Gbe} and \eqref{Hcheck}. Now, using \eqref{eq-hbeta-productof3},  $H_\be e^{u^\star}=0$ is equivalent to
\begin{equation}
\begin{pmatrix}
   \hat H_\beta  & -W_{I,I^c}\\ 0 & \check H_\beta
    \end{pmatrix}
    \begin{pmatrix}
 \left(e^{u^\star}\right)_I    \\ \left(e^{u^\star}\right)_{I^c}
    \end{pmatrix}
   =0.
\end{equation}
Therefore
\begin{equation}
\label{chbe}
\check H_\beta \left(e^{u^\star}\right)_{I^c}=0,
\end{equation} 
which implies 
\begin{equation}
\label{detchbe}
\det \check H_\beta=0.
\end{equation}

Let us first assume that $i_0$ is self-dual: then \eqref{detchbe} is equivalent to
\begin{equation}
\label{bee}
\be_{i_0}=\check W_{i_0,i_0}
\end{equation}
and, using \eqref{eq:res}, $e^{u^\star}$ is given by
\begin{equation}
\label{eusd}
 \begin{cases}
e^{u^\star_{i_0}}=C
\\
\left(e^{u^\star}\right)_I=  \hat G_\be W_{I,i_0} e^{u^\star_{i_0}}
 \end{cases},
\end{equation}
where the constant $C>0$ is uniquely determined by $\sum_{i\in V}u_i=0$.

Let us now assume that $i_0$ is not self-dual, i.e. $i_0\ne i_0^\star$. Then \eqref{detchbe} implies
\begin{equation}
(\be_{i_0}-\check W_{i_0,i_0})^2=\check W_{i_0,i_0^\star}\check W_{i_0^\star,i_0}
\end{equation}
On the other hand, using $e^{u_{i_0}}$, $e^{u_{i_0^\star}}$, $\check W_{i_0,i_0^\star}$ $>0$ and \eqref{chbe}, it holds that $\be_{i_0}>\check W_{i_0,i_0}$ and, subsequently
\begin{equation}
\label{bee2}
\be_{i_0}=\check W_{i_0,i_0}+\sqrt{\check W_{i_0,i_0^\star}\check W_{i_0^\star,i_0}}.
\end{equation}
Therefore $\be_{i_0}=\be_{i_0^\star}$ is uniquely determined from $(\be_i)_{i\in I}$ by \eqref{bee2}.
On the other hand, \eqref{chbe} is equivalent to the combination of \eqref{bee2} and 
\begin{equation}\label{expu_W}
e^{u_{i_0^\star}-u_{i_0}}=\sqrt{\frac{\check W_{i_0,i_0^\star}}{\check W_{i_0^\star,i_0}}}.
\end{equation}
Therefore, using \eqref{eq:res},  $e^{u^\star}$ is given by
\begin{equation}
\label{eunsd}
 \begin{cases}
e^{u^\star_{i_0}}= C \sqrt{\check W_{i_0,i_0^\star}}  
\\
e^{u^\star_{i_0^\star}}=e^{u_{i_0}}= C \sqrt{\check W_{i^\star_0,i_0}}
\\
\left(e^{u^\star}\right)_I= \hat G_\be W_{I,I^c}\left(e^{u^\star}\right)_{I^c}
 \end{cases}
\end{equation}
where again the constant $C>0$ is uniquely determined by $\sum_{i\in V}u_i=0$. Remark also that $\hat G_\be$ and $\check W_{i_0,i_0^\star}$, $\check W_{i^\star_0,i_0}$, are continuous functions of $\beta_I$ on the domain $\ddd_{i_0}$.

In summary, given $(\be_i)_{i\in I}\in\mathcal{D}_{i_0}$,  if $i_0$ self-dual (respectively if $i_0\ne i_0^\star$), then $e^{u}$ is uniquely determined by \eqref{eusd} (resp. by \eqref{eunsd}) which is $C^1$ for $\beta_I\in \ddd_{i_0}$. Besides,
 $\be_{i_0}$ is equal to \eqref{bee}  (resp. by \eqref{bee2}). Conversely, those definitions of  $e^u$ ensure that $H_\be e^{u^\star}=0$, which enables us to conclude the proof of Lemma \ref{diff}.

\subsection{Proof of Corollary~\ref{thm:identity_0_a} and Corollary~\ref{thm:identity_0_b}}
\label{sec:identity}
By Proposition~\ref{conditionning}, since $i_0\in I^c$, we know that
$$
\mathcal{Q}_{I,i_0}^{W}(d\beta_I,da_{I^c})=
\nu_{\sss_I}^{W_{I,I},1_I,\hat \eta_I}(d\beta_I) \nu_{\aaa_{I^c},i_0}^{\check W_{I^c,I^c}}(da_{I^c})
$$
with $\hat \eta_I=W_{I,I^c} 1_{I^c}$. Assume first that $i_0=i^\star_0$, then  $\nu_{\aaa_{I^c},i_0}^{\check W_{I^c,I^c}}=1$ is the measure constant equal to 1, and the $\beta_I$ marginal of $\mathcal{Q}_{I,i_0}^{W}(d\beta_I,da_{I^c})$ is $\nu_{\sss_I}^{W_{I,I},1_I,\hat \eta_I}(d\beta_I)$.  With this notation we have, using \eqref{detchbe},
\begin{eqnarray}
\nonumber
\;\;\;\;\;\;\;\;\;\; &&\nu_{\sss_I}^{W_{I,I},1_I,\hat \eta_I}(d\beta_I)
\\
\nonumber
&=&{1\over \sqrt{2\pi}^{\vert\sss_I\vert}\sqrt{\vert \hat H_\beta\vert}}e^{-\demi\sum_{i\in I} \beta_i +\demi\left<1_I, W_{I,I} 1_I\right>-\demi\left<1_{I^c},W_{I^c,I} \hat G_\beta W_{I^c,I} 1_{I^c}\right>+\left<1_I,W_{I,I^c}1_{I^c}\right>} d\beta_I
\\
\nonumber&=&
{1\over \sqrt{2\pi}^{\vert\sss_I\vert}\sqrt{\vert \hat H_\beta\vert}}e^{-\demi\sum_{i\in I} \beta_i +\demi\left<1, W 1\right>-\demi \check W_{i_0,i_0}} d\beta_I
\\
\label{nu-i0}
&=&
{1\over \sqrt{2\pi}^{\vert\sss_I\vert}\sqrt{\vert \hat H_\beta\vert}}e^{-\demi\sum_{i\in V} \beta_i +\demi\left<1, W 1\right>} d\beta_I.
\end{eqnarray}

When $i_0\neq i_0^\star$,  $I^c=\{i_0,i_0^\star\}$, the integral of $\nu_{\aaa_{I^c},i_0}^{\check W_{I^c,I^c}}(da_{I^c})$ can be computed explicitly. Indeed
\begin{eqnarray*}
&&\int_{\aaa_{I^c}} \nu_{\aaa_{I^c},i_0}^{\check W_{I^c,I^c}}(da_{I^c})
={e^{\demi(\check W_{i_0,i_0^\star}+\check W_{i_0^\star,i_0})}\over \sqrt{2\pi}}
\int_{-\infty}^{+\infty} e^{-a}\exp\left(-\demi  W_{i_0,i_0^\star}e^{2a}-\demi \check W_{i^\star_0,i_0}e^{-2a}\right)
da
\\
&=&
{e^{\demi(\check W_{i_0,i_0^\star}+\check W_{i_0^\star,i_0})}\over \sqrt{2\pi}}
\int_{-\infty}^{+\infty} e^{-a} \exp\left(-\demi\sqrt{\check W_{i_0,i_0^\star} \check W_{i^\star_0,i_0}} \left( \sqrt{{\check W_{i_0,i_0^\star}\over  \check W_{i^\star_0,i_0}}}e^{2a} + \sqrt{{\check W_{i_0^\star,i_0}\over  \check W_{i_0,i_0^\star}}}e^{-2a}\right)\right)
da
\\
&=&
{1\over \sqrt{2\pi}} \left({\check W_{i_0,i_0^\star}\over  \check W_{i^\star_0,i_0}}\right)^{1\over 4}
e^{\demi(\check W_{i_0,i_0^\star}+\check W_{i_0^\star,i_0})}
\int_0^\infty \cosh(y/2)\exp\left({-\sqrt{\check W_{i_0,i_0^\star}\check W_{i_0^\star,i_0}}\cosh(y)}\right) dy
\\
&=&{1\over 2} {1\over \sqrt{\check W_{i^\star_0,i_0}}}e^{\demi(\check W_{i_0,i_0^\star}+\check W_{i_0^\star,i_0})-\sqrt{\check W_{i_0,i_0^\star}\check W_{i_0^\star,i_0}}},
 \end{eqnarray*}
where we changed to variable $y$ such that $e^{y}=\sqrt{{\check W_{i_0,i_0^\star}\over  \check W_{i^\star_0,i_0}}}e^{2a}$ in the second equality and where the integral in the third line is the modified Bessel function $K_{\demi}(\sqrt{\check W_{i_0,i_0^\star}\check W_{i_0^\star,i_0}})$ (recall that $K_\demi$ has explicit value $K_\demi(z)=\sqrt{\pi\over 2z}e^{-z}$).

From $\beta_I$ we can define $\beta_{i_0}=\beta_{i_0^\star}$ by \eqref{bee2}. With this notation we deduce from Lemma \ref{lem_identities} that
 \begin{align}
\label{nu-i02}
& 
\int_{\aaa_{I^c}} \mathcal{Q}_{I,i_0}^{W}(d\beta_I,da_{I^c})
\\
\nonumber&=
{e^{\demi(\check W_{i_0,i_0^\star}+\check W_{i_0^\star,i_0})-\sqrt{\check W_{i_0,i_0^\star}\check W_{i_0^\star,i_0}}}\over 2 \sqrt{\check W_{i^\star_0,i_0}} \sqrt{2\pi}^{\vert\sss_I\vert}\sqrt{\vert \hat H_\beta\vert}}e^{-\demi\sum_{i\in I} \beta_i +\demi\left<1_I, W_{I,I} 1_I\right>-\demi\left<1_{I^c},W_{I^c,I} \hat G_\beta W_{I^c,I} 1_{I^c}\right>+\left<1_I,W_{I,I^c}1_{I^c}\right>} d\beta_I
\\
\nonumber&= 
{e^{\demi(\check W_{i_0,i_0^\star}+\check W_{i_0^\star,i_0})-\sqrt{\check W_{i_0,i_0^\star}\check W_{i_0^\star,i_0}}}\over 2 \sqrt{\check W_{i^\star_0,i_0}} \sqrt{2\pi}^{\vert\sss_I\vert}\sqrt{\vert \hat H_\beta\vert}}e^{-\demi\sum_{i\in I} \beta_i +\demi\left<1, W 1\right>-\demi\left<1_{I^c},\check W_{I^c,I^c} 1_{I^c}\right>} d\beta_I
\\
\nonumber
&=
{1\over 2 \sqrt{\check W_{i^\star_0,i_0}} \sqrt{2\pi}^{\vert\sss_I\vert}\sqrt{\vert \hat H_\beta\vert}}e^{-\demi\sum_{i\in V} \beta_i +\demi\left<1, W 1\right>} d\beta_I.
\end{align}
Corollary~\ref{thm:identity_0_a} and Corollary~\ref{thm:identity_0_b} are therefore  consequences of the following proposition.

\begin{proposition}
\label{prop:identity}
The mixing measure $\mu_{i_0}^{W}$ of $(U_i)_{i\in V}$ for the randomized  \sVRJP, starting from $i_0$ with conductances $W$,  is equal to the pullback measure by $\Xi_{i_0}$ of the measure \eqref{nu-i0} when $i_0=i_0^\star$ (resp. \eqref{nu-i02} when $i_0\neq i_0^\star$).
\end{proposition}


\begin{proof} We start by the computation of the Jacobian.
\begin{lemma}\label{lem_chgt_var}
By the change of variables 
$u=\Xi_{i_0}^{-1}(\beta_I)$, the measure $d\beta_I$ on $\ddd_{i_0}$ is transformed to the measure on $\uuu_0^W$ given by
$$
{d_{i_0} \sqrt{2}^{-\vert V_1\vert}}\left( \prod_{i\in \tilde I} e^{-u_i-u_{i^\star}}\right){\sqrt{\vert V\vert}} {D(W^u)\over \det_\aaa({-}K^u)} \sigma_0^W(du),
$$
where $d_{i_0}=2$ if $i_0\neq i_{0}^\star$ and $d_{i_0}=1$ if $i_0=i_0^\star$.
\end{lemma}
\begin{proof}
We need to remind a few definitions and properties from Part~\ref{partI}~Section~\ref{s_limiting}. The space $\sss_0$ is defined by $\sss_0=\{
x\in \sss, \; \sum_{i\in V} x_i=0\}$, and $\ppp_{\sss_0}:\uuu_0^W\mapsto \sss_0$ is the projection $\ppp_{\sss_0}(u)=\demi(u+u^\star)$. 
By 
Lemma~\ref{lem:proj}, for any $u\in \sss_0$ there is a unique $a\in \aaa$ such that $u+a\in \uuu_0^W$. Hence, $\ppp_{\sss_0}$ is invertible and we denote by  
$$\xi: \sss_0\mapsto\uuu_0^W$$ 
its inverse.
Remind that the volume measure on $\sigma_{\mathcal{U}_0^W}(du)$ is defined as the pull-back of the euclidean measure $d\lambda_{\sss_0}$ on $\sss_0$ by the orthogonal projection $\ppp_{\sss_0}$, i.e. $\sigma_{\mathcal{U}_0^W}(du)= \ppp_{\sss_0}^\star(d\lambda_{\sss_0})$ (see Part~\ref{partI}, \eqref{volume}). The hyperplane tangent to $\uuu_0^W$ at the point $u\in \uuu_0^W$ can be parametrized as $(\Ktrans^u)^{-1}(\sss_0)$, see Part~\ref{partI} Lemma~\ref{tangent}.  

It will be convenient to define $\Xi:\uuu_0^W\mapsto \sss$ as the function defined by
$$
\Xi(u)_i=\sum_{j, i\to j} W_{i,j} e^{u_{j^\star}-u_{i^\star}},\;\;\; i\in V,
$$
which extends the definition \eqref{Xi} so that $\Xi(u)_I=\Xi_{i_0}(u)$. Let $\tilde I=(V_0\cup V_1)\cap I$ be the set obtained from $I$ by choosing one representative in each dual pair of points. Let us also denote by $\iota$ be the projection from $\R^V$ to $\R^{\tilde I}$. Hence, since on $\sss_0$ we chose the Euclidian volume measure and on $\ddd_{i_0}$ the measure $\prod_{i\in \tilde I} d\beta_{i}$, we need to compute the determinant of the differential of the application from $\sss_0$ to $\R^{\tilde I}$ given by
\begin{align}\label{chgt-var-2}
\iota \circ \Xi \circ \xi,
\end{align}
where on $\sss_0$ we take an orthonormal base 
and the canonical base on $\R^{\tilde I}$.  

Remind the definition of the matrix $K^u$ in \eqref{Ku}. Let us compute the differential of $\Xi$. We have for $i\in V$,
\begin{align*}
\frac{\partial\Xi(u)_{i^\star}}{\partial u_j}&=W_{ji}e^{u_j-u_i}=e^{-u_i-u_{i^\star}} K_{j,i}^u, \; \tx{ if }j\ne i\\
\frac{\partial\Xi(u)_{i^\star}}{\partial u_i}&=-\sum_{j:\,j\to i}W_{ji}e^{u_j-u_i}=e^{-u_i-u_{i^\star}} K_{i,i}^u
\end{align*}
where in the last equality we used that $u\in \uuu_0^W$, hence that $\sum_{j, j\to i} W_{j,i}^u=\sum_{j,i\to j} W_{i,j}^u=-K_{i,i}^u$.
It follows that
$$
D \Xi^\star(du)= e^{-u-u^\star}{}^t\! K^u(du), \;\;\; \forall du\in T_{u}(\uuu_0^W),
$$
where $e^{-u-u^\star}$ is the operator of multiplication by $ e^{-u_i-u_{i^\star}}$ at each point $i\in V$. From Part~\ref{partI} Lemma~\ref{tangent}, we know that $\Ktrans^u$ is bijective from $T_u(\uuu_0^W)$ to $\sss_0$. Besides, for $s\in \sss_0$, $(d\xi_s)^{-1}=(\ppp_{\sss_0})_{T_{\xi(s)}(\uuu_0^W)}$. Hence, $\Ktrans\circ d\xi_s$ is bijective on $\sss_0$ with determinant
$$
{1\over {\det}_{\sss_0}\left((\Ktrans^{\xi(s)})^{-1}\right)}.
$$ 
To compute the Jacobian of the transformation \eqref{chgt-var-2}, we still have to compute the determinant of
$$
d\iota\circ (e^{-u-u^\star}):\sss_0\mapsto \R^{\tilde I},
$$
with the choice of base specified earlier, which is equal to (see below for the proof)
\begin{align}\label{change_volume_tau}
\left( \prod_{i\in \tilde I} e^{-u_i-u_{i^\star}}\right){1\over \sqrt{\vert V\vert}} {{d_{i_0} \over \sqrt{2}^{\vert V_1\vert }}}.
\end{align}
Overall, by the change of variable given by $\beta_I=\Xi_{i_0}(u)$, we obtain that
$$
d\beta_I= \left( \prod_{i\in \tilde I} e^{-u_i-u_{i^\star}}\right){1\over \sqrt{\vert V\vert}} {{d_{i_0} \over \sqrt{2}^{\vert V_1\vert }}} \left({\det}_{\sss_0}\left((\Ktrans^{u})^{-1}\right)\right)^{-1} \sigma_0^W(du).
$$
We conclude the proof of the lemma using elementary computations on determinants given in 
formula \eqref{ratio_det} and \eqref{N_D_K}.
\end{proof}
\begin{proof}[Proof of formula \eqref{change_volume_tau}]
For $i\in V$, we set $d_i=1$ if $i=i^\star$ and $d_i=2$ if $i\neq i^\star$. We have that  $d\iota\circ (e^{-u-u^\star})=(e^{-u-u^\star})\circ d\iota$ so that we have to compute the Jacobian of the change of variable given by $\iota$ from $\sss_0$ with the Euclidian volume measure to $\R^{\tilde I}$ with the measure $\prod_{i\in \tilde I} d\beta_i$. On $\sss_0$, we choose the base $(e_i)_{i\in \tilde I}$ given by
$$
e_i=\begin{cases} 
(\delta_i+\delta_{i^\star})-(\delta_{i_0}+\delta_{i_0^\star}),&\hbox{ if $i\neq i^\star$}
\\
\delta_i-\demi(\delta_{i_0}+\delta_{i_0^\star}),&\hbox{ if $i= i^\star$}
\end{cases}
$$
With the basis $(e_i)_{i\in \tilde I}$ on $\sss_0$ and the canonical base on $\R^{\tilde I}$, $\iota$ is represented by the identity matrix. Besides, by a simple computation, we have
$$
(e_i,e_j)= \indic_{i=j}d_i+{d_id_j\over d_{i_0}},
$$
so that the determinant of the Gramm matrix of the base $(e_i)_{i\in \tilde I}$ is given by
$$
\det\left((e_i,e_j)\right)_{i,j\in \tilde I}= \left(\prod_{i\in \tilde I}d_i\right)\left(1+\sum_{i\in \tilde I} {d_i\over d_{i_0}}\right)=
{1\over d_{i_0}^2}\left(\prod_{i\in \tilde V} d_i\right)\left( d_{i_0}+\sum_{i\in \tilde I} d_i\right)= {2^{\vert V_1\vert}\vert V\vert \over d_{i_0}^2}.
$$
\end{proof}
Let us denote by $\nu_{I,i_0}^W(d\beta_I)$ the measures which appear in \eqref{nu-i0} and \eqref{nu-i02}, depending on whether $i_0=i_0^\star$ or $i_0\neq i_0^\star$. 
Remind that, if $u\in \uuu_0^W$ and $\beta_I\in \ddd_{i_0}$ are related by $\beta_I=\Xi_{i_0}(u)$, then $u$ can be computed with \eqref{eusd} and \eqref{eunsd}, and that $\beta:=\Xi(u)$ can be computed on $i_0,i_0^\star$ from $\beta_I$ by \eqref{bee} and \eqref{bee2}. 
\begin{remark}\label{rk-nu_I_i0}
The law $\nu_{I,i_0}^W(d\beta_I)$ is in fact the law of the asymptotic jump rate $(\bbb_i)_{i\in V}$ of the \sVRJP on the full set of vertices, mentionned in Remark~\ref{rk:identity_0_b}. It is an easy consequence of the current proof.
\end{remark}

If $i_0\neq i_0^\star$, using Schur complement, we can write that
$$
\det \left( H_\beta\right)_{V\setminus\{i_0\},V\setminus\{i_0\}}= \left(\beta_{i^\star_0}-W_{i_0^\star}\hat G_\beta W_{I,i_0^\star}\right) \det(\hat H_\beta)=\sqrt{W_{i_0^\star,i_0}W_{i_0,i_0^\star}}\det\left(\hat H_\beta\right),
$$
where we use \eqref{bee2} in the last equality. Besides, by factorizing $e^{-u_i-u_i^\star}$ on each line we deduce
$$
\det \left( H_\beta\right)_{V\setminus\{i_0\},V\setminus\{i_0\}}= \left(\prod_{i\neq i_0}e^{-u_i-u_{i^\star}}\right) D(W^u).
$$
Combining with \eqref{nu-i02} we obtain that with the notation $u=\Xi^{-1}(\beta)$ as before,
\begin{align*}
\nu_{I,i_0}^W(d\beta_I)=&{1\over 2} {1\over \sqrt{2\pi}^{\vert\sss_I\vert}}
\left({W_{i_0,i_0^\star}\over W_{i_0^\star,i_0}}\right)^{{1\over 4}}
{\left(\prod_{i\neq i_0}e^{u_i+u_{i^\star}}\right)^{\demi}\over \sqrt{D(W^u)}}
e^{-\demi\sum_{i\to j} W_{i,j} (e^{u_{j^\star}-u_{i^\star}}-1) } d\beta_I
\\
=&
{1\over 2} {1\over \sqrt{2\pi}^{\vert\sss_I\vert}}
e^{u_{i_0^\star}}
{\left(\prod_{i\in I}e^{u_i+u_{i^\star}}\right)^{\demi}\over \sqrt{D(W^u)}}
e^{-\demi\sum_{i\to j} W_{i,j} (e^{u_{j^\star}-u_{i^\star}}-1) } d\beta_I,
\end{align*}
using \eqref{expu_W} in the last equality. Together with Lemma~\ref{lem_chgt_var}, this implies that by the change of variable $u=\Xi_{i_0}^{-1}(\beta_I)$, we have the relation
\begin{align*}
\nu_{I,i_0}^W(d\beta_I)=&
{\sqrt{\vert V\vert}\over \sqrt{2\pi}^{\vert\sss_I\vert}}\sqrt{2}^{-\vert V_1\vert}
e^{u_{i_0^\star}}
\left(\prod_{i\in V_0}e^{-u_i}\right)
e^{-\demi\sum_{i\to j} W_{i,j} (e^{u_{j^\star}-u_{i^\star}}-1) } {\sqrt{D(W^u)}\over \det_\aaa(-K^u)} \sigma_0^W(du).
\end{align*}
This concludes the proof when $i_0\neq i_0^\star$. The computation is similar and even simpler in the case $i_0=i_0^\star$.
\end{proof}


\begin{appendix}
\section*{Proof of Lemma \ref{Convergence}}
\label{pflemconv}
\begin{proof}[Step 1: Proof for the randomized $\star$-VRJP]
Let us first  prove the result when the initial random rates $W^A$ are random, $A\sim \nu_{i_0}^W$, as in section \ref{pervrjp}. We want to show that $T(t)-t/N$ converges 
to a random variable $\al\in \lll_0^{W^A}$ for $\nu_{\aaa,i_0}^W$ almost all $A$, hence for Lebesgue-almost all $A\in \aaa$ since $\nu_{\aaa,i_0}^W$ is absolutely continuous. We use the partial exchangeability of the process proved in 
Proposition \ref{exchangeability}.
We know from Proposition \ref{exchangeability} (using the notation $C$ therein)  that, under the randomized law $\overline \P_{i_0}^W$,  $Z_s=X_{C^{-1}(s)}$ is a mixture of Markov jump processes, hence
that 
$$
\lim_{s\to \infty} {1\over s} \ell_i^Z(s)=V_i
$$
exists a.s. ($V_i$ is random), where $\ell^Z(s)$ is the local time of $Z$ at time $s$, see \eqref{loc-time-Z}.

From \eqref{formule_chgt_tps}, with $s=C(t)$,
$$
\ell_i^Z(s)+\ell_{i^\star}^Z(s)= e^{T_i(t)+T_{i^\star}(t)}-1, \;\;\; \forall i\in V,
$$ 
hence, 
$$
\lim_{t \to\infty} \left( T_i(t)+T_{i^\star}(t)-\log(C(t))\right)= \log(V_i+V_{i^\star}).
$$
Since
\beq
2t= \sum_{i\in V} (T_i(t)+T_{i^\star}(t))
=
\sum_{i\in V} \log(1+\ell_i^Z(s)+\ell_{i^\star}^Z(s))
\eeq
we have,
$$
\lim_{t\to\infty}
2t- N \ln C(t) =
\sum_{i\in V} \log (V_i+V_{i^\star}),
$$
hence, we deduce
$$
\lim_{t\to\infty} T_i(t)+T_{i^\star}(t) -  2t/N = \log(V_i+V_{i^\star})- {1\over N}\left( \sum_{i\in V} \log (V_i+V_{i^\star})\right).
$$

Let us now conclude that $T(t)-t/N$ converges a.s. and that
its limit $U$ is in $\lll_0^{W^A}$.
Let 
$$
H(t)=\ppp_{\lll_0^{W^A}}(T(t)-t/N).
$$
Since $T_i(t)+T_{i^\star}(t)-2t/N$ convergences, it implies by lemma \ref{transversal} that $H(t)$ converges a.s. to a point $U$ in $\lll^{W^A}_0$.
It remains to show that $\xi_i(t)=T_i(t)-H_i(t)$ converges a.s. to 0.

Let $N_{i,j}(t)$ be the number of crossings of the directed edges $(i,j)$ or $(j^\star,i^\star)$ at time $t>0$.
Since the discrete time process associated with the randomized $\star$-VRJP $X_t$ is partially exchangeable, 
$${N_{i,j}(t)\over \sum_{(k,l)\in \tilde E}N_{k,l}(t)}$$ converges a.s. to a random variable $x_{i,j}$.

For each edge $e=(i,j)\sim(j^\star,i^\star)$ of $\tilde E$, let $\tau_k^{e}$ be a point process at rate $W_{i,j} e^t dt$.
Then the process $X_t$ can be represented as follows : when $X_t$ is at position $i$, it waits for the first time that, for an adjacent edge $(i,j)$,
$T_i(t)+T_{j^\star}(t)$ coincides with $\tau^{i,j}_k$ for some $k$, at which time it jumps to $j$.
This implies that, for all $(i,j)\in E$, 
$$
\tau_{N_{i,j}}(t)\le T_i(t)+T_{j^\star}(t) < \tau_{N_{i,j}(t)+1}.
$$
Now, direct computation yields that $e^{\tau_k^e}-1$ is a Poisson Point Process at rate $W_{i,j}$ for all $(i,j)\in \tilde E$, hence 
$W_e e^{\tau_k^e} \sim k $ a.s.. This implies that 
$$
W_{i,j}e^{T_i(t)+T_{j^\star}(t)} \sim N_{i,j}(t) \tx{ a.s.}
$$
Hence the convergence of the ratios ${N_{i,j}(t)\over \sum_{(k,l)\in \tilde E}N_{k,l}(t)}$ implies the a.s. convergence
of $T_i(t)+T_{j^\star}(t)-2t/ N$. 
Taking into account that 
$$
\dive (N(t))= \delta_{i_0} -\delta_{X_t},
$$
it implies that 
$$
\lim_{t\to 0} \dive(W^A_{i,j}e^{T_{i}+T_{j^\star}-2t/N})=0,
$$ 
hence, that $\lim_{t\to \infty} \xi(t)=0$.
\ali\ali
{\it Step 2: General case.}
Now we consider the $\star$-VRJP with initial rates $W_{i,j}$. Let $i_0\in V$, we define $\tau$ as the first return time to $i_0$ after the cover time of the graph, 
$$
\tau=\inf\{t>\tilde \tau, \; X_t=i_0\hbox{ and $\exists$ $t'$, $\tilde \tau<t'< t$ such that $X_{t'}\neq i_0$}\},$$
where 
$$\tilde \tau=\inf\{ t>0, \; T_i(t)>0\; \forall i\in V\}.
$$
From \eqref{density-A} (with $A=0$) we see that, under the law of the non-randomized VRJP, summing on all possible discrete paths,  the distribution of the local time $(T_i(\tau))$ at time $\tau$ has a density on $(0,\iy)^V$, that we denote $g^W_{i_0}((t_i)_{i\in V})$, and that this density is continuous in the initial conductances $W$.
Besides, conditioned on $\fff_\tau^X$, $(X_{t+\tau})_{t\ge 0}$ has the law of the $\star$-VRJP starting from the initial rates $W^{T(\tau)}_{i_0}$.
Hence, conditioned on $H:=\ppp_\sss(T(\tau))$, $\ppp_\aaa(T(\tau))$ is absolutely continuous with respect to 
$\nu_{X_\tau}^{W^H}$, hence $T(t+\tau)-T(\tau)-t/N$ converges a.s. to a point in $\lll_0^{W^{T(\tau)}}$.  This implies, that a.s. $T(t)-t/N$
converges a.s. to a point $U\in \lll_0^W$.
\end{proof}

\end{appendix}

\begin{appendix}
\section*{Results on $M$-matrices}\label{Append_M_matrices}
We will need the following results on $M$-matrices, taken from ~\cite{berman1994nonnegative}, chapter 6. The following definition is equivalent to the more classical definition (1.2) of \cite{berman1994nonnegative}, using
theorem 2.3, property $(G_{20})$ in  \cite{berman1994nonnegative}.
\begin{definition}\label{def-M}
A real $n\times n$ matrix $A$ is called a non-singular M-matrix 
if it has nonpositive off-diagonal coefficients, i.e. 
$$a_{i,j}\le 0, \;\;\; \forall i\ne j,$$
and if the real parts of all of its eigenvalues are positive.

\end{definition}
\begin{proposition}[theorem 2.3, chapter 6, \cite{berman1994nonnegative}]
\label{prop-M}
Assume that $A$ is a real $n\times n$ matrix with nonpositive off-diagonal coefficients. The assertion " $A$ is a non-singular M-matrix " is equivalent to each of the following assertions
\begin{enumerate}
\item[(1)] \label{inv-pos}
(Property $(N_{38})$ in \cite{berman1994nonnegative})
$A$ is invertible and $A^{-1}$ has nonnegative coefficients. If moreover $A$ is irreducible, this implies that $A^{-1}$ has positive coefficients by theorem 2.7 \cite{berman1994nonnegative}.
\item[(2)] (Property $(L_{32})$ in \cite{berman1994nonnegative}) 
\label{cond-M}
There exists a vector $x$ with positive coefficients
such that $y:=A x$ has nonnegative coefficients and such that if $y_{j_0}=0$ for some $j_0$, then there exists a sequence of indices 
$j_1, \ldots j_k$ with $y_{j_k}>0$ such that
$a_{j_l,j_{l+1}}\neq 0$ for all $l=0, \ldots k-1$.
\end{enumerate}
\end{proposition}
\end{appendix}

%
%


\begin{funding}
This work is supported by National Science Foundation of China (NSFC) 
grant No.\ 11771293, by the LABEX MILYON
(ANR-10-LABX-0070) of Universit\'e de Lyon and the  project  ANR LOCAL (ANR-22-CE40-0012-02)  operated by the Agence Nationale de la Recherche (ANR) 
\end{funding}



\bibliographystyle{imsart-number} 
\bibliography{Star-VRJP.bib}       

\begin{thebibliography}{34}

\bibitem{ACK14}
\begin{barticle}[author]
\bauthor{\bsnm{Angel},~\bfnm{Omer}\binits{O.}},
  \bauthor{\bsnm{Crawford},~\bfnm{Nicholas}\binits{N.}} \AND
  \bauthor{\bsnm{Kozma},~\bfnm{Gady}\binits{G.}}
(\byear{2014}).
\btitle{Localization for linearly edge reinforced random walks}.
\bjournal{Duke Mathematical Journal}
\bvolume{163}
\bpages{889--921}.
\end{barticle}
\endbibitem

\bibitem{Bacallado2}
\begin{barticle}[author]
\bauthor{\bsnm{Bacallado},~\bfnm{Sergio}\binits{S.}}
(\byear{2011}).
\btitle{Bayesian analysis of variable-order, reversible {M}arkov chains}.
\bjournal{Ann. Statist.}
\bvolume{39}
\bpages{838--864}.
\bdoi{10.1214/10-AOS857}
\bmrnumber{2816340}
\end{barticle}
\endbibitem

\bibitem{Bacallado1}
\begin{barticle}[author]
\bauthor{\bsnm{Bacallado},~\bfnm{S.}\binits{S.}},
  \bauthor{\bsnm{Chodera},~\bfnm{J.~D.}\binits{J.~D.}} \AND
  \bauthor{\bsnm{Pande},~\bfnm{V.}\binits{V.}}
(\byear{2009}).
\btitle{Bayesian comparison of {M}arkov models of molecular dynamics with
  detailed balance constraint}.
\bjournal{J. Chem. Phys.}
\bvolume{131}
\bpages{045106}.
\end{barticle}
\endbibitem

\bibitem{BST20}
\begin{barticle}[author]
\bauthor{\bsnm{Bacallado},~\bfnm{Sergio}\binits{S.}},
  \bauthor{\bsnm{Sabot},~\bfnm{Christophe}\binits{C.}} \AND
  \bauthor{\bsnm{Tarr\`es},~\bfnm{Pierre}\binits{P.}}
(\byear{2020}).
\btitle{The $\star$-{E}dge {R}einforced {R}andom {W}alk}.
\bjournal{preprint, arXiv:2102.08984}.
\end{barticle}
\endbibitem

\bibitem{bs}
\begin{barticle}[author]
\bauthor{\bsnm{Basdevant},~\bfnm{Anne-Laure}\binits{A.-L.}} \AND
  \bauthor{\bsnm{Singh},~\bfnm{Arvind}\binits{A.}}
(\byear{2012}).
\btitle{Continuous-time vertex reinforced jump processes on {G}alton-{W}atson
  trees}.
\bjournal{Ann. Appl. Probab.}
\bvolume{22}
\bpages{1728--1743}.
\end{barticle}
\endbibitem

\bibitem{BHS20}
\begin{barticle}[author]
\bauthor{\bsnm{Bauerschmidt},~\bfnm{Roland}\binits{R.}},
  \bauthor{\bsnm{Helmuth},~\bfnm{Tyler}\binits{T.}} \AND
  \bauthor{\bsnm{Swan},~\bfnm{Andrew}\binits{A.}}
(\byear{2021}).
\btitle{The geometry of random walk isomorphism theorems}.
\bjournal{Ann. Inst. Henri Poincar\'e{} Probab. Stat.}
\bvolume{57}
\bpages{408--454}.
\bdoi{10.1214/20-aihp1083}
\bmrnumber{4255180}
\end{barticle}
\endbibitem

\bibitem{berman1994nonnegative}
\begin{bbook}[author]
\bauthor{\bsnm{Berman},~\bfnm{A}\binits{A.}} \AND
  \bauthor{\bsnm{Plemmons},~\bfnm{RJ}\binits{R.}}
(\byear{1979}).
\btitle{{Nonnegative matrices in the mathematical sciences}}.
\bpublisher{Academic Press, New York}.
\end{bbook}
\endbibitem

\bibitem{coppersmith}
\begin{barticle}[author]
\bauthor{\bsnm{Coppersmith},~\bfnm{D.}\binits{D.}} \AND
  \bauthor{\bsnm{Diaconis},~\bfnm{P.}\binits{P.}}
(\byear{1986}).
\btitle{Random walks with reinforcement}.
\bjournal{Unpublished manuscript}.
\end{barticle}
\endbibitem

\bibitem{Cox}
\begin{bbook}[author]
\bauthor{\bsnm{Cox},~\bfnm{David~A.}\binits{D.~A.}}
(\byear{2012}).
\btitle{Galois theory},
\bedition{second} ed.
\bseries{Pure and Applied Mathematics (Hoboken)}.
\bpublisher{John Wiley \& Sons, Inc., Hoboken, NJ}.
\end{bbook}
\endbibitem

\bibitem{dv1}
\begin{barticle}[author]
\bauthor{\bsnm{Davis},~\bfnm{Burgess}\binits{B.}} \AND
  \bauthor{\bsnm{Volkov},~\bfnm{Stanislav}\binits{S.}}
(\byear{2002}).
\btitle{Continuous time vertex-reinforced jump processes}.
\bjournal{Probab. Theory Related Fields}
\bvolume{123}
\bpages{281--300}.
\bdoi{10.1007/s004400100189}
\bmrnumber{1900324 (2003e:60078)}
\end{barticle}
\endbibitem

\bibitem{dv2}
\begin{barticle}[author]
\bauthor{\bsnm{Davis},~\bfnm{Burgess}\binits{B.}} \AND
  \bauthor{\bsnm{Volkov},~\bfnm{Stanislav}\binits{S.}}
(\byear{2004}).
\btitle{Vertex-reinforced jump processes on trees and finite graphs}.
\bjournal{Probab. Theory Related Fields}
\bvolume{128}
\bpages{42--62}.
\bdoi{10.1007/s00440-003-0286-y}
\bmrnumber{2027294 (2004m:60179)}
\end{barticle}
\endbibitem

\bibitem{diaconis-freedman}
\begin{barticle}[author]
\bauthor{\bsnm{Diaconis},~\bfnm{P.}\binits{P.}} \AND
  \bauthor{\bsnm{Freedman},~\bfnm{D.}\binits{D.}}
(\byear{1980}).
\btitle{de {F}inetti's theorem for {M}arkov chains}.
\bjournal{Ann. Probab.}
\bvolume{8}
\bpages{115--130}.
\bmrnumber{556418 (81f:60090)}
\end{barticle}
\endbibitem

\bibitem{DST15}
\begin{barticle}[author]
\bauthor{\bsnm{Disertori},~\bfnm{Margherita}\binits{M.}},
  \bauthor{\bsnm{Sabot},~\bfnm{Christophe}\binits{C.}} \AND
  \bauthor{\bsnm{Tarres},~\bfnm{Pierre}\binits{P.}}
(\byear{2015}).
\btitle{{Transience of edge-reinforced random walk}}.
\bjournal{Communications in Mathematical Physics}
\bvolume{339}
\bpages{121-148}.
\end{barticle}
\endbibitem

\bibitem{DS10}
\begin{barticle}[author]
\bauthor{\bsnm{Disertori},~\bfnm{M.}\binits{M.}} \AND
  \bauthor{\bsnm{Spencer},~\bfnm{T.}\binits{T.}}
(\byear{2010}).
\btitle{Anderson localization for a supersymmetric sigma model}.
\bjournal{Comm. Math. Phys.}
\bvolume{300}
\bpages{659--671}.
\bdoi{10.1007/s00220-010-1124-6}
\bmrnumber{2736958}
\end{barticle}
\endbibitem

\bibitem{DSZ10}
\begin{barticle}[author]
\bauthor{\bsnm{Disertori},~\bfnm{M.}\binits{M.}},
  \bauthor{\bsnm{Spencer},~\bfnm{T.}\binits{T.}} \AND
  \bauthor{\bsnm{Zirnbauer},~\bfnm{M.~R.}\binits{M.~R.}}
(\byear{2010}).
\btitle{Quasi-diffusion in a 3{D} supersymmetric hyperbolic sigma model}.
\bjournal{Comm. Math. Phys.}
\bvolume{300}
\bpages{435--486}.
\bdoi{10.1007/s00220-010-1117-5}
\bmrnumber{2728731 (2011k:82035)}
\end{barticle}
\endbibitem

\bibitem{Dobrushin}
\begin{barticle}[author]
\bauthor{\bsnm{Dobrushin},~\bfnm{R.~L.}\binits{R.~L.}},
  \bauthor{\bsnm{Sukhov},~\bfnm{Yu.~M.}\binits{Y.~M.}} \AND
  \bauthor{\bsnm{Fritts},~\bfnm{\u{I}.}\binits{u.}}
(\byear{1988}).
\btitle{A. {N}. {K}olmogorov---founder of the theory of reversible {M}arkov
  processes}.
\bjournal{Uspekhi Mat. Nauk}
\bvolume{43}
\bpages{167--188}.
\end{barticle}
\endbibitem

\bibitem{Enriquez-Sabot06}
\begin{barticle}[author]
\bauthor{\bsnm{Enriquez},~\bfnm{Nathana\"{e}l}\binits{N.}} \AND
  \bauthor{\bsnm{Sabot},~\bfnm{Christophe}\binits{C.}}
(\byear{2006}).
\btitle{Random walks in a {D}irichlet environment}.
\bjournal{Electron. J. Probab.}
\bvolume{11}
\bpages{no. 31, 802--817}.
\bdoi{10.1214/EJP.v11-350}
\bmrnumber{2242664}
\end{barticle}
\endbibitem

\bibitem{Freedman}
\begin{barticle}[author]
\bauthor{\bsnm{Freedman},~\bfnm{David~A.}\binits{D.~A.}}
(\byear{1965}).
\btitle{Bernard {F}riedman's urn}.
\bjournal{Ann. Math. Statist}
\bvolume{36}
\bpages{956--970}.
\bmrnumber{0177432 (31 \#\#1695)}
\end{barticle}
\endbibitem

\bibitem{Galuzzi}
\begin{barticle}[author]
\bauthor{\bsnm{Galuzzi},~\bfnm{Massimo}\binits{M.}}
\btitle{Equations et substitutions avant {G}alois : {L}agrange et {C}auchy}.
\bjournal{preprint}.
\end{barticle}
\endbibitem

\bibitem{keane-rolles1}
\begin{barticle}[author]
\bauthor{\bsnm{Keane},~\bfnm{M.~S.}\binits{M.~S.}} \AND
  \bauthor{\bsnm{Rolles},~\bfnm{S.~W.~W.}\binits{S.~W.~W.}}
(\byear{2000}).
\btitle{Edge-reinforced random walk on finite graphs}.
\bjournal{Infinite dimensional stochastic analysis (Amsterdam, 1999) R. Neth.
  Acad. Arts. Sci}
\bpages{217-234}.
\end{barticle}
\endbibitem

\bibitem{Kozma_Peled21}
\begin{barticle}[author]
\bauthor{\bsnm{Kozma},~\bfnm{Gady}\binits{G.}} \AND
  \bauthor{\bsnm{Peled},~\bfnm{Ron}\binits{R.}}
(\byear{2021}).
\btitle{Power-law decay of weights and recurrence of the two-dimensional
  {VRJP}}.
\bjournal{Electron. J. Probab.}
\bvolume{26}
\bpages{Paper No. 82, 19}.
\bdoi{10.1214/21-ejp639}
\bmrnumber{4278593}
\end{barticle}
\endbibitem

\bibitem{Lagrange}
\begin{barticle}[author]
\bauthor{\bsnm{Lagrange},~\bfnm{J.~L.}\binits{J.~L.}}
(\byear{1770}).
\btitle{R\'eflexions sur la r\'esolution alg\'ebrique des \'equations}.
\bjournal{M\'emoires de l'Acad\'emie royale des sciences et Belles-Lettres de
  Berlin, pages 205--421, Publi\'e dans 1772. Oeuvres, 3, 205--421.}
\end{barticle}
\endbibitem

\bibitem{LW17}
\begin{barticle}[author]
\bauthor{\bsnm{Letac},~\bfnm{G\'{e}rard}\binits{G.}} \AND
  \bauthor{\bsnm{Weso\l{}owski},~\bfnm{Jacek}\binits{J.}}
(\byear{2020}).
\btitle{Multivariate reciprocal inverse {G}aussian distributions from the
  {S}abot-{T}arr\`es-{Z}eng integral}.
\bjournal{J. Multivariate Anal.}
\bvolume{175}
\bpages{104559, 18}.
\end{barticle}
\endbibitem

\bibitem{Poudevigne19}
\begin{barticle}[author]
\bauthor{\bsnm{Poudevigne-Auboiron},~\bfnm{R\'emy}\binits{R.}}
(\byear{2022}).
\btitle{Monotonicity and phase transition for the {VRJP} and the {ERRW}}.
\bjournal{J. Eur. Math. Soc.}
\end{barticle}
\endbibitem

\bibitem{sabot1}
\begin{barticle}[author]
\bauthor{\bsnm{Sabot},~\bfnm{C.}\binits{C.}}
(\byear{2011, arXiv:0811.4285}).
\btitle{Random walks in random {D}irichlet environment are transient in
  dimension {$d\geq 3$}}.
\bjournal{Probab. Theory Related Fields}
\bvolume{151}
\bpages{297--317}.
\bdoi{10.1007/s00440-010-0300-0}
\bmrnumber{2834720}
\end{barticle}
\endbibitem

\bibitem{Sabot21}
\begin{barticle}[author]
\bauthor{\bsnm{Sabot},~\bfnm{Christophe}\binits{C.}}
(\byear{2021}).
\btitle{Polynomial localization of the 2{D}-vertex reinforced jump process}.
\bjournal{Electron. Commun. Probab.}
\bvolume{26}
\bpages{Paper No. 1, 9}.
\bdoi{10.1214/20-ecp356}
\bmrnumber{4218029}
\end{barticle}
\endbibitem

\bibitem{ST15}
\begin{barticle}[author]
\bauthor{\bsnm{Sabot},~\bfnm{Christophe}\binits{C.}} \AND
  \bauthor{\bsnm{Tarr\`es},~\bfnm{Pierre}\binits{P.}}
(\byear{2015}).
\btitle{Edge-reinforced random walk, vertex-reinforced jump process and the
  supersymmetric hyperbolic sigma model}.
\bjournal{J. Eur. Math. Soc. (JEMS)}
\bvolume{17}
\bpages{2353--2378}.
\bdoi{10.4171/JEMS/559}
\bmrnumber{3420510}
\end{barticle}
\endbibitem

\bibitem{STZ17}
\begin{barticle}[author]
\bauthor{\bsnm{Sabot},~\bfnm{C.}\binits{C.}},
  \bauthor{\bsnm{Tarr\`es},~\bfnm{P.}\binits{P.}} \AND
  \bauthor{\bsnm{Zeng},~\bfnm{X.}\binits{X.}}
(\byear{2017}).
\btitle{The vertex reinforced jump process and a random {S}chr\"odinger
  operator on finite graphs}.
\bjournal{Ann. Probab.}
\bvolume{45}
\bpages{3967--3986}.
\bdoi{10.1214/16-AOP1155}
\bmrnumber{3729620}
\end{barticle}
\endbibitem

\bibitem{sabot-tournier17}
\begin{barticle}[author]
\bauthor{\bsnm{Sabot},~\bfnm{Christophe}\binits{C.}} \AND
  \bauthor{\bsnm{Tournier},~\bfnm{Laurent}\binits{L.}}
(\byear{2017}).
\btitle{Random walks in {D}irichlet environment: an overview}.
\bjournal{Ann. Fac. Sci. Toulouse Math. (6)}
\bvolume{26}
\bpages{463--509}.
\bdoi{10.5802/afst.1542}
\bmrnumber{3640900}
\end{barticle}
\endbibitem

\bibitem{SZ19}
\begin{barticle}[author]
\bauthor{\bsnm{Sabot},~\bfnm{Christophe}\binits{C.}} \AND
  \bauthor{\bsnm{Zeng},~\bfnm{Xiaolin}\binits{X.}}
(\byear{2019}).
\btitle{A random {S}chr\"{o}dinger operator associated with the vertex
  reinforced jump process on infinite graphs}.
\bjournal{J. Amer. Math. Soc.}
\bvolume{32}
\bpages{311--349}.
\bdoi{10.1090/jams/906}
\bmrnumber{3904155}
\end{barticle}
\endbibitem

\bibitem{Yaglom}
\begin{barticle}[author]
\bauthor{\bsnm{Yaglom},~\bfnm{A.~M.}\binits{A.~M.}}
(\byear{1947}).
\btitle{On the statistical treatment of {B}rownian motion}.
\bjournal{Doklady Akad. Nauk SSSR (N.S.)}
\bvolume{56}
\bpages{691--694}.
\end{barticle}
\endbibitem

\bibitem{Yaglom0}
\begin{barticle}[author]
\bauthor{\bsnm{Yaglom},~\bfnm{A.~M.}\binits{A.~M.}}
(\byear{1949}).
\btitle{On the statistical reversibility of {B}rownian motion}.
\bjournal{Mat. Sbornik N.S.}
\bvolume{24(66)}
\bpages{457--492}.
\end{barticle}
\endbibitem

\bibitem{Zeng16}
\begin{barticle}[author]
\bauthor{\bsnm{Zeng},~\bfnm{Xiaolin}\binits{X.}}
(\byear{2016}).
\btitle{How vertex reinforced jump process arises naturally}.
\bjournal{Ann. Inst. Henri Poincar\'{e} Probab. Stat.}
\bvolume{52}
\bpages{1061--1075}.
\bdoi{10.1214/15-AIHP671}
\bmrnumber{3531700}
\end{barticle}
\endbibitem

\bibitem{zirnbauer91}
\begin{barticle}[author]
\bauthor{\bsnm{Zirnbauer},~\bfnm{Martin~R}\binits{M.~R.}}
(\byear{1991}).
\btitle{{Fourier analysis on a hyperbolic supermanifold with constant
  curvature}}.
\bjournal{Communications in mathematical physics}
\bvolume{141}
\bpages{503--522}.
\end{barticle}
\endbibitem

\end{thebibliography}


\end{document}